\documentclass[10pt, letter]{amsart}
\usepackage{amsfonts, amsthm, amsgen, amsmath, amssymb, amscd,verbatim, enumerate, textcomp}

\let\LaTeXStandardTableOfContents\tableofcontents

\renewcommand{\tableofcontents}{%
\begingroup%
\renewcommand{\bfseries}{\relax}%
\LaTeXStandardTableOfContents%
\endgroup%
}%

\usepackage[margin=1in]{geometry}

\newtheorem{theorem}{Theorem}[section]

\newtheorem{lemma}[theorem]{Lemma}
\newtheorem{corollary}[theorem]{Corollary}
\newtheorem{proposition}[theorem]{Proposition}
\theoremstyle{definition}
\newtheorem{definition}[theorem]{Definition}

\newtheorem{example}[theorem]{Example}
\newtheorem{examples}[theorem]{Examples}

\theoremstyle{remark}
\newtheorem{remark}[theorem]{Remark}
\newtheorem{remarks}[theorem]{Remarks}
\theoremstyle{convention}

\numberwithin{equation}{section}
\usepackage{amscd}
\usepackage{xypic}
\usepackage{amsmath, amssymb}

\numberwithin{equation}{subsection}
\newcommand{\be}
  {\protect\setcounter{equation}{\value{subsubsection}}}  
  \newcommand{\ee}%
   {\protect\setcounter{subsubsection}{\value{equation}}}

  {\protect\setcounter{subsubsection}{\value{equation}}}

\def \A{{\bold A}}
\def \bA {\bold A}
\def \rmA{\rm A}

\def \B{\bold B}
\def \bB{\bold B}
\def \rmB{\rm B}

\def \rmC{\rm C}

\def \bC{\bold C}

\def \colimK.{\underset {\underset K^.  \rightarrow}  {\hbox {lim}}}

\def \colimU.{\underset {\underset U_.  \rightarrow}  {\hbox {lim}}}

\def \compl{\, \, {\widehat {}}}
\def \rmCone{\rm Cone}

\def \rmD{\rm D}

\def \E{{\mathcal E}}

\def \EG1{E{(G \times {\mathbb C}^*)}{\underset {G\times {\mathbb C}^*}  \times}}

\def \EZ(s)1{E{(Z(s) \times {\mathbb C}^*)}{\underset {(Z(s)\times {\mathbb C}^*)}  \times}}

\def \EM(u){EM(u){\underset {M(u)}  \times}}
\def \EM(us){EM(u,s){\underset {M(u, s)}  \times}}

\def \rmE{\rm E}

\def \f{{\bf f}}
\def \rmf{\rm f}
\def \itf{\it f}
\def \rmF{\rm F}

\def \g{{\bf g}}

\def \rmG{\bold G}
\def \bw\rmG{w{\bold G}}
\def \rmG{\rm G}

\def \rmH{\rm H}

\def \rmI{\rm I}

\def \invlim1{\underset {\infty \leftarrow q}  {\hbox {lim}}^1}

\def \rmK{\rm K}

\def \L3{\Lambda \times \Lambda \times \Lambda}
\def \L2{\Lambda \times \Lambda}

\def \longright2arrow{{\overset \longrightarrow  {\overset {}  \longrightarrow}}}

\def \L{\mathcal L}

\def \rmL{\rm L}

\def \rmM{\rm M}

\def \rmN{\rm N}
\def \rmNerve{\rm Nerve}

\def \O{{\mathcal O}}
\def \rmQ{\rm Q}

\def \rmP{\rm P}

\def \rmpr_2{{\rm pr}_2}
\def \PwS{{\rm P}{\it w}{\rmS}}
\def \rmp{\rm p}

\def \Q{{\mathbb  Q}}

\def \ra{\rightarrow}
\def \Ra{\Rightarrow}
\def \RG^{R(G)^{\hat {}}\ }

\def \res{respectively}
\def \r{\it r}

\def \rmR{\rm R}

\def \S{\mathcal S}

\def \rmS{\rm S}

\def \s{\it s}

\def \topGcoh*{^{top, *} _{G}}
\def \topGho*{ _{top,*} ^{G}}

\def \rmT{\rm T}
\def \Tot{\rm Tot}

\def \rmU{\rm U}

\def \rmV{\rm V}

\def \wG{{\it w}{\rm G}}
\def\wS{{\it w}{\rm S}}

\def \rmX{\rm X}

\def \x{\it x}

\def \rmY{\rm Y}

\def \Z(s){Z(s) \times {\mathbb C}^*}
\def \Z{\mathbb Z}

\begin{document}

\title{$\lambda$-ring structures on the K-Theory of algebraic stacks} 

\author{Roy Joshua}
\address{Department of Mathematics, Ohio State University, Columbus, Ohio,
43210, USA.}
\email{joshua.1@math.osu.edu}
\author{Pablo Pelaez}
\address{Instituto de Matem\'aticas, Ciudad Universitaria, UNAM, DF 04510, M\'exico.}
\email{pablo.pelaez@im.unam.mx}
\thanks{The first author thanks the Simons Foundation for support.}

\maketitle

\setcounter{tocdepth}{1}

\begin{abstract}
In this paper  we consider the K-theory of  smooth algebraic stacks, establish $\lambda$ and $\gamma$ 
 operations, and show that the higher K-theory of such stacks is always  a  pre-$\lambda$-ring, and is a 
$\lambda$-ring if every coherent sheaf is the quotient of a vector bundle. As a consequence, we are 
able to define Adams operations and absolute cohomology for smooth algebraic stacks satisfying this hypothesis. 
We also obtain a comparison of the absolute cohomology with the equivariant higher Chow groups in certain special cases.

\end{abstract}

\tableofcontents
\input xy
\vfill \eject
\vskip .2cm

\section{\bf Introduction.}
In the case of smooth schemes of finite type over a field, the existence of $\lambda$-operations on algebraic K-theory enables one to define
absolute cohomology as the eigenspace for the Adams operations on rational algebraic K-theory. In this paper we investigate the corresponding situation for algebraic stacks, beginning with $\lambda$-operations.
\vskip .2cm 
To begin with, it ought to be pointed out that  it has been an open question whether there exist $\lambda$ and Adams operations on the
 higher K-theory of algebraic stacks. 
The first result in this paper is an affirmative answer to this question, at least for many smooth quotient stacks; in fact, we show that the higher K-groups of smooth algebraic stacks are pre-$\lambda$-rings, and that
for algebraic stacks where every coherent sheaf is the quotient of a vector bundle, they are in fact $\lambda$-rings. Though the last {\it resolution property} is closely related to being a quotient stack, our proof is stack-theoretic in that it does not require an explicit presentation of the stack as a quotient stack. In fact, finding a presentation for a stack with the above resolution property as a quotient stack may be quite involved depending on the situation. Moreover,
invoking  a minimal number of background results proved on the K-theory and G-theory of algebraic stacks as in \cite{J-6}, the proofs here are quite straightforward and only a little background on stacks is required.
\vskip .2cm
Before we proceed any further, it seems important to point out why it is essential to work with the K-theory of perfect complexes, using the 
 machinery of categories with cofibrations and weak equivalences in the sense of \cite{Wald}. The main difficulty is that the 
 Quillen K-theory of the category of vector bundles even on a general smooth scheme, let alone on an algebraic space or algebraic stack, does 
 not have good properties like Poincar\'e duality or Mayer-Vietoris property. Here Poincar\'e duality refers to the identification  between 
 the K-theory of the exact category of vector bundles with the K-theory of the exact category of coherent sheaves when the stack is regular or smooth.
 We begin with a few basic
definitions so as to be able to state this in a precise manner.
\vskip .2cm
\begin{definition} 
\label{def1.1}
(See \cite[p. 240]{Kr-2}, or \cite[p. 98]{Wei}.) \footnote{Please note that what we call a pre-$\lambda$-ring is called a $\lambda$-ring and
what we call a $\lambda$-ring is called a special $\lambda$-ring in \cite{Wei}.} A {\it pre-$\lambda$}-ring $\rmR$ is a commutative ring with unit and provided with
maps $\lambda^i:\rmR\ra \rmR$, $i \ge 0$, (which are in general not homomorphisms) so that 
(i) $\lambda ^0 (r) =1$, for all $r \in \rmR$, (ii) $\lambda ^1 =id$, and (iii) 
$\lambda^n (r+s) = \Sigma _{i=0}^n \lambda ^i(r) \cdot\lambda ^{n-i}(s)$ for all $r, s \in \rmR$.
A {\it pre-$\lambda$-ring without unit} is a commutative ring $\rmR$ without a unit element and provided
with maps $\lambda^i: \rmR \ra \rmR$, $i >0$ satisfying the conditions (ii) and (iii). 
(For (iii) to make sense,  we use the convention that $\lambda^0(r)\cdot\lambda^n(s) = \lambda^n(s)$ and   $\lambda^n(r) \cdot\lambda^0(s) =\lambda^n(r)$, for all $r \in \rmR$, $s \in \rmR$.)
Given a pre-$\lambda$ ring $\rmR$, a {\it pre-$\lambda$-algebra} over $\rmR$ is a commutative ring $\rmS$ (not
necessarily with a unit) provided with the structure of a module over $\rmR$, and so that $\rmR \oplus \rmS$
gets the structure of a pre-$\lambda$-ring with the following operations:
\vskip .2cm \noindent
$\bullet$ the sum on $\rmR\oplus \rmS$ is the obvious sum induced by the sum on $\rmR$ and $\rmS$, 
\vskip .2cm \noindent
$\bullet$ the product on $\rmR \oplus \rmS$ is defined by $(r, s) \cdotp (r', s') = (r\cdotp r', r\cdotp s'+r'\cdotp s +s\cdotp s')$, where
the products $r\cdotp s'$ and $r'\cdotp s$ are formed using the module structure of $\rmS$ over $\rmR$, and the product $s \cdotp s'$
is formed using the product on $\rmS$, and
\vskip .2cm \noindent
$\bullet$ one is given maps  $\lambda^i: \rmR \oplus \rmS \ra \rmR \oplus \rmS$, $i \ge 0$ so that the following hold:
\begin{enumerate}[\rm (i)]
\item for all $i \ge0$, $\lambda^i$ restricted to
$\rmR$ identifies with the given operation $\lambda^i$ on $\rmR$, 
\item for all $i>0$, $\lambda^i$ restricted to $\rmS$ maps to $\rmS$, and
\item $\lambda^n(r, s) = (\lambda ^n(r), \Sigma _{i=0}^{n-1}\lambda ^i(r) \cdot \lambda ^{n-i}(s))$ for all
$r \in \rmR$, $s \in \rmS$.
\end{enumerate}
The product $\lambda^i(r). \lambda^{n-i}(s)$ uses the $\rmR$-module structure of
$\rmS$.
\vskip .2cm
A pre-$\lambda$-ring $\rmR$ is a {\it $\lambda$-ring}, if $\lambda ^n(1) =0$ for $n > 1$, and for certain universal polynomials  $\rmP_{\it k, l}$ and $\rmP_{\it k}$ with integral coefficients 
defined as in \cite[p. 258]{AT} the following equations hold:
\be \begin{align}
\label{lambda.ids}
\lambda^{\it k}(r \cdotp s) &= \rmP_{\it k}(\lambda^1(\r), \cdots, \lambda^{\it k}(\r); \lambda^1(\s), \cdots, \lambda^{\it k}(\s)) \ and \\
\lambda^{\it k}(\lambda^l (r )) &= \rmP_{\it k,l}(\lambda^1(\r), \cdots, \lambda^{\it k\cdotp l}(\r)), \, {\it r, s} \in \rmR.\notag \end{align} \ee
\vskip .2cm \noindent
One defines a {\it $\lambda$-ring without unit} to be a pre-$\lambda$-ring without unit satisfying the relations in ~\eqref{lambda.ids}. If $\rmR$ is a $\lambda$-ring and $\rmS$ is a pre-$\lambda$-algebra over $\rmR$, we say
$\rmS$ is a {\it $\lambda$-algebra} over $\rmR$ if the relations above also hold for $\lambda ^{\it k} (\lambda ^l(r+s))$, and
for $\lambda^{\it k}((r+s)\cdotp (r'+s'))$ if $r, r' \in \rmR$ and $s, s' \in S$, that is, $\rmR \oplus \rmS$ is a $\lambda$-ring with the operations defined above. 
\end{definition}
 Given an algebraic stack $\S$, ${\rm K}(\S)$ will denote the space obtained by applying the constructions of Waldhausen
 (see \cite{Wald}) to the category of perfect complexes on the stack $\S$: see Definition ~\ref{K.G.def}. 
 For a closed algebraic substack $\S'$ of $\S$, 
 $\rmK_{\S'}(\S)$ will denote the {\it space} defining the higher algebraic K-theory of $\S$ with supports in $\S'$ as in Definition ~\ref{K.G.def}.
 Then we obtain the theorem stated below, which is one of the main results of this paper. 
\vskip .2cm
\begin{theorem}
\label{mainthm} (i) Let $\S$ denote a smooth algebraic stack of finite type over a regular Noetherian base scheme $S$.
Then $\pi_0({\rm K}(\S))$ is a pre-$\lambda$-ring. 
\vskip .1cm
(ii) For $\S'$ denoting a closed algebraic sub-stack, $\pi_n(\rmK_{\S'}(\S))$, for each fixed $n \ge 0$, is a  pre-$\lambda$-algebra over the pre-$\lambda$-ring $\pi_0(\rmK (\S))$.
\vskip .2cm 
 The above pre-$\lambda$-ring structure is compatible with pull-backs: that is,
if $f: \tilde \S \ra \S$ is a map of smooth algebraic stacks and $\tilde \S' = \tilde \S {\underset {\S} \times} \S'$, then 
the induced map $f^*:\pi_0(\rmK (\S)) \ra \pi_0(\rmK (\tilde \S))$ is a map of pre-$\lambda$-rings, and the induced map
$f^*: \pi_n(\rmK_{\S'}(\S)) \ra \pi_n(\rmK_{\tilde \S'}(\tilde \S))$, for each fixed $n \ge 0$, is a map of  pre-$\lambda$-algebras over $\pi_0(\rmK (\S))$. 
The $\lambda$-operations are homomorphisms on $\pi_n(\rmK_{\S'}(\S))$ for all $n>0$.
\vskip .2cm
(iii) In case every coherent sheaf on the smooth stack $\S$ is the quotient of a vector bundle, each $\pi_n(\rmK_{\S'}(\S))$, for $n \ge 0$, is a {\it $\lambda$-algebra} 
over $\pi_0(\rmK (S))$ in the above sense.
\end{theorem}
\begin{remark}
 The proof of Theorem ~\ref{mainthm} is split into two parts: the first part that discusses the proofs of the two statements (i) and (ii)
 appear at the end of section 5. All of section 6 is devoted to a proof of statement (iii).
 \end{remark}
\vskip .2cm
The following is a quick summary of the techniques adopted in this paper to prove the above theorem and Theorems ~\ref{thm.2}
and ~\ref{loc.abs.coh} discussed below.
First, we invoke the technique in \cite{BL}, whereby
higher K-groups can be reduced to certain relative Grothendieck groups. This needs to make use of the homotopy property
for the K-theory of algebraic stacks (see \cite[Theorem 5.17]{J-6}), which makes it necessary to restrict to smooth stacks.
But then we need to interpret the relative Grothendieck groups in terms of a relative form of the Gillet-Grayson
un-delooping adapted to the Waldhausen setting. Considerable effort (in fact, all of section 3) is needed to carry this out. These, together with some well-known arguments due to
Grayson (see \cite[section 7]{Gray}), suffice to put a pre-$\lambda$-ring structure on the higher K-theory of all {\it smooth} algebraic stacks of finite type over any regular Noetherian base scheme. 
\vskip .2cm
In order to verify that the higher
K-theory of algebraic stacks form $\lambda$-rings, we are forced to restrict to smooth stacks that have the {\it resolution property}, namely 
where every coherent sheaf is the quotient of a vector bundle. In fact, the Quillen K-theory of vector bundles and the K-theory of perfect complexes are known to be isomorphic {\it only when
 they have the resolution property: namely, the property that every coherent sheaf is the quotient of a vector bundle}. See \cite[8.6. Exercise]{T-T}
 for the example of a scheme $\rmX$, which is the union of two copies of the affine $n$-space ${\mathbb A}^n$, for $n\ge 2$, glued along ${\mathbb A}^n - \{0\}$. This scheme does not have
  the resolution property, as the resolution property would imply that the diagonal morphism is affine, and it is not in this case. 
  More details on the resolution property may be found in \cite{Tot}: see also \cite[Proposition 2.8]{J-6}.
  One may also see \cite[Expos\'e II]{SGA6} and \cite{BS} for related results.
Finally we adapt certain arguments of Gillet and Soul\'e (see \cite{GS}) to prove that the corresponding relative Grothendieck groups
are $\lambda$-rings.
\vskip .2cm
\label{comparison}
Here is a comparison of our results with other related results in the literature. There is a great deal of literature on the 
$\lambda$-ring structure and related operations on the higher K-theory of schemes, including schemes that are possibly singular: see \cite{Kr-1}, \cite{Kr-2}, \cite{Lev}, \cite{GS2} and \cite{Lec}. We will not discuss such results any further, except to point out that 
when restricted to the category of (suitably nice schemes) our results in the present paper reduce to these.
While it has been known for sometime, especially after \cite{Gray}, 
how to define $\lambda$-operations for the Quillen K-theory of exact categories or the Waldhausen analogue of it, it was not clear that the
required relations are satisfied. For example, in \cite{Gray}, $\lambda$-operations are defined for the K-theory of the exact category of vector bundles, but it was left open whether they satisfied the required properties to define a $\lambda$-ring in general. 
Even for quotient stacks, or equivalently for schemes provided with an action by a smooth group-scheme, it has not been known till very recently if
there exists a $\lambda$-ring structure on their  higher K-theory. In fact, this was posed as a conjecture in the literature.
In \cite[(2.5) Proposition]{K}, it was shown that the Higher K-theory of quotient stacks is a pre-$\lambda$-ring and satisfies the first
 relation in ~\eqref{lambda.ids}. It was conjectured there (see \cite[(2.7) Conjecture]{K} that the second relation in ~\eqref{lambda.ids} is also satisfied
 by the higher K-theory of quotient stacks and it remained open till the very recent preprint \cite{KZ}.
\vskip .1cm
Most of the recent progress in this area in the literature follows the relatively recent results of Grayson defining higher K-theory using binary 
complexes as in \cite{Gray12}. Making use of this approach, the authors of \cite{HKT} prove the existence of a $\lambda$-ring structure 
for the higher K-theory of schemes, including ones that are possibly singular, but still left open the corresponding question for equivariant algebraic K-theory or for the algebraic K-theory 
of (quotient) stacks. (There is also the work of
\cite{Riou} and \cite{Z}, which provide a $\lambda$-ring structure on the higher K-theory of certain classes of schemes, including ones that are possibly singular: while these also involve a reduction to Grothendieck groups,
this approach does not seem to extend to any larger category than schemes.) In the very recent preprint, \cite[Theorem 5.1]{KZ}, the authors provide a proof that the second relation in ~\eqref{lambda.ids} is also satisfied
 by higher equivariant K-theory, making essential use of binary complexes. 

\vskip .1cm
Therefore, the corresponding question for the Higher Algebraic K-theory of Algebraic stacks in general has not been even looked at in the literature so far. \footnote{$\lambda$-operations can be defined at the level of 
Grothendieck groups for quotient stacks quite easily. See \cite{EJK} where they define such operations on the Grothendieck groups of certain inertia stacks
associated to smooth quotient stacks that are Deligne-Mumford. In case the stack is a quotient stack of the form 
$[\rmX/\rmT]$ for a split torus $\rmT$, the inertia stack is a disjoint union of quotient stacks of the form $[\rmX^{\it t}/\rmT]$, $t \in \rmT$: see \cite{St}. 
In this case, Theorem ~\ref{mainthm} would extend the $\lambda$-operations of \cite{EJK} to the higher K-theory of these inertia stacks, even when
the stack $[\rmX/\rmT]$ is not Deligne-Mumford.} \footnote{Quotient stacks of the form $[\rmX/\rmG]$, for a
scheme $\rmX$ and an affine group scheme $\rmG$ are quite special, and there are as many algebraic stacks which are not such global quotient stacks.}
Here are some of the main features of our work.
\vskip .1cm
$\bullet$ Our first result in Theorem ~\ref{mainthm}, proving the existence of a pre-$\lambda$-ring structure on the higher Algebraic K-theory of 
all smooth Algebraic stacks satisfying certain mild finiteness conditions therefore is the first positive result for smooth algebraic stacks in general. 
\vskip .1cm
Secondly, the second statement in Theorem ~\ref{mainthm}
has the following features: 
\vskip .1cm
$\bullet$ We prove the existence of  a $\lambda$-ring structure on the higher K-theory of smooth algebraic stacks satisfying the
{\it resolution property}. While this property is closely related to the stack being a quotient stack, it is not always equivalent to being a quotient stack (see
\cite[Theorems 1.1 and 1.2]{Tot} as well as the discussion following \cite[Proposition 1.3]{Tot} for a precise comparison), and our proof does {\it not} require the stack to be a quotient stack. 
\vskip .1cm
$\bullet$ Even when
the resolution property holds and one knows the given stack is a quotient stack, finding an explicit presentation for the stack as a quotient stack for
the action of an affine group scheme on a scheme is rather involved, and our techniques do not require knowing explicitly any such  presentation.
\vskip .1cm
$\bullet$ Moreover, as shown in Theorem ~\ref{mainthm}(iii), our results on $\lambda$-algebra structures hold for 
{\it the relative case} also, that is, also for the higher K-theory with supports in a closed substack. (The results of 
\cite[Theorem 5.1]{KZ}, as stated are only for the absolute case.) This is important, 
as we are then able to obtain certain long-exact sequences in the 
associated absolute cohomology as in Theorem ~\ref{loc.abs.coh}.
\vskip .1cm
$\bullet$ Finally, our techniques do not make use of the definition of higher Algebraic K-theory using binary complexes, but instead use more traditional
(topological) methods, such as a relative form of the Gillet-Grayson $\rmG_{\bullet}$-construction. The price we pay for this may be that we have to restrict to smooth stacks, so that the higher K-theory of these stacks have the homotopy property, and we 
are able to invoke the methods of \cite{BL} to reduce higher K-theory to certain relative Grothendieck groups.

\vskip .2cm
Making use of Theorem ~\ref{mainthm}, we define {\it $\gamma$ and Adams operations} on the higher K-groups of algebraic stacks that satisfy the property that every coherent sheaf is the quotient of a vector bundle;
further, making use of these operations we are also able to define the {\it absolute cohomology with $\Q$-coefficients} for such algebraic stacks. 
These results may be summarized in the following theorems.

\begin{theorem} 
\label{thm.2}
Let $\S$ denote a smooth algebraic stack as in Theorem ~\ref{mainthm} having the resolution property, and let $\S'$ denote a closed algebraic sub-stack.
\vskip .2cm
(i) Then there are $\gamma$ and Adams operations on each $\pi_n(\rmK_{\S'}(\S))$ that satisfy the (usual) relations:  
\[\gamma ^1 =id, \gamma^{\it k}(K ) = \Sigma _{k'+k''=k}\gamma^{k'}(K') . \gamma^{k''}(K''),\]
if $K = K' +K''$ in the $\lambda$-ring $\pi_n(\rmK_{\S'}(\S))$, and if $\S'=\S$, then $\gamma^0(K) = [\O_{\S}]=$ the class of the structure sheaf $\O_{\S}$ for any $K \in w{\rm {Perf}}_{fl}(\S)$. 
Moreover, for certain universal polynomials $\rmQ_{\it k, l}$ and $\rmQ_{\it k}$ with integral coefficients (see \cite[p. 262]{AT}) the following relations hold:
\[\gamma^{\it k}(\gamma^l (\alpha )) = \rmQ_{\it k,l}(\gamma^1(\alpha), \cdots, \gamma^{\it k.l}(\alpha)) \mbox{ and}\] 
\[\gamma^{\it k}(\alpha . \beta) = \rmQ_{\it k}(\gamma^1(\alpha), \cdots, \gamma^{\it k}(\alpha); \gamma^1(\beta), \cdots, \gamma^{\it k}(\beta)),\]
\vskip .2cm \noindent
for $\alpha$, $\beta \in \pi_n(\rmK_{\S'}(\S))$. 
\vskip .1cm
 The Adams operations $\psi^{\it k}$ preserve the
additive and multiplicative structures on $\pi_*(\rmK_{\S'}(\S))$. The Adams operations and the $\gamma$-filtration are natural with respect to pull-back. 
The graded piece \\ $gr_n(\pi_*(\rmK_{\S'}(\S))\otimes \Q)$ is the eigenspace for the induced action of $\psi^{\it k}$ with eigenvalue $k^n$.
\vskip .2cm
(ii) In case there exists a coarse moduli space ${\mathfrak M}$ (${\mathfrak M}'$) for the stack $\S$(the sub-stack $\S'$, \res) as an algebraic space, the $\gamma$-filtration on $\pi_*(\rmK_{\S'}(\S)) \otimes \Q$ above
is compatible with the $\gamma$-filtration on $\pi_*(\rmK_{{\mathfrak M'}}({\mathfrak M})) \otimes \Q$. 
 \end{theorem}
\vskip .2cm
One important distinction from the corresponding situation for schemes is that the $\gamma$-filtration is almost
never {\it nilpotent} for algebraic stacks, as may be seen in remark ~\ref{non.nilp}. Though the hypothesis that {\it the resolution property} holds seems strong, it is clearly satisfied by
many quotient stacks thanks to the work of Thomason (see \cite{T-1} through \cite{T-3}) and the work of Totaro (see \cite{Tot}). 
 
\vskip .2cm
We will define {\it absolute cohomology} by $\rmH^i_{\S', abs}(\S, \Q(j)) =
gr^j(\pi_{2j-i}(\rmK_{\S'}(\S)) \otimes \Q)$. (See Definition ~\ref{def.abs.coh} for more details.)
\begin{theorem} (Localization theorem for absolute cohomology)
\label{loc.abs.coh}
Let $\S$ denote a smooth algebraic stack as in Theorem ~\ref{mainthm}, and let $\S'_0 \subseteq \S_1'$ denote two closed algebraic sub-stacks. We will further assume that every coherent sheaf on the stack $\S$ is the quotient of
a vector bundle. Then one obtains the long exact sequence of absolute cohomology groups:
\[\cdots \ra \rmH^n_{\S_0', abs}(\S, \Q(i)) \ra \rmH^n_{\S_1', abs}(\S, \Q(i)) \ra \rmH^n_{\S_1'-\S_0', abs}(\S - \S_0', \Q(i)) \ra \rmH^{n+1}_{\S_0', abs}(\S, \Q(i)) \ra \cdots.\]
\end{theorem}

\vskip .2cm
Here is the layout of the paper. Section 2 is a quick review of the
basic properties of the K-theory and G-theory of algebraic stacks proved in \cite{J-5} and \cite{J-6}. We make a special effort here
in order to make
the paper accessible to readers who are primarily interested in the case of quotient stacks.
\vskip .2cm
Section 3
introduces a {\it key technique}: we obtain an explicit description of 
relative K-theory in terms of a {\it relative version} of the Gillet-Grayson G-construction (see \cite{GG} and \cite{Gray}), adapted to the setting of
categories with cofibrations and weak equivalences  by the methods of \cite{GSVW} and \cite{GSch}. (Though there is another description of relative K-theory due to Waldhausen (see \cite[Definition 1.5.4]{Wald}), that description 
 does not use the G-construction and therefore we cannot use it in our context.) 
\vskip .1cm
A non-trivial issue that shows up in this section is the difficulty of finding a categorical model for path spaces, that is,
a categorical construction whose nerve gives the usual path space. This is possible with the Waldhausen $\rmS_{\bullet}$-construction, as
is shown in \cite[section 2]{GSVW}, but does not extend to the $\rmG_{\bullet}$-construction. We circumvent this issue by
defining the path space only after topological realization. On the other hand, the mapping cone construction (which is
in a sense dual to the path space construction) readily extends to functors between simplicial categories. We invoke this
construction in section 4 to establish an additivity theorem for relative K-theory defined using a relative form of the
$\rmG_{\bullet}$-construction.
 \vskip .1cm
Section 5 is a key section, where we start by defining $\lambda$-operations on the Waldhausen K-theory of perfect complexes
in the framework of the $\rmG_{\bullet}$-construction. Since $\lambda$-operations are non-additive, we make use of
simplicial methods to do this, as in \cite{DP}, and \cite{GS}. Making use of techniques developed in the earlier sections, this
is then extended to relative K-theory defined using the relative $\rmG_{\bullet}$-construction. 
 The above techniques, along with the technique of reducing higher K-theory to relative Grothendieck groups (as in \cite{BL})  enable us to prove the first part of the main theorem. This puts a pre-$\lambda$-ring structure on
 the higher K-theory of all smooth algebraic stacks that are of finite type over any regular Noetherian base scheme, thereby completing the proofs
 of the first two statements in Theorem ~\ref{mainthm}.
\vskip .1cm
 In order to verify that the higher
K-groups of algebraic stacks form $\lambda$-rings, we are forced to restrict to smooth algebraic stacks that have the {\it resolution property}, namely 
where every coherent sheaf is the quotient of a vector bundle. Finally we adapt certain arguments of Gillet and Soul\'e (see \cite{GS}) to prove that the corresponding relative Grothendieck groups
are $\lambda$-rings.
 These occupy all of section 6 and complete the proof of the last statement of Theorem ~\ref{mainthm}. 
 \vskip .1cm
 In section 7 we define and study $\gamma$-operations and absolute cohomology for algebraic stacks. This section also contains
 the proofs of Theorems ~\ref{thm.2} and ~\ref{loc.abs.coh}. We conclude with several explicit examples in section 8. This section already has a brief comparison of absolute cohomology
  with the equivariant higher Chow groups in a few special cases. As pointed out here, the relationship with the equivariant higher Chow groups
 needs the machinery of derived completion (as in \cite{CJ23}) in general. Therefore, we have decided to explore this
 in a sequel, where we also plan to discuss Riemann-Roch theorems. A couple of short appendices are added to make
the paper self-contained. Appendix A summarizes the main results of Waldhausen K-theory. Appendix B summarized some well-known relations between
simplicial objects, cosimplicial objects and chain complexes in abelian categories.
\vskip .2cm
{\bf Quick summary of the notational terminology.} The basic terminology on algebraic stacks as well as algebraic K-theory
is discussed in the beginning of section 2: therefore, we do not repeat them here. The Gillet-Grayson G-construction will be denoted 
${\rm G}_{\bullet}$, while the Waldhausen S-construction will be denoted ${\rm S}_{\bullet}$:  these are discussed in section 3. ${\rm Nerve}$  denotes
the functor sending a small category to the simplicial set which is its nerve: this appears in sections 3 and 4. Given an exact category ${\mathbf A}$, ${\rm Cos.mixt}({\mathbf A})$
will denote the category of cosimplicial-simplicial objects in ${\mathbf A}$: this set-up is used in the definition of the derived functors 
of the exterior power operation in section 5. Given a simplicial set $\rmX_{\bullet}$, ${\rm sub}_k\rmX_{\bullet}$ produces a
multi-simplicial set of order $k$: this is discussed in \cite[section 4]{Gray} and recalled in section 5 along with certain other functors such as $\Xi$. 

\vskip .2cm 
{\bf Acknowledgments}. We would like to thank  Zig Fiedorowicz, 
 (the late) Rainer Vogt and an anonymous reviewer for various discussions related to a much earlier draft of 
the paper. We also owe an enormous thanks to Dan Grayson for undertaking a critical reading of the paper, suggesting several improvements in exposition, and also for several e-mail 
 exchanges that clarified the Gillet-Grayson ${\rmG}_{\bullet}$-construction to us. We also thank Bernhard Koeck for 
 undertaking a critical reading and providing several helpful comments which we have incorporated into the current version. Finally the 
 authors are very grateful to the referee for undertaking a careful and detailed reading of the manuscript and for making a number of very helpful
 comments and suggestions which surely have improved the paper.
\section{\bf K-theory and G-theory of Quotient stacks and Algebraic stacks: a quick review}
This section is a quick summary of the basic results on the K-theory and G-theory of algebraic stacks proved in \cite{J-5} and \cite{J-6}. Assuming these results, there is very little new stack-theoretic material needed in the later sections so that several of the basic results of this paper are no harder to state and prove
for general Artin stacks than for the special case of quotient stacks. 
\vskip .1cm
We will fix a regular Noetherian base scheme $\rmS$ throughout the paper and will consider only objects defined and finitely presented 
over $\rmS$. 
\begin{definition} (i) An {\it algebraic stack} $\S$ will mean an algebraic stack (of Artin type) 
which is finitely presented over a regular Noetherian base scheme $\rmS$. 
 An {\it action }
of a group scheme $\rmG$ on a stack $\S$ will mean morphisms $\mu, \rmpr_2: \rmG \times \S \ra \S$ and 
$e:\S \ra \rmG \times \S$ satisfying the usual relations. 
\vskip .1cm
(ii) A quotient stack $[\rmX/\rmG]$ will denote the Artin stack associated to the action of a smooth affine group-scheme
$\rmG$ on an algebraic space $\rmX$, both defined over $\rmS$. \footnote{In fact, the reader may observe as in the discussion in Definition ~\ref{eq.case.0} that working with quotient stacks corresponds 
 to working in the equivariant framework, and does not require any special knowledge of stack-theoretic machinery.}
\end{definition}
It is shown in \cite[Appendix]{J-4} that if $\rmG$ is 
a smooth group scheme acting on an algebraic stack $\S$, a quotient stack $[\S /G]$ exists
as an algebraic stack. In this case, there is
 an equivalence between the category of  
 $\rmG$-equivariant ${\mathcal O}_{\S}$-modules on $\S$ and the category of ${\mathcal O}_{[\S /G]}$-modules.
(See \cite[Appendix]{J-4}.)  Therefore, one may incorporate
 the equivariant situation into the following discussion by considering quotient stacks of the
 form $[\S /G]$. 
\vskip .2cm 
We have chosen to work mostly with the lisse-\'etale site (see  \cite[Chapter 12]{L-MB}, \cite{Ol}), though it seems
possible to work instead with the smooth site.
 Observe that if $\S$ is an algebraic stack, the underlying
category of $\S_{lis-et}$ is the same as the underlying category of the smooth site
$\S_{smt}$, whose objects are smooth maps $u:\rmU \ra \S$, with $\rmU$ an algebraic space.
The coverings of an object $u:\rmU \ra \S$ in the site $\S_{lis-et}$ are \'etale surjective maps $\{u_i:\rmU_i \ra \rmU|i\}$.
We will provide $\S_{lis-et}$ with the structure sheaf $\O_{\S}$. One defines a sheaf of $\O_{\S}$-modules $M$ on 
$\S_{lis-et}$ to be {\it cartesian} as in \cite[Definition 12.3]{L-MB},  that is, if for each map $\phi: \rmU \ra \rmV$ in
$\S_{lis-et}$, the induced map $\phi^{-1}(M_{|V_{et}})
\ra M_{|\rmU_{et}}$ is an isomorphism. In fact, it suffices to have this property
for all smooth maps $\phi$. {\it In this paper, we will restrict to complexes of 
$\O_{\S}$-modules $M$ whose cohomology sheaves are all cartesian.}
\vskip .2cm 
\begin{definition}
\label{K.G.def}
\begin{enumerate}[\rm (i)]
 \item {\it Throughout the paper, unless explicitly mentioned to the contrary, a complex will mean a cochain complex, that is,
 where the differentials are of degree $+1$.} 
 A bounded complex of $\O _{\S}$-modules $M$ is {\it strictly perfect}, if its  cohomology sheaves are all cartesian and
 locally on the site $\S_{lis-et}$, $M$ is a bounded complex of locally-free coherent
${\mathcal O}_{\S}$-modules. The complex $M$ is 
{\it perfect} if the cohomology sheaves are all cartesian, and 
locally on the site $\S_{lis-et}$, $M$ is quasi-isomorphic to 
a strictly perfect complex of $\O_{\S}$-modules.  
\item 
$M$ is
{\it pseudo-coherent},  if it is locally quasi-isomorphic to a bounded above complex of
${\mathcal O}_{\S}$-modules with bounded coherent cohomology sheaves, which are cartesian. (One may readily prove that if $M$ is perfect, it is pseudo-coherent. Observe that the usual definition of pseudo-coherence as in \cite{SGA6}
 does not require the cohomology sheaves to be bounded; we have included this hypothesis in the
definition of pseudo-coherence mainly for convenience.) 
\item
Let $\S'$ denote a closed algebraic sub-stack of $\S$. Then the category of all perfect (pseudo-coherent, strictly perfect) 
complexes with supports contained in $\S'$, along with quasi-isomorphisms forms a category with cofibrations and weak equivalences (see Definition ~\ref{def.Wald.cat}):
the {\it cofibrations } are those maps that are {\it degree-wise split monomorphisms}. It
will be denoted by ${\rm {Perf}}_{\S'}(\S)$ (${\rm {Pseudo}}_{\S'}(\S)$, ${\rm {StPerf}}_{\S'}(\S)$, \res); the K-theory {\it space}
(G-theory {\it space}) of $\S$ with supports in $\S'$ will be
defined to be the K-theory space of the category with cofibrations and weak equivalences ${\rm {Perf}}_{\S'}(\S)$ (${\rm {Pseudo}}_{\S'}(\S)$, \res) 
and denoted $\rmK_{\S'}(\S)$ ($G_{\S'}(\S)$, \res): the weak equivalences in these categories with cofibrations and weak equivalences  are quasi-isomorphisms. We distinguish these from the corresponding K-theory spectra
which will be denoted ${\mathbf K}_{\S'}(\S)$ (${\mathbf G}_{\S'}(\S)$, \res). 
\vskip .1cm
We also let ${\rm {Perf}}_{fl, \S'}(\S)$ (${\rm {Pseudo}}_{fl, \S'}(\S)$) denote the full subcategory
of ${\rm {Perf}}_{\S'}(\S)$ (${\rm {Pseudo}}_{\S'}(\S)$) consisting of complexes of flat $\O_{\S}$-modules in each degree. Observe from \cite[Chapitre I, Th\'eor\`eme 4.2.1.1]{Ill}  that flat $\O_{\S}$-modules have 
the {\it additional property that they are direct limits of finitely generated flat sub-modules at each stalk}. 
(Observe also that the existence of flat resolutions and the Waldhausen approximation theorem (see Theorem ~\ref{W.app.thm}) 
imply that one obtains a weak equivalence: $\rmK ({\rm {Perf}}_{\S'}(\S)) \simeq \rmK ({\rm {Perf}}_{fl, \S'}(\S))$.)
\end{enumerate}
\end{definition}
\begin{definition} 
\label{quasi.coh} We define a sheaf
of $\O_{\S}$-modules on $\S_{lis-et}$ to be {\it quasi-coherent} with respect to a given atlas, if its
restriction to the \'etale site of {\it the given} atlas for $\S$ is
quasi-coherent. Coherent sheaves and locally free coherent sheaves are defined
similarly. (Observe that this is slightly different from the usage in \cite{L-MB},
where a quasi-coherent sheaf  is also assumed to be cartesian as in \cite{L-MB}
Definition 12.3. However, such a definition would then make it difficult to
define a quasi-coherator that converts a complex of $\O_{\S}$-modules to a
complex of quasi-coherent $\O_{\S}$-modules. This justifies our choice. Since we always restrict to complexes
of $\O_{\S}$-modules whose cohomology sheaves are cartesian, the present definition works
out in practice to be more or less equivalent to the one in \cite{L-MB}.)  An
$\O_{\S}$-module will
always mean a sheaf of $\O_{\S}$-modules on $\S_{lis-et}$. The category of
$\O_{\S}$-modules will be denoted $Mod(\S, \O_{\S})$
(or $Mod(\S_{lis-et}, \O_{\S})$ to be more precise). 
\vskip .1cm
 Let ${\rm {Mod}}(\S, \O_{\S})$ (${\rm {QCoh}}(\S, \O_{\S})$, ${\rm
{Coh}}(\S, \O_{\S})$) denote the category of all $\O_{\S}$-modules (all
quasi-coherent $\O_{\S}$-modules, all coherent $\O_{\S}$-modules, \res).  
\end{definition}
Let $\rmX$ denote a scheme of finite type
over a regular Noetherian base scheme $\rmS$ and let $\rmG$ denote a smooth affine group scheme of finite type over $\rmS$ acting on 
$\rmX$. Then we point out that, if one restricts to quotient stacks of the form $[\rmX/G]$, then we 
may choose to work with the following (somewhat more familiar) choices.
\begin{definition}{\bf The case of quotient stacks}.
\label{eq.case.0} 
Assuming the above framework, 
let ${\rm Pseudo}([\rmX/\rmG]) = {\rm Pseudo}^{\rmG}(\rmX)$ where the right-hand-side denotes the category of bounded above complexes of
$\rmG$-equivariant $\O_{\rmX}$-modules (on the Zariski site of $\rmX$), with bounded coherent cohomology sheaves. Similarly, one may
let ${\rm Perf}([\rmX/\rmG]) = {\rm Perf}^{\rmG}(\rmX)$ denote the category of complexes of $\rmG$-equivariant $\O_{\rmX}$-modules that are locally quasi-isomorphic on
 the Zariski site of $\rmX$ to  bounded complexes of locally free $\O_{\rmX}$-modules with bounded coherent cohomology sheaves.
 \end{definition}
 \vskip .1cm
In this case we may replace ${\rm {Mod}}(\S, \O_{\S})$ (${\rm {QCoh}}(\S, \O_{\S})$, ${\rm
{Coh}}(\S, \O_{\S})$) by the category ${\rm {Mod}}^{\rmG}(\rmX)$ (${\rm {QCoh}}^{\rmG}(\rmX)$, ${\rm {Coh}}^{\rmG}(\rmX)$),
which will denote the category of all $\rmG$-equivariant $\O_{\rmX}$-modules ($\rmG$-equivariant quasi-coherent $\O_{\rmX}$-modules,
$\rmG$-equivariant coherent $\O_{\rmX}$-modules, \res). Moreover, in this context,  {\it cartesian sheaves} of $\O_{\S}$-modules 
 correspond to sheaves of $\O_X$-modules that are $\rmG$-equivariant.
\vskip .2cm
Let ${\mathbf A}$ denote any of the abelian categories considered in Definitions ~\ref{quasi.coh} or ~\ref{eq.case.0}.
Let ${\rm C}^b_{cc}({\mathbf A})$
(${\rm C}^b_{cart}({\mathbf A})$) denote the category of all bounded complexes of
objects in ${\mathbf A}$ with cohomology sheaves that are cartesian and coherent
(cartesian, \res). Similarly, we will let ${\rm C}_{bcc}({\mathbf A})$ denote the full
subcategory of complexes in ${\mathbf A}$ with cohomology sheaves that are
cartesian, coherent and vanish in all but finitely many degrees. These are
all bi-Waldhausen categories (see Definition ~\ref{def.Wald.cat}(iv)) with the same structure as above, that is, with
co-fibrations (fibrations) being maps of complexes that are degree-wise  split
monomorphisms (degree-wise split epimorphisms, \res), and weak equivalences being
maps that are quasi-isomorphisms.
 \vskip .2cm
We summarize in the next theorem several basic properties of the K-theory and G-theory of such stacks proven elsewhere: 
see \cite[section 2]{J-5} and \cite[sections 2, 3 and 5]{J-6}.
\begin{theorem}
\label{K}
(i) $ \S \mapsto \pi_*(\rmK (\S))$ is a contravariant functor from the category of algebraic stacks and morphisms of finite type to the category of graded
rings.
\vskip .2cm
\label{K=G}
(ii)  Let $\S$ denote a {\it smooth} algebraic stack. Then
the natural map $\rmK (\S) \ra \rmG (\S)$ is a weak equivalence. In case $\S'$ is a closed
algebraic sub-stack of $\S$, the natural map $\rmK_{\S'}(\S) \ra \rmG (\S')$ is a weak equivalence.
\vskip .2cm
\label{Gth.1}
(iii) The obvious inclusion functors 
\[{\rm C}^b_{cart}({\rm {Coh}}(\S, \O_{\S})) \ra {\rm C}^b_{cc}({\rm Mod}(\S, \O_{\S})) \ra {\rm C}_{bcc}({\rm Mod}(\S, \O_{\S})) \ra {\rm Pseudo}(\S)\]
induce weak equivalences on taking the 
corresponding K-theory spaces.
\vskip .2cm 
(iv)  Assume that every coherent sheaf on the algebraic stack $\S$ is the quotient of a vector
bundle. Then the obvious map $\rmK_{naive}(\S) \ra \rmK (\S)$ is a weak equivalence, where $\rmK_{naive}(\S) = \rmK ({\rm StPerf}(\S))$.
\end{theorem} 
\vskip .2cm 
\noindent
\textbf{Examples}
\vskip .2cm \noindent
$\bullet$ Assume the base scheme is a field $k$ and $\rmG$ is a linear algebraic group. On a quotient stack $[\rmX / \rmG]$, where the scheme $\rmX$ is assumed to be $\rmG$-quasi-projective (that is, admits a $\rmG$-equivariant locally closed immersion into a projective space ${\mathbb P}^n$ on which $\rmG$ acts linearly), every coherent sheaf is the quotient of a vector bundle. This follows from the work of Thomason: 
see \cite{T-3} (and also Theorem ~\ref{thomason.thm}).
\vskip .2cm \noindent
$\bullet$ Converse (See \cite[Theorems 1.1 and 1.2]{Tot} and \cite[Theorem 2.18]{EHKV}): 
Any smooth Deligne-Mumford stack $\S$ over a Noetherian base scheme, with generically trivial stabilizer is a quotient stack, $[\rmX/\rmG]$ for an algebraic space $\rmX$.
If in addition, the stack $\S$ is defined over a field, and the coarse moduli space is a scheme with affine diagonal, then the stack has the resolution property.
If $\S$ is a normal
Noetherian algebraic stack over $Spec \, {\mathbb Z}$ whose stabilizer groups at closed points are affine, and every coherent sheaf on $\S$ is a quotient of a vector bundle, then $\S$ is a quotient stack.
\begin{theorem} (See \cite[section 5]{J-6}.)
\label{loc.homotopy.prop}
(i) (Closed immersion). Let $i:\S' \ra \S$ denote the closed immersion of an algebraic sub-stack. Then the obvious map
$\rmG (\S') = \rmK ({\rm C}^b_{cart}({\rm {Coh}}(\S'))) 
\ra \rmK ({\rm C}^b_{cart, \S'}({\rm {Coh}}(\S)))$ is a weak equivalence, where ${\rm C}^b_{cart, \S'}({\rm {Coh}}(\S))$ denotes the full subcategory of ${\rm C}^b({\rm {Coh}}(\S))$ of complexes whose cohomology sheaves are cartesian and have supports in $\S'$.
\vskip .2cm 
(ii) (Localization) Let $i: \S'  \ra \S$ denote a closed immersion of  algebraic stacks with open complement
$j:\S'' \ra \S$. Then one obtains the fibration sequence $\rmG (\S') \ra \rmG (\S) \ra \rmG (\S'') \ra \Sigma \rmG (\S')$. 
\vskip .2cm
(iii) (Homotopy Property). Let $\S$ denote an algebraic stack and let $\pi: \S \times {\mathbb A}^1 \ra \S$ denote the obvious
projection. Then $\pi^*: \rmG (\S) \ra \rmG (\S \times {\mathbb A}^1)$ is a weak equivalence.
\end{theorem}

\section{\bf The $\rmG$-construction and relative K-theory}
The $\rmG$-construction is an un-delooping of K-theory. Recall that the Waldhausen K-theory (defined as in \cite{Wald}) involves
a delooping of K-theory given by the $S_{\bullet}$-construction. However, this means to define the K-groups, one needs to
perform an un-delooping. One way to do this is to simply take the loop-space on the space produced by the $\rmS_{\bullet}$-construction. 
The $\rmG$-construction is a way to perform this instead at the categorical level. Such a construction is in fact needed to obtain
a presentation of (higher) K-theory groups, as well as in being able to define $\lambda$-operations on higher K-theory.
\vskip .2cm
Though the basics of such a construction are outlined in \cite{GSVW} for categories with cofibrations and weak equivalences 
 (and in \cite{GG} for the $\rmQ$-construction on exact categories), their construction cannot be used in this paper, because of the following issue.
In this paper, a key technique we use is to reduce higher K-theory to the Grothendieck group
of a relative K-theory space. We would need a suitable G-construction that applies to this relative K-theory space. The construction
 appearing in \cite[Definition 2.2]{GSVW} (and which is related to the one in \cite{GG}), only applies to
the absolute case. Therefore, we provide a somewhat different, but related $\rmG$-construction that applies to relative K-theory
and is a suitable relative variant of the one considered in \cite[Definition 2.2]{GSVW}. The main difference stems from the
fact that for the $\rmS_{\bullet}$-construction applied to a category with cofibrations and weak equivalences $\B$, a path space may be defined readily by shifting the constituent categories in
the simplicial category $\rmS_{\bullet}(\B)$ by $1$ and throwing away the face map $d_0$. This works fine since $\rmS_0(\B)$ is a point, so that
the resulting path space is simplicially contractible. But with $\rmG_{\bullet}(\B)$, $\rmG_0(\B) = \B \times \B$, so that the above shifting technique
does not define a path space for ${\it w}\rmG_{\bullet}(\B)$. 
\vskip .1cm
Instead, we directly define the homotopy fiber of a map on the $\rmG_{\bullet}$-construction as in ~\eqref{Gwf.1} and make use 
of that to define the relative K-theory space using the $\rmG_{\bullet}$-construction.
\vskip .2cm
We will presently recall the {\it $\rmG$-construction for categories with cofibrations and weak equivalences } from \cite[section 2]{GSVW}. 
\begin{definition}
\label{def.Wald.cat} First, a category
is {\it pointed}, if it is equipped with a distinguished {\it zero object}: this zero object will often be denoted $*$.
Then, {\it a category with cofibrations and weak equivalences} will mean the following throughout the paper (see \cite[1.1]{Wald}):

\begin{enumerate}[\rm(i)]
\item A {\it pointed category} ${\bold A}$, provided with a subcategory ${\it co{\bold A}}$ of cofibrations satisfying the axioms \cite[Cof.1 through Cof.3 in 1.1]{Wald}
and also provided with a subcategory ${w{\bold A}}$ of weak equivalences satisfying the axioms \cite[Weq.1 and Weq.2 in 1.2]{Wald}, as well as the
{\it Saturation and Extension axioms} in \cite[1.2]{Wald}.
We will often refer to this as {\it a Waldhausen category.}\footnote{This terminology is consistent with \cite{T-T}, and is convenient, though we are told Waldhausen personally does not prefer to use it.}
\item
A subcategory of {\it fibrations} will denote a subcategory of the pointed category ${\bold A}$ satisfying the dual of the axioms \cite[Cof.1 through Cof.3 in 1.1]{Wald}.
A {\it bi-Waldhausen category} will denote a pointed category ${\bold A}$ provided with a subcategory of cofibrations, a subcategory of fibrations and 
a subcategory of weak equivalences, satisfying the above axioms, as well as the dual of \cite[Weq.2 in 1.2]{Wald}.
\item A functor ${\mathbf f}: {\bold A} \ra {\bold B}$ between categories with cofibrations (fibrations) and weak equivalences is an {\it exact functor} if it preserves 
the subcategories of cofibrations (fibrations) and weak equivalences.
\end{enumerate}
\end{definition}
\vskip .1cm
Given a category 
${\bold A}$ with cofibrations and weak equivalences, $\wS_{\bullet}({\bold A})$ will denote the simplicial category (that is, a simplicial object in the
category of all small categories), so that the objects of
$\wS_n({\bold A})$  are sequences \xymatrix{{A_{1} \,} \ar@{>->} @<2pt> [r] & {A_2\,}  \ar@{>->} @<2pt> [r] & {\cdots \, \,}  \ar@{>->} @<2pt> [r] & {A_n}} of cofibrations 
in ${\bold A}$ (with a choice of sub-quotients $A_i/A_{i-1}$). A morphism in this category between 
\[\xymatrix{{A_{1}\,} \ar@{>->} @<2pt> [r] & {A_2\,}  \ar@{>->} @<2pt> [r] &  {\cdots \, \,}   \ar@{>->} @<2pt> [r] & {A_n}} \mbox { and } \xymatrix{{B_{1}\,} \ar@{>->} @<2pt> [r] & {B_2\,}  \ar@{>->} @<2pt> [r] &  {\cdots \, \,} \ar@{>->} @<2pt> [r] & {B_n}}\]
is given by a sequence of  weak equivalences $A_i \ra B_i$ and $A_i/A_{i-1} \ra B_i/B_{i-1}$, compatible with the
given cofibrations. The {\it path-object} ${\PwS}_{\bullet}({\bold A})$ associated to $\wS_{\bullet}({\bold A})$ is the simplicial category given by  ${\PwS}_{\bullet}({\bold A})_n = \wS_{n+1}({\bold A})$, together with the functor $d_0:\wS_{n+1}({\bold A}) \ra
\wS_{n}({\bold A})$. The face map $d_i$ (the degeneracy $s_i$) of the simplicial category ${\PwS}_{\bullet}({\bold A})$ 
is the face map $d_{i+1}$ (the degeneracy $s_{i+1}$, \res) of the simplicial category $\wS_{\bullet}({\bold A})$: see \cite[p. 341]{Wald}. Then the {\it G-construction on}  ${\bold A}$ is the simplicial category defined by the 
fibered product ${\it w}\rmG_{\bullet}({\bold A}) =\PwS_{\bullet}({\bold A}){\underset {d_0, {\wS_{\bullet}({\bold A})}, d_0} \times}\PwS_{\bullet}({\bold A})$, 
 with the cofibrations and weak equivalences defined in the obvious manner using these structures on ${\rm PwS}_{\bullet}({\bold A})$ and
$\wS_{\bullet}({\bold A})$. 
\begin{remark}
Recall that the map $d^0:[n] \ra [n+1]$ in $\Delta$ omits $0$ in its image. Therefore, one may
readily see that the description of ${\rm {\it w}\rmG}_n{\bold A}$ as in \cite[section 1, (2)]{GSch} holds, which
is also the same as that in \cite[section 3]{Gray}. An $n$ simplex of this simplicial category is
given by two $n+1$ simplices of $\wS_{\bullet}{\bold A}$ whose successive quotients are provided with
compatible isomorphisms. In particular, a vertex of this simplicial category
is given by a pair of objects in $w{\bold A}$. 
\end{remark}
\vskip .2cm
Let ${\it wCof}$ denote the category whose objects are categories with cofibrations and weak equivalences in the above 
sense.  Since the construction ${\bold A} \mapsto w{\rmG}_{\bullet}({\bold A})$ is 
covariantly functorial, we will view it as a functor
\be \begin{equation}
\label{G.constr}
     w{\rmG}_{\bullet}: {\it wCof} \ra {\Delta^{op}-{\it wCof}},
\end{equation} \ee
\vskip .1cm \noindent
where $\Delta^{op}-{\it wCof}$ denotes the category of all simplicial objects in ${\it wCof}$. 
\vskip .2cm
We will apply these constructions to various categories with cofibrations and weak equivalences  we encounter, for example, the following ones. Let $\S$ denote an algebraic stack and let ${\rm {Perf}}(\S)$ denote the category of
perfect complexes on $\S$. For what follows, one may let $\S'$ denote a closed algebraic sub-stack of the given 
stack $\S$ and consider ${\rm {Perf}}_{\S'}(\S)$ also in the place of ${\rm {Perf}}(\S)$; however, for the most part we will explicitly
discuss only the case where $\S' = \S$.
\vskip .2cm 
One may readily verify 
 that the category ${\rm {Perf}}(\S)$ is pseudo-additive (see
\cite[Definition 2.3]{GSVW}: observe that it suffices to show that if \xymatrix{{A\quad }\ar@{>->} @<2pt> [r] &{C}} is
a degree-wise split monomorphism of complexes, then the natural maps $C{\underset A \oplus} C \ra C \times  C/A \leftarrow C \oplus C/A$ are
 quasi-isomorphisms. In fact, the second map is clearly an isomorphism and one may show that the first map is an isomorphism in each degree as follows. The assumption that \xymatrix{{A\quad }\ar@{>->} @<2pt> [r] &{C}}
 is a degree-wise split monomorphism shows that one has an isomorphism in each degree $n$: $C^n \cong A^n\oplus C^n/A^n$. This then 
 implies that in each degree $n$, one obtains the isomorphism: $C^n {\underset {A^n} \oplus} C^n \cong C^n \oplus (C^n/A^n)$.
 It is shown in \cite[Theorem 2.6]{GSVW} that $|{\it w}\rmG_{\bullet}({\bold A})| \simeq \Omega (|wS_{\bullet}({\bold A})|)$, provided
$\A$ is {\it pseudo-additive}.

\vskip .2cm
Next let ${\bf f}: {\bold A} \ra {\bold B}$ denote an {\it exact functor} of categories with cofibrations and weak equivalences.
We will assume that $\A$ and $\B$ are both pseudo-additive categories. First we let $|{\it w}\rmG_{\bullet}(\A)|$
($|{\it w}\rmG_{\bullet}(\B)|$) denote the topological realization of the diagonal of the bisimplicial set obtained by taking the
nerve of the simplicial category ${\it w}\rmG_{\bullet}(\A)$ (${\it w}\rmG_{\bullet}(\B)$, \res).
Since  ${\it w}\rmG_{\bullet}(\bB)$ is a simplicial category, the nerve functor ${\rm Nerve}$ applied to it produces a 
bisimplicial set. $\Delta{\rm Nerve}({\it w}\rmG_{\bullet}(\B))$ will denote its diagonal.
Now one may observe that each $0$-simplex of the simplicial space $\Delta{\rm Nerve}({\it w}\rmG_{\bullet}(\B))$, which is a pair of objects $(P, Q) \in \B$,
defines a connected component of the  space $|{\it w}\rmG_{\bullet}(\B)|$.
 We will choose for each connected component of $|{\it w}\rmG_{\bullet}(\B)|$ a $0$-simplex  that will remain fixed
throughout the following discussion, and will serve as the 
 base point for that component.
\vskip .1cm
We next consider
the {\it path space} $\rmP(|{\it w}\rmG_{\bullet}(\B)|)$ of pointed paths: it will consist of paths $p:\rmI =|\Delta [1]| \ra 
|{\it w}\rmG_{\bullet}(\B)|$, so that $p(1)$ is at the chosen base point for some connected component of $|{\it w}\rmG_{\bullet}(\B)|$.
Clearly the map sending a path $p$ to $p(0)$ defines a map $\pi: \rmP(|{\it w}\rmG_{\bullet}(\B)|) \ra |{\it w}\rmG_{\bullet}(\B)|$.
We define $w\rmG(\f)$ by the pullback square:
\vskip .1cm
\be \begin{equation}
 \label{Gwf.1}
\xymatrix{ {{\it w}\rmG({\bf f})} \ar@<1ex>[d] \ar@<1ex>[r] & {\rmP(|{\it w}\rmG_{\bullet}(\B)_*|)} \ar@<1ex>[d]^{\pi}\\
           {|{\it w}\rmG_{\bullet}(\A)|} \ar@<1ex>[r]^{|{\it w}\rmG _{\bullet}(\f)|} & {|{\it w}\rmG_{\bullet}(\B)|},}
\end{equation} \ee
\vskip .1cm \noindent 
where ${\rmP(|{\it w}\rmG_{\bullet}(\B)_*|)}$ denotes the path component of ${\rmP(|{\it w}\rmG_{\bullet}(\B)|)}$ that is sent
by $\pi$ to the path component of $|{\it w}\rmG_{\bullet}(\B)|$ containing the
base point $(*, *)$, with $*$ denoting the base point of the category $\B$.
\vskip .2cm
\subsection{The connected components of  ${{\it w}\rmG({\bf f})}$}
\label{path.comp}
\vskip .1cm
Observe that each triple 
$((P, Q), (\bar P, \bar Q), p)$ defines a connected component of ${{\it w}\rmG({\bf f})}$, where $(P, Q)$ is a $0$-simplex of ${\it w}\rmG_{\bullet}(\bA)$,
$(\bar P, \bar Q)$ is a $0$-simplex of ${\it w}\rmG_{\bullet}(\bB)$ in the same connected component of the base point $(*, *)$ of $|{\it w}\rmG_{\bullet}(\bB)|$ 
so that 
\be \begin{equation}
\label{path.p}
{\bf f}(P) = \bar P,  \quad \f(Q) = \bar Q, and \mbox{ p is a path in } |{\it w}\rmG_{\bullet}(\bB)| \mbox{ joining
the vertex } (\bar {\it P}, \bar {\it Q}) \mbox{ to the  base point} (*, *). 
\end{equation} \ee
\vskip .1cm
One may observe that the 1-simplices of the simplicial set $\Delta (\rmNerve({\it w}\rmG_{\bullet}(\bB)))$,
will be given by {\it pairs of commutative squares}:
\be \begin{equation}
\label{1.simpl.Diag}
 \xymatrix{{\bar P^1_0 } \ar@{>->}[d] \ar@<1ex>[r]^{\simeq} & {\bar P^2_0}  \ar@{>->}[d]\\   
            {\bar P^1_1 } \ar@<1ex>[r]^{\simeq} & {\bar P^2_1}},     \xymatrix{{\bar Q^1_0 } \ar@{>->}[d] \ar@<1ex>[r]^{\simeq} & {\bar Q^2_0}  \ar@{>->}[d]\\ 
                                                                               {\bar Q^1_1 } \ar@<1ex>[r]^{\simeq} & {\bar Q^2_1},}
\end{equation} \ee
\vskip .1cm \noindent
{\it together with an isomorphism} $\bar P^j_1/\bar P^j_0 \cong \bar Q^j_1/\bar Q^j_0$, for $j=1, 2$, where $(\bar P^j_i, \bar Q^j_i)$ are $0$-simplices in $w\rmG_{\bullet}(\bB)$, the vertical maps are cofibrations, while
the horizontal maps are weak equivalences.
\vskip .1cm
The path components of the topological space $|{\it w}\rmG_{\bullet}(\bB)|$ correspond to the path components of
the simplicial set $\Delta(\rmNerve({\it w}\rmG_{\bullet}(\bB)))$. Therefore, we obtain the 
following equivalent description of the connected components of ${\it w}\rmG(\f)$ by viewing the path $p$ in ~\eqref{path.p} as a {\it zig-zag}-path (see, for example: \cite[Chapter II, 7.3]{GZ})
in $\Delta(\rmNerve({\it w}\rmG_{\bullet}(\bB)))$
\be \begin{equation}
 \label{zig.zag.path}
 (\bar P, \bar Q) = (\bar P_0, \bar Q_0) \rightarrowtail (\bar P_1, \bar Q_1) \leftarrowtail (\bar P_2, \bar Q_2) \rightarrowtail \cdots \leftarrowtail (\bar P_{m-1}, \bar Q_{m-1}) \rightarrowtail (\bar P_m, \bar Q_m) =(*, *),
\end{equation} \ee
where each arrow $(\bar P_i, \bar Q_i) \rightarrowtail (\bar P_{i+1}, \bar Q_{i+1})$ $( (\bar P_i, \bar Q_i) \leftarrowtail (\bar P_{i+1}, \bar Q_{i+1}))$
is a $1$-simplex of 
$\Delta(\rmNerve({\it w}\rmG_{\bullet}(\bB)))$ in the above sense. (To see this, observe that
such a zig-zag path in $\Delta(\rmNerve({\it w}\rmG_{\bullet}(\bB)))$ corresponds to a simplicial map $q:{\rm I}_n \ra \Delta(\rmNerve({\it w}\rmG_{\bullet}(\bB)))$,
and therefore to a map on the realizations: $p:{\rm I} =|{\rm I}_n| \ra |\Delta(\rmNerve({\it w}\rmG_{\bullet}(\bB)))|$. Here ${\rm I}_n$
is the simplicial set considered in \cite[2.5.1]{GZ}.)
\vskip .1cm

\begin{example}
\label{eg}
 The basic application of the construction in \eqref{Gwf.1} is to the following situation. Let $\S_1$ denote a
smooth algebraic stack of the given stack $\S$ and let $\Delta[n] = Spec (\O_S[x_0, \cdots, x_n]/\Sigma x_i -1)$, where $\rmS$ is the base scheme, which is assumed to be a regular Noetherian scheme.  We may let $\S_0 = \S_1 \times_{\rmS} \Delta[n]$. Since $\Delta[n] \cong {\mathbb A}^n$, $\S_0$ is also smooth. 
Then the following are closed substacks of $\S_0$: 
\begin{enumerate}[\rm(i)]
\item $ \S_1 \times_{\rmS} \delta \Delta[n]$, where $\delta \Delta [n] =
{\underset {i =0, \cdots, n} \cup} \delta ^i \Delta [n]$ with $\delta ^i \Delta[n]$ denoting the $i$-th face of
$\Delta[n]$, 
\item $\S_1 \times_{\rmS} \Sigma$, 
where $\Sigma = {\underset {i=0, \cdots, n-1} \cup} \delta ^i \Delta [n]$, and, 
\item 
$\S_2 \times_{\rmS} \Delta [n]$, where $\S_2$ is any closed algebraic sub-stack of $\S_1$.
\end{enumerate}
In each of the above
 cases, one may let $f$ denote the corresponding closed immersion, and ${\bold f}$ denote the corresponding functor of categories with cofibrations and weak equivalences .
\end{example}
\section{\bf Another model for the homotopy fiber of the G-construction}
It is clearly preferable to obtain a categorical model for the homotopy fiber, whose realization identifies with the homotopy fiber
 of the realizations constructed in the last section. Here the difficulty is with obtaining a suitable model for the path space, 
 which seems to be possible only in special cases, like in the case of the $\rmS_{\bullet}$-construction. On the other hand,
 it is relatively straightforward to obtain a model for the homotopy cofiber, which we proceed to discuss next.
 \vskip .2cm
 For this we begin with a rather general construction. First, a simplicial category will denote a simplicial object in the category
 of all (small) categories, rather than a category that is simplicially enriched. Then a functor
 $f_{\bullet}: \rmS'_{\bullet} \ra \rmS_{\bullet}$ between simplicial categories will denote a collection of functors
 $\{f_n: \rmS'_n \ra \rmS_n|n\}$ so that they commute with the face maps and degeneracies. Let $*$ denote a chosen category with
 just one object denoted $*$, and only one morphism, namely the identity morphism of $*$.
 \begin{definition}
  \label{cat.mapping.cone}
  Let $f: \rmS'_{\bullet} \ra \rmS_{\bullet}$ denote a functor between two simplicial categories. Then we let
  ${\rm Cone}(\itf)_{\bullet}$ denote the simplicial category that is given in degree $n$ by the category:
  \be \begin{equation}
  \label{mapping.cone}
  \rmCone(\itf_{\bullet})_n = * \sqcup (\sqcup_{\alpha \in  {\rm \Delta[1]_n-\{(0, \cdots, 0), (1, \cdots, 1)\}}} \rmS'_n) \sqcup \rmS_n,
  \end{equation} \ee
  where we regard $\rmS_n$ as indexed by $(1, \cdots, 1) \in \Delta[1]_n$ and $*$ as indexed by $(0, \cdots, 0) \in \Delta[1]_n$, with the face maps and degeneracies induced from those of $\rmS'_{\bullet}$, $\rmS_{\bullet}$ and those of $\Delta[1]$.
  More precisely, we define the face map $d_i:\rmCone(\itf)_n \ra \rmCone(\itf)_{n-1}$ by:
  \begin{enumerate}[\rm(i)]
   \item the summand $*$ is sent to the summand * by the identity, 
   \item the summand $\rmS_n$ indexed by $(1, \cdots, 1) \in \Delta[1]_n$ is sent to the summand $\rmS_{n-1}$ indexed by $(1, \cdots, 1) \in \Delta[1]_{n-1}$
   by the face map $d_i^{S_{\bullet}}$,
   \item if $d_i(\alpha) = \alpha' $, with $\alpha \in \Delta[1]_n -\{(0, \cdots, 0), (1, \cdots, 1)\}$ and \\
   $\alpha ' \in \Delta[1]_{n-1}-\{(0, \cdots, 0), (1, \cdots, 1)\}$,
   then $d_i$ sends the summand $\rmS_n'$ indexed by $\alpha$ to the summand $\rmS_{n-1}'$ indexed by $\alpha'$ by the face map $d_i^{\rmS'_{\bullet}}$,
   \item if $d_i(\alpha) = (0, \cdots, 0)$, $\alpha \in \Delta[1]_n -\{(0, \cdots, 0), (1, \cdots, 1)\}$, then
   $d_i$ sends the $\rmS_n'$ indexed by $\alpha$ to $*$, and
   \item if $d_i(\alpha) = (1, \cdots, 1)$, $\alpha \in \Delta[1]_n -\{(0, \cdots, 0), (1, \cdots, 1)\}$, then $d_i$ sends the
   summand $\rmS_n'$ indexed by $\alpha$ to $\rmS_{n-1}$ by $f_{n-1} \circ d_i^{\rmS'_{\bullet}} = d_i^{\rmS_{\bullet}} \circ f_n: \rmS'_n \ra \rmS_{n-1}$.
  \end{enumerate}
  \vskip .1cm
We define the degeneracy $s_i: \rmCone(\itf)_{n-1} \ra \rmCone(\itf)_n$ by:
\begin{enumerate}[\rm(i)]
   \item the summand $*$ is sent to the summand * by the identity, 
   \item the summand $\rmS_{n-1}$ indexed by $(1, \cdots, 1) \in \Delta[1]_{n-1}$ is sent to the summand $\rmS_{n}$ indexed by $(1, \cdots, 1) \in \Delta[1]_{n}$
   by the degeneracy $s_i^{S_{\bullet}}$ and,
   \item  the summand $\rmS'_{n-1}$ indexed by $ \alpha \in \Delta[1]_{n-1} -\{(0, \cdots, 0), (1, \cdots, 1)\}$ is sent 
   to the summand $\rmS_{n}'$ indexed by $s_i(\alpha)\in \Delta[1]_{n} -\{(0, \cdots, 0), (1, \cdots, 1)\}$ by $s_i: \rmS_{n-1}' \ra \rmS_{n}'$.
\end{enumerate}
 \end{definition}
We skip the verification that, so defined, $\rmCone(\itf)_{\bullet}$ is a simplicial category, together with a natural functor
$\rmS_{\bullet} \ra \rmCone(\itf)_{\bullet}$, sending $\rmS_n$ to the summand in $\rmCone(\itf)_n$ indexed by $(1, \cdots, 1) \in \Delta[1]_n$.
In fact, one may also define a bisimplicial category 
\be \begin{equation}
\label{bisimpl.Cone}
\rmCone(\itf)_{\bullet, \bullet}
\end{equation} \ee
so that in bi-degree $(n, m)$ one has 
$\rmS'_m$ ($\rmS_m$) replacing $\rmS'_n$ ($\rmS_n$, respectively) in ~\eqref{mapping.cone}, 
and with the face maps and
degeneracies defined suitably. Then the simplicial category in ~\eqref{mapping.cone} will be the diagonal of this 
bisimplicial category.
\begin{proposition}
\label{cat.cone}
Let $f: \rmS'_{\bullet} \ra \rmS_{\bullet}$ denote a functor of simplicial categories. Let 
$\rmCone(\Delta \rmNerve(\itf))$ denote the
 mapping cone of the map of simplicial sets 
 \[\Delta \rmNerve(\itf) : {\rm \Delta} \rmNerve(\rmS'_{\bullet})\ra \Delta {\rmNerve}(\rmS_{\bullet}),\]
  which is defined as in Definition ~\ref{cat.mapping.cone}, with the simplicial set $\rmNerve(\rmS'_n)$ ($\rmNerve(\rmS_n)$)
 replacing $\rmS'_n$ ($\rmS_n$, \res).
Then 
 $\Delta \rmNerve(\rmCone(\itf)_{\bullet})$ can be identified with $\rmCone(\Delta \rmNerve(\itf))$.
\end{proposition}
\begin{proof} One may readily observe from its definition 
that the nerve functor ${\rm Nerve}$ commutes with finite coproducts. 
Therefore, it follows that 
\[ (\Delta {\rm Nerve}(\rmCone(\itf)))_n = {\rm Nerve}_n(*) \sqcup (\sqcup_{\alpha \in {\rm \Delta[1]_n-\{(0, \cdots, 0), (1, \cdots, 1)}\}} {\rm Nerve}_n(\rmS'_n)) \sqcup {\rm Nerve}_n(\rmS_n), \mbox{ }\]
while $\rmCone(\Delta {\rm Nerve}(\itf))_n$ is also given by the same set. We skip the verification that the structure maps for both simplicial sets
$\Delta {\rm Nerve}(\rmCone(\itf))$ and $\rmCone(\Delta \rmNerve(\itf))$ are the same.
\end{proof}
\vskip .1cm
Next let $\f:\bA \ra \bB$ denote an {\it exact} functor between categories with cofibrations and weak equivalences.
\begin{theorem}
 \label{cat.cone.G}
 Let ${\it w}\rmG_{\bullet}(\f): {\it w}\rmG_{\bullet}(\bA) \ra {\it w}\rmG_{\bullet}(\bB)$ denote the induced functor
 of the $\rmG$-constructions. Then one obtains the natural identification:
 \[\Delta \rmNerve(\rmCone({\it w}\rmG_{\bullet}(\f))) = \rmCone (\Delta \rmNerve({\it w}\rmG_{\bullet}(\bA)) {\overset {\Delta \rmNerve({\it w}\rmG_{\bullet}(\f))} \longrightarrow } \Delta \rmNerve({\it w}\rmG_{\bullet}(\bB))).\]
\end{theorem}
\begin{proof} This is clear from Proposition ~\ref{cat.cone}.\end{proof}
\begin{corollary}
\label{additivity.0}
 Assume the same hypotheses as in Theorem ~\ref{cat.cone.G}. Then we obtain the natural weak equivalence:
 \[{\it w}\rmG(\f) \simeq \Omega(|\rmNerve(Cone({\it w}\rmG_{\bullet}(\f)))|),\]
 where ${\it w}\rmG(\f)$ is the space defined in ~\eqref{Gwf.1}.
\end{corollary}
\begin{proof} This is clear in view of Theorem ~\ref{cat.cone.G} and the observation that ${\it w}\rmG(\f)$ is in fact an
infinite loop space. Since ${\it w}\rmG(\f)$ is the homotopy fiber of a map induced by 
${\it w}\rmG_{\bullet}(\f): {\it w}\rmG_{\bullet}(\bA) \ra {\it w}\rmG_{\bullet}(\bB)$, it suffices to observe that
$|{\it w}\rmG_{\bullet}(\bA)|$ and $|{\it w}\rmG_{\bullet}(\bB)|$ are both infinite loop spaces. 
\vskip .1cm
Here are some additional details. First, observe that for any category ${\mathbf C}$ with cofibrations and weak equivalences, there is
a natural map
\[ |{\it w}\rmG_{\bullet}({\mathbf C})| \ra \Omega |{\it w} {\rm S}_{\bullet}({\mathbf C})|.\]
Applying this to the functor ${\mathbf f}: {\mathbf A} \ra {\mathbf B}$ of categories with cofibrations and weak equivalences, we obtain the
homotopy commutative diagram:
\[ \xymatrix{{|{\it w} {\rmG}_{\bullet}({\mathbf A})|} \ar@<1ex>[r]^{} \ar@<1ex>[d] & {|{\it w} {\rmG}_{\bullet}({\mathbf B})|} \ar@<1ex>[r] \ar@<1ex>[d] & {|{\rm Nerve}(\rmCone({\it w}\rmG_{\bullet}({\mathbf f})))|} \ar@<1ex>[d] \\
             {\Omega|{\it w} {\rm S}_{\bullet}({\mathbf A})|} \ar@<1ex>[r]^{} & {\Omega|{\it w} {\rm S}_{\bullet}({\mathbf B})|} \ar@<1ex>[r] & {\Omega|{\rm Nerve}(\rmCone ({\it w}{\rm S}_{\bullet} ({\mathbf f})))|},}
\]
where the vertical maps are weak equivalences. The bottom row is clearly a diagram of infinite loop spaces. Therefore, the homotopy fiber of the first map 
in the second row is $ \Omega^2|{\rm Nerve}(\rmCone ({\it w}{\rm S}_{\bullet} ({\mathbf f})))|$. Since $|{\it w}{\rm G}_{\bullet}({\mathbf f})|$ is the homotopy fiber of 
the map $|{\it w}{\rm G}_{\bullet}({\mathbf A})| \ra |{\it w}{\rm G}_{\bullet}({\mathbf B})|$, one sees that it is weakly-equivalent to
$ \Omega^2|{\rm Nerve}(\rmCone ({\it w}{\rm S}_{\bullet} ({\mathbf f})))|$, which is weakly-equivalent to $\Omega|{\rm Nerve}(\rmCone({\it w}\rmG_{\bullet}({\mathbf f})))|$.
(See \cite{GSVW} and/or \cite{Wald} for further details.)
\end{proof}
\vskip .1cm
Next we consider additivity on exact sequences for $\rmCone({\it w}\rmG_{\bullet}(\f))$. Let $\f: \bA \ra \bB$ denote an exact
functor of categories with cofibrations and weak equivalences,  and let 
\be \begin{equation}
\label{map.g}
\g= {\it w}\rmG_{\bullet}(\f): {\it w}\rmG_{\bullet}(\bA) \ra {\it w}\rmG_{\bullet}(\bB)
\end{equation} \ee
\vskip .1cm \noindent
denote the induced functor of simplicial categories. Let $\rmE(\bA)$ ($\rmE(\bB)$) denote the Waldhausen category of cofibration sequences, that is, the category whose objects are short exact sequences of the form $F' \rightarrowtail F \twoheadrightarrow F''$ in $\bA$ ($\bB$, \res). These 
induce the simplicial categories $\rmE({\it w}\rmG_{\bullet}(\bA)) = {\it w}\rmG_{\bullet}(\rmE(\bA))$ and
$\rmE({\it w}\rmG_{\bullet}(\bB)) = {\it w}\rmG_{\bullet}(\rmE(\bB))$. Let $\rmE(\g)= \rmE{\it w}\rmG_{\bullet}(\f): \rmE{\it w}\rmG_{\bullet}(\bA) \ra \rmE{\it w}\rmG_{\bullet}(\bB)$. This is a functor between the
simplicial categories $\rmE{\it w}\rmG_{\bullet}(\bA)$ and $\rmE{\it w}\rmG_{\bullet}(\bB)$.
We consider its mapping cone, $\rmCone(\rmE(\g)) = \rmE(\rmCone(\g))$ as in Definition ~\ref{cat.mapping.cone}. 
Let 
\be \begin{equation}
  \label{ConewG.prdct}
 \rmCone({\it w}\rmG_{\bullet}(\f)) \boxtimes \rmCone({\it w}\rmG_{\bullet}(\f))
    \end{equation} \ee
denote the simplicial category given in degree $n$ by
\be \begin{equation}
     \label{ConewG.prdct.1}
 * \sqcup ( \sqcup_{\alpha \in \Delta[1]_n-\{(0, \cdots, 0), (1, \cdots, 1)\}} ({\it w}\rmG_n(\bA) \times  {\it w}\rmG_n(\bA)) ) \sqcup ({\it w}\rmG_n(\bB) \times {\it w}\rmG_n(\bB) ),
\end{equation} \ee
\vskip .1cm \noindent
with the face maps and degeneracies defined as in Definition ~\ref{cat.mapping.cone}. We define a functor of simplicial categories
\be \begin{equation}
 \label{Phi}
\Phi: \rmCone(\rmE(\g)) = \rmE(\rmCone(\g)) \ra Cone({\it w}\rmG_{\bullet}(\f)) \boxtimes \rmCone({\it w}\rmG_{\bullet}(\f)),
\end{equation} \ee
which will be induced by the projections to either factor, that is, sending a cofiber sequence $(X \rightarrowtail Z \twoheadrightarrow Y) \mapsto (X, Y)$.
\begin{theorem} (Additivity Theorem:I)
 \label{additivity}
 Assuming the above situation, the functor $\Phi$ induces a weak equivalence on taking the Nerve.
\end{theorem}
\begin{proof} For the proof it is convenient to view the Cone construction for a map of simplicial categories
as first defining a bisimplicial category (as in ~\eqref{bisimpl.Cone}), and then taking its diagonal.
Therefore, we reduce to considering a corresponding functor $\Phi_{\bullet, \bullet}$ of bi-simplicial categories:
then the proof reduces to showing that the corresponding functors ${\it w}\rmG_{\bullet}\rmE(\bA) \ra {\it w}\rmG_{\bullet}(\bA) \times {\it w}\rmG_{\bullet}(\bA)$ and
${\it w}\rmG_{\bullet}\rmE(\bB) \ra {\it w}\rmG_{\bullet}(\bB) \times {\it w}\rmG_{\bullet}(\bB)$
are weak equivalences. But this is proven in \cite[Theorem 2.10]{GSVW}.
\end{proof}
\vskip .2cm
\begin{definition}
 Let $\bA, \bB$ denote categories with cofibrations and weak equivalences. Let $F', F, F'': \bA \ra \bB$ denote
 exact functors. Then $F' \rightarrowtail F \twoheadrightarrow F''$ is a cofibration sequence, if for each object $A \in \bA$,
 $F'(A) \rightarrowtail F(A) \twoheadrightarrow F''(A)$ is a cofibration sequence in $\bB$.
\end{definition}
\vfill \eject
\begin{corollary} (Additivity Theorem for the mapping Cone.)
 \label{additivity.cor}
\end{corollary} Let $\f: \bA \ra \bB$ denote an exact functor of categories with cofibrations and weak equivalences  and let $\g={\it w}\rmG_{\bullet}(\f)$.
\vskip .1cm
(i) Assume $\{X'_m \rightarrowtail Z'_m \twoheadrightarrow Y'_m|m\}$ is a cofibration sequence in ${\it w}\rmG_{\bullet}(\bA)$, (that is,
 an object of $\rmE({\it w}\rmG_{\bullet}(\bA))$) and
$ \{X_m\rightarrowtail Z_m \twoheadrightarrow Y_m|m\}$ is a cofibration sequence in ${\it w}\rmG_{\bullet}(\bB)$, 
(that is, an object of  $\rmE({\it w}\rmG_{\bullet}(\bB))$)
so that they
are compatible under the functor $\g$, that is, $\g(X'_m) \cong X_m$, $\g(Y'_m) \cong Y_m$ and $\g(Z'_m) \cong Z_m$, for all $m$, 
where $\cong$ denotes
isomorphisms. Then the above data provides a cofibration sequence in $\rmCone({\it w}\rmG_{\bullet}(\f))$. 
\vskip .1cm
(ii) Let $F' \rightarrowtail F \twoheadrightarrow F''$ denote a cofibration sequence of exact functors from 
the Waldhausen category $\bC$ to 
 $\bA$, and let $\bar F' \rightarrowtail \bar F \twoheadrightarrow \bar F''$  denote a cofibration sequence of exact functors from the Waldhausen category
$\bC$ to  $\bB$, so that the diagram of functors
 \be \begin{equation}
 \label{compat.exact.functors}
 \xymatrix{{\rmf \circ F'} \ar@{>->}[r] \ar@<1ex>[d]^{\cong} & {\rmf \circ F} \ar@{->>}[r] \ar@<1ex>[d]^{\cong} & {\rmf \circ F''} \ar@<1ex>[d]^{\cong}\\
           {\bar F'} \ar@{>->}[r] & {\bar F} \ar@{->>}[r] & {\bar F''} }
\end{equation} \ee
\vskip .1cm \noindent
commutes. Denoting by ${\mathfrak F}$ (${\mathfrak F}'$, ${\mathfrak F}''$) the induced functor defined on $\bC$ and taking values in  $\rmCone ({\it w}\rmG_{\bullet}(f))$ induced by the
pair $((F, \bar F), (F', \bar F'), (F'', \bar F'')$, we obtain a weak equivalence:
\[{\mathfrak F} \simeq {\mathfrak F}' \vee {\mathfrak F}'',\]
where ${\mathfrak F}' \vee {\mathfrak F}''$ is the functor defined by $({\mathfrak F}' \vee {\mathfrak F}'')(C) = {\mathfrak F}'(C) \bigvee {\mathfrak F''}(C)$ for all objects $C \in \bC$.
 In other words, we obtain a weak equivalence ${\mathfrak F}(C) \simeq {\mathfrak F}'(C) \bigvee {\mathfrak F''}(C)$ for all objects $C \in \bC$.
\begin{proof}
 (i) follows readily from the definition of $\rmCone({\it w}\rmG_{\bullet}(\f))$ as in Definition ~\ref{cat.mapping.cone}. (ii) follows
 from Theorem ~\ref{additivity}, along the same lines as the proof of the corresponding result in \cite[Proposition 1.3.2]{Wald}.
\end{proof}
\begin{remark} Corollary ~\ref{additivity.0} translates the additivity theorem in Corollary ~\ref{additivity.cor} to an additivity theorem for the 
 homotopy fiber of the K-theory spaces associated to an exact functor of categories with cofibrations and weak equivalences. 
 This is then invoked in a key step showing that the relative Higher K-groups of smooth algebraic stacks have a pre-$\lambda$-ring structure: see 
 ~\eqref{lambda.on.sum}.
\end{remark}

\section{\bf Lambda operations on the higher K-theory of algebraic stacks: the pre-lambda ring structure}
\vskip .2cm
We first recall briefly the construction of lambda and Adams operations for Waldhausen style K-theory in \cite{GSch}. This involves first finding an un-delooping of algebraic K-theory at the
categorical level considered first in \cite{GSVW}, and then defining operations corresponding to the
exterior powers at the level of categories with cofibrations and weak equivalences. The definitions of these operations have been given in \cite{GSch} for categories with cofibrations and weak equivalences 
and based on the approach in \cite{Gray} (which is worked out in
the framework of the Quillen K-theory of exact categories), but 
under the assumption that the power operations preserve weak equivalences. To make sure that this hypothesis is satisfied in our context, we consider the left derived functors of these operations in the sense
of \cite{DP} and \cite[Chapitre I]{Ill}. (See also \cite[pp. 22-27]{SABK} for a readable account.) The techniques of the 
last 2 sections, and the technique of \cite{BL} whereby higher K-groups can be reduced to a relative form of Grothendieck groups
then enable us to show that the higher K-groups of all smooth
algebraic stacks are pre-$\lambda$-rings.

\vskip .2cm
We next recall the definition of $\Lambda^{\it k}$ (as in \cite{Gray}). Let ${\rm {Mod}}(\S, \O_{\S})$ denote the category of all
$\O_{\S}$-modules. Let $Filt_{\it k}(\S)$ denote the category whose objects are sequences of split monomorphisms in ${\rm {Mod}}(\S, \O_{\S})$:
\vskip .2cm
\[M = \xymatrix{ {M_{0, 1} \,} \ar@{>->} @<2pt> [r] & {M_{0,2} \,} \ar@{>->} @<2pt> [r]& {{\cdots} {\quad }} \ar@{>->} @<2pt> [r] &{M_{0, k}}}\]
\vskip .2cm \noindent
together with sub-quotients $M_{i,j} = M_{0,j}/M_{0, i} \in {\rm {Mod}}(\S, \O_{\S})$, for $i<j$. We let
\be \begin{equation}
\label{Lambda.0}
\Lambda^{\it k}(M) = M_{0, 1} \Lambda M_{0,2} \Lambda \cdots \Lambda M_{0, k},
\end{equation} \ee
which is defined locally on $\S$ as the quotient of $(M_{0,1} \otimes M_{0,2} \otimes \cdots \otimes M_{0, k})$ by the submodule
generated by terms of the form $m_1 \otimes m_2 \otimes \cdots \otimes m_i \otimes m_{i+1} \otimes \cdots m_n $, with $m_i = m_{i+1}$ for some $i$. One may verify readily that the functor $\Lambda ^{\it k}$ applied to an object $M$ in $Filt_{\it k}(\S)$, where each $M_{0, i}$ is {\it flat} and is the direct limit of its finitely generated flat sub-modules, will produce $\Lambda ^{\it k}(M)$ which is also {\it flat}. (To see this, one may localize on $\S$ to reduce to the case of a local ring, in which case flat and finitely generated implies free.)
\vskip .2cm \noindent
Henceforth, we will denote $\Lambda^{\it k}(M)$ by $\Lambda^{\it k}(M_{0,1}, \cdots, M_{0, k})$. 
Now one may observe that there is a natural map
\be \begin{equation} 
\label{map.to.Lambda}
M_{0,1} \otimes \cdots \otimes M_{0, k} \ra \Lambda ^{\it k}(M_{0,1}, \cdots, M_{0, k}). \footnote{The definitions
~\eqref{Lambda.0} and ~\eqref{map.to.Lambda} are defined on $Filt_{\it k}(\S)$ for convenience. We will be 
in fact applying these often to objects belonging to ${\rm {\wS}}_{\it k}{\rm {Perf}}_{fl, \S_0}(\S)$ which 
are filtered objects satisfying more restrictive conditions.}
\end{equation} \ee
\vskip .2cm
Next we  proceed to consider the derived functor of the exterior power: we follow \cite{DP} or \cite[Chapitre I, section 4]{Ill} in this. Observe that the exterior power $\Lambda ^{\it k}$ is a non-additive
functor, and therefore one needs to use simplicial techniques in defining its derived functors.
Let ${\rm P}_{fl}(\S)$ denote the full subcategory of flat 
${\mathcal O}_{\S}$-modules which are the direct limits of their finitely generated flat sub-modules. 
Let ${\rm {Perf}}_{fl}(\S)$ denote the full subcategory of perfect complexes that
consist of flat ${\mathcal O}_{\S}$-modules in each degree, which are also the direct limits of their finitely generated flat sub-modules. 
If $\S_0$ is a closed
algebraic substack of $\S$, ${\rm Perf}_{fl, \S_0}(\S)$ will denote the full subcategory of ${\rm Perf}_{fl}(\S)$ consisting of
complexes of flat $\O_{\S}$-modules with supports in $\S_0$.
\vskip .1cm
Lemma ~\ref{flat.res} below shows the existence of functorial flat resolutions, so that one may restrict to flat 
$\O_{\S}$-modules without loss of generality.
Then one obtains an imbedding 
\be \begin{equation}
\label{i}
i:{\rm {Perf}}_{fl, \S_0}(\S) \ra {\rm {Cos.mixt}}({\rm P}_{fl}(\S)),
\end{equation} \ee
where ${\rm {Cos.mixt}}({\rm P}_{fl}(\S))$ denotes the category of all cosimplicial-simplicial objects of ${\rm P}_{fl}(\S)$ as follows. (See \cite[Chapitre I, 4.1]{Ill}.) Let
$\cdots \ra K^{-n} \ra K^{-n+1} \ra \cdots \ra K^0 \ra K^1 \ra \cdots \ra K^m\ra \cdots $ denote an object in ${\rm {Perf}}_{fl}(\S)$. One first
sends it to the double complex in the second quadrant with $\cdots \ra K^{-n} \ra \cdots K^0$ along the negative $x$-axis, with $K^0$ in position $(0, 0)$, and
the complex $K^0 \ra K^1 \ra \cdots \ra K^m \ra \cdots $ along the positive $y$-axis. Next one applies de-normalization functors (see Appendix B) that produce the cosimplicial-simplicial object, $i(K)$ in ${\rm P}_{fl}(\S)$ from this, that is,
an object in ${\rm {Cos.mixt}}({\rm P}_{fl}(\S))$. 
\vskip .1cm 
Let $\rmN= \rmN^{\it v} \circ \rmN_{\it h}$ denote the normalization functor as in ~\eqref{NvNh} and let $\Tot$ denote the functor defined in  ~\eqref{Tot}.
Recall $\rmN = \rmN^{\it v} \circ \rmN_{\it h}$ sends a
cosimplicial-simplicial object to a double complex, and $\Tot$ denotes taking the total complex of the
corresponding double complex.
We define a morphism $K' \ra K$ in ${\rm {Cos.mixt}}({\rm P}_{fl}(\S))$ to be a quasi-isomorphism, if the 
induced map on applying the functor $\Tot \circ \rmN$ is a quasi-isomorphism. One may now verify that the composition $\Tot \circ \rmN \circ {\it i} = id$  so that the functor 
$i$ in fact induces a faithful functor of the associated derived categories
obtained by inverting quasi-isomorphisms. (Observe from Appendix B that the normalization functor $\rmN$ and the
functor $i$ preserve flatness.)
\vskip .2cm
Let ${\rm {\wS}}_{\it k}{\rm {Perf}}_{fl, \S_0}(\S)$ denote the category whose objects are sequences of cofibrations
\[\xymatrix{{\rmK_{0,1} \,} \ar@{>->} @<2pt>[r]& {\rmK_{0,2}\,} \ar@{>->} @<2pt>[r] &{\cdots \, \,} \ar@{>->} @<2pt> [r]&{\rmK_{0,k}\,},}\]
together with choices of sub-quotients $\rmK_{i,j} = \rmK_{0,j}/\rmK_{0,i} \in {\rm {wPerf}}_{fl}(\S)$, for $i<j$. One defines
\newline \noindent
${\rm {\wS}}_{\it k}({\rm {Cos.mixt}}({\rm P}_{fl}(\S))$ similarly.
We define
\be \begin{equation}
\label{Lambda.1}
\Lambda^{\it k}_{cs}:{\rm {\wS}}_{\it k}({\rm {Cos.mixt}}({\rm P}_{fl}(\S))) \ra w{\rm {Cos.mixt}}({\rm P}_{fl}(\S)), \quad \Lambda^{\it k}: {\rm {\wS}}_{\it k}{\rm {Perf}}_{fl, \S_0}(\S) \ra {\rm {wPerf}}_{fl, \S_0}(\S)
\end{equation} \ee
as functors of categories with cofibrations and weak equivalences  in the following manner.
\vskip .1cm
Let $K = 
\xymatrix{{\rmK_{0,1}\,} \ar@{>->} @<2pt>[r]& {\rmK_{0,2}\,} \ar@{>->} @<2pt>[r] &{\cdots \, \,} \ar@{>->} @<2pt> [r]&{\rmK_{0,k} \,}} \in {\rm {\wS}}_{\it k}({\rm {Cos.mixt}}({\rm P}_{fl}(\S)))$ ($\in {\rm {\wS}}_{\it k}{\rm {Perf}}_{fl, \S_0}(\S)$, \res). We define $\Lambda^{\it k}_{cs}$ to be the functor
induced by applying $\Lambda^{\it k}:Filt_{\it k}({\rm P}_{fl}(\S)) \ra {\rm P}_{fl}(\S)$ in each cosimplicial-simplicial degree. The functor $\Lambda^{\it k}: {\rm {\wS}}_{\it k}{\rm {Perf}}_{fl, \S_0}(\S) \ra {\rm {wPerf}}_{fl, \S_0}(\S)$ is defined by 
\be \begin{equation}
\label{Lambda1}
\Lambda ^{\it k}(K) = \Tot \circ \rmN(\Lambda^{\it k}_{cs}({\it i}(K))), \quad K \in {\rm {\wS}}_{\it k}{\rm {Perf}}_{fl, \S_0}(\S),
\end{equation} \ee
where $\rmN =\rmN^{\it v} \circ \rmN_{\it h}$ once again.
\vskip .1cm
\begin{lemma} 
\label{flat.res}
(Functorial flat resolutions) Let $\S_0$ denote a closed algebraic sub-stack of the given stack $\S$. 
Then there exists a functor ${\mathfrak F}:{\rm {\wS_{\bullet} Perf}}_{\S_0}(\S) \ra {\rm {\wS_{\bullet}Perf}}_{fl, \S_0}(\S)$ having 
the following properties. Let $\rmU: {\rm {\wS_{\bullet}Perf}}_{{\it fl}, \S_0}(\S) \ra {\rm {\wS_{\bullet}Perf}}_{\S_0}(\S)$ denote the
obvious forgetful functor. Then, there exists a natural transformation $ U \circ {\mathfrak F}   \ra id$ so that for each $M \in {\rm {\wS_{\bullet}Perf}}_{\S_0}(\S)$, the corresponding map $(U \circ  {\mathfrak F})(M) \ra M$
is a quasi-isomorphism.
\end{lemma}
\begin{proof}
Recall that the stack $\S$ is of finite type over the regular Noetherian base scheme $\rmS$. Therefore, the lisse-\'etale site $\S_{lis-et}$ is essentially small. Given an $M \in {\rm  {Mod}}(\S, \O_{\S})$, one may define 
\[{\mathfrak F}(M) = {\underset {U \in \S_{lis-et}} \oplus} {\underset {\phi \in Hom(j_{U!}j_U^*(\O_{\S}), M)} \oplus} j_{U!, \phi}j_U^*(\O_{\S}).\]
(Here $j_{U!, \phi}j_U^*(\O_{\S}) = j_{U!}j_U^*(\O_{\S})$.) We define a surjection $\epsilon_{-1}:{\mathfrak F}(M) \ra M$
by mapping the summand indexed by $\phi \in Hom(j_{U!}j_U^*(\O_{\S}), M)$ to $M$ by the map $\phi$. Given a map $f: M' \ra M$ in ${\rm {Mod}}(\S, \O_{\S})$, one defines the induced map ${\mathfrak F}(f): {\mathfrak F}(M') \ra {\mathfrak F}(M)$ by sending the summand $j_{U!, \phi}j_U^*(\O_{\S})$
to $j_{U!, f\circ \phi}j_U^*(\O_{\S})$ by the identity map. Now one may readily see that
the assignment $M \mapsto {\mathfrak F}(M)$ is functorial in $M$. Moreover, if $M' \ra M$ is a split monomorphism in ${\rm {Mod}}(\S, \O_{\S})$,
the induced map ${\mathfrak F}(M') \ra {\mathfrak F}(M)$ is also a split monomorphism. One may repeatedly apply the functor
${\mathfrak F}$ to the kernel of $\epsilon _{-1}$ to obtain a resolution ${\mathfrak F}_{\bullet}(M) \ra M$. It follows, therefore, that the functor ${\mathfrak F}$ induces a
functor $\rmS_{\bullet}{\rm {Perf}}_{\S_0}(\S) \ra \rmS_{\bullet}{\rm {Perf}}_{fl, \S_0}(\S)$ that preserves cofibrations (that is, degree-wise split monomorphisms) and weak equivalences.\end{proof}
 \vskip .2cm
 Recall that we only consider complexes of sheaves of $\O$-modules whose cohomology sheaves are cartesian. The following proposition shows that 
 the exterior power operations preserve cofibrations, weak equivalences and the property that the cohomology sheaves are cartesian.
\begin{proposition} 
\label{cart.coh.sheaves}
The functor $\Lambda^{\it k}: {\rm {\wS}}_{\it k}{\rm {Perf}}_{fl, \S_0}(\S) \ra {\rm {wPerf}}_{fl, \S_0}(\S)$ is a functor of categories with cofibrations and weak equivalences.
\end{proposition}
\begin{proof} First observe that exterior powers  preserve degree-wise split monomorphisms, 
so that these functors in fact preserve cofibrations. That they preserve weak equivalences follows essentially 
from the observation that the exterior power  $\Lambda^{\it k}: {\rm {\wS}}_{\it k}{\rm {Perf}}_{fl, \S_0}(\S) \ra {\rm {wPerf}}_{fl, \S_0}(\S)$ is in fact a derived functor (see \cite[Chapter I, Proposition 4.2.1.3]{Ill} ).  \cite[Proposition 4.2.1.3]{Ill}  applies to the case where all the monomorphisms $\rmK_{0, i} \ra \rmK_{0, i+1}$ are the identity maps. This Proposition may be extended to apply to
the situation in hand as follows. Recall that a map $f:K = \xymatrix{{\rmK_{0,1}\,} \ar@{>->} @<2pt>[r]& {\rmK_{0,2}\,} \ar@{>->} @<2pt>[r] &{\cdots \, \,} \ar@{>->} @<2pt> [r]&{\rmK_{0,k}}} \ra
L = \xymatrix{{L_{0,1}\,} \ar@{>->} @<2pt>[r]& {L_{0,2}\,} \ar@{>->} @<2pt>[r] &{\cdots \, \,} \ar@{>->} @<2pt> [r]&{L_{0,k}}}$ in  ${\rm {\wS}}_{\it k}{\rm {Perf}}_{fl, \S_0}(\S)$ is a weak equivalence, if the induced maps $f_{0, i}: \rmK_{0, i} \ra L_{0, i}$ are all quasi-isomorphisms.
  One may replace the mapping cone $\rmCone(\itf)$ by $K$, and assume each $\rmK_{0, i}$ is acyclic. 
  Then the same argument as in the proof of \cite[Chapitre I, Proposition 4.2.1.3]{Ill}
  applies to show that $K$ is the locally filtered colimit (locally filtered in the sense of \cite[2.2.5, Chapitre I]{Ill})
of  complexes $\rmK_{\alpha}= \xymatrix{{\rmK_{0,1}(\alpha)\,} \ar@{>->} @<2pt>[r]& {\rmK_{0,2}(\alpha)\,} \ar@{>->} @<2pt>[r] &{\cdots \, \,} \ar@{>->} @<2pt> [r]&{\rmK_{0,k}(\alpha)}}$, with the property that  each $\rmK_{0, i}(\alpha)$ is a bounded complex of finitely
generated free $\O_{\S, p}^{str.h}$-modules at each geometric point $p$, and that 
 there is a chain null-homotopy of each $\rmK_{\alpha}$ at each stalk. (Here $\O_{\S, p}^{str.h}$ denotes the strict henselization of $\O_{\S}$ at the geometric point $p$.)
\vskip .2cm
In fact, one may apply  \cite[Chapitre I, Proposition 4.2.1.3]{Ill} to see that each $\rmK_{0, i}$ is
such a locally filtered colimit. By re-indexing, we may assume that we obtain a locally filtered direct system of complexes
$\{\rmK_{0, 1}(\alpha) \ra \rmK_{0, 1}(\alpha) \ra \cdots \ra \rmK_{0, k}(\alpha)|\alpha\}$ so that each $\rmK_{0, i}$ is such a filtered colimit. 
Then one may replace $\rmK_{0,i}(\alpha)$ for $i\ge 2$ by the mapping cylinder (see \cite[1.1.2]{T-T}) of the given map $\rmK_{0, i-1}(\alpha) \ra \rmK_{0, i}(\alpha)$, so that 
one may assume the maps $\rmK_{0, i-1}(\alpha) \ra \rmK_{0, i}(\alpha)$, for all $i \ge 2$ are cofibrations.
This observation will then provide the required extension of \cite[Chapitre I,  Proposition 4.2.1.3]{Ill} to complexes provided with
a finite increasing filtration by subcomplexes. Since $\Lambda^{\it k}$ commutes with taking stalks,  filtered colimits, and preserves
chain homotopies, it follows that $\Lambda^{\it k}: {\rm {\wS}}_{\it k}{\rm {Perf}}_{fl, \S_0}(\S) \ra {\rm {wPerf}}_{fl, \S_0}(\S)$ preserves weak equivalences.
\vskip .2cm
Now it suffices to show that $\Lambda^{\it k}(K)$ has cohomology sheaves which are cartesian. In view of
\cite[ Lemma 3.6]{Ol}, it suffices to show the following: if $f:U \ra V$ denotes a {\it smooth} map
between schemes in $\S_{lis-et}$,  then $f^*{\mathcal H}^i(\Lambda^{\it k}(\rmK_{|V_{et}})) \simeq {\mathcal H}^i(\Lambda^{\it k}(\rmK_{|U_{et}}))$ for all $i$. The definition of $\Lambda^{\it k}$ above shows that $f^*(\Lambda^{\it k}(\rmK_{|V_{et}})) \cong
\Lambda^{\it k}(f^*(\rmK_{|V_{et}})) \cong \Lambda^{\it k}(\rmK_{|U_{et}})$: the last quasi-isomorphism follows
from the observation that $\Lambda^{\it k}$ preserves quasi-isomorphisms. Next observe that
$f^*$ is an exact functor since ${\it f}$ is smooth. Therefore, taking cohomology sheaves commutes
with $f^*$, proving that ${\mathcal H}^i( \Lambda^{\it k}(\rmK_{|U_{et}})) \cong {\mathcal H}^i(f^*(\Lambda^{\it k}(\rmK_{|V_{et}}))) \cong f^*({\mathcal H}^i(\Lambda^{\it k}(\rmK_{|V_{et}})))$: this proves that $\Lambda^{\it k}(K)$ has cohomology sheaves which are cartesian.
\end{proof}
In order to show that we obtain power operations in $K$-theory, 
one needs to verify that certain conditions are satisfied by the exterior powers.
These are the conditions denoted (E1) through (E5) in \cite{Gray}, in \cite[section 2]{GSch} and also in \cite[Definition 1.1]{KZ}. We summarize them here:
\begin{itemize}
\item[(E1)] Given \xymatrix{{V_{1}\,} \ar@{>->} @<2pt>[r]& {V_{2}\,} 
\ar@{>->} @<2pt>[r] &{\cdots \, \,} \ar@{>->} @<2pt> [r]&{V_{k} \,} 
\ar@{>->} @<2pt> [r]&{W_{1} \,} \ar@{>->} @<2pt> [r]&{\cdots \, \,} \ar@{>->} @<2pt> [r] &{W_n}}  $ \in \wS_{k+n}{\rm {Perf}}_{fl, \S_0}(\S)$, there is a natural map
$\Lambda ^{\it k}(V_1, ..., V_{\it k}) \otimes \Lambda ^n(W_1,..., W_n) \ra \Lambda^{k+n}(V_1,..., V_{\it k}, W_1,..., W_n)$. These maps are  associative in the obvious sense. 
\item[(E2)] Given \xymatrix{{V_{1}\,} \ar@{>->} @<2pt>[r]& {V_{2}\,} 
\ar@{>->} @<2pt>[r] &{\cdots \, \,} \ar@{>->} @<2pt> [r]&{V_{k}\, } \ar@{>->} @<2pt> [r]& {W_{1} \,} \ar@{>->} @<2pt> [r]&{\cdots \, \,} \ar@{>->} @<2pt> [r] &{W_n}}  
$ \in \wS_{k+n}{\rm {Perf}}_{fl, \S_0}(\S)$, there is a natural map
\[\Lambda ^{k+n}(V_1,..., V_{\it k}, W_1,..., W_n) \ra 
\Lambda^{\it k}(V_1,..., V_{\it k}) \otimes \Lambda^n(W_1/V_{\it k},..., W_n/V_{\it k}).\]
These maps are
associative in the obvious sense. The above conditions are for any choice of quotient
objects $W_1/V_{\it k},..., W_n/V_{\it k}$.
\item[(E3)] Given \[\xymatrix{{V_{1}\, } \ar@{>->} @<2pt>[r]& {V_{2}\,} 
\ar@{>->} @<2pt>[r] &{\cdots \, \,} \ar@{>->} @<2pt> [r]&{V_{k}} 
\ar@{>->} @<2pt> [r]&{W_{1}\,} \ar@{>->} @<2pt> [r] &{\cdots \, \,} \ar@{>->} @<2pt> [r] &{W_n \,}\ar@{>->} @<2pt> [r] &{U_1 \,}
\ar@{>->} @<2pt> [r] & {\cdots \, \,} \ar@{>->} @<2pt> [r] & {U_l} } \]
$ \in \wS_{k+n+l}{\rm {Perf}}_{fl, \S_0}(\S)$, the following diagram commutes
\vskip .2cm \hskip -.5cm
\xymatrix{{\Lambda^{k+n}(V_1,..., V_{\it k}, W_1,..., W_n) \otimes \Lambda^l(U_1,..., U_l)}  \ar@<-1ex>[d]  \ar@<1ex>[r] & X \ar@<1ex>[d]\\
{\Lambda^{\it k}(V_1,..., V_{\it k}) \otimes \Lambda^n(W_1/V_{\it k},..., W_n/V_{\it k}) \otimes \Lambda ^l(U_1/V_{\it k},..., U_l/V_{\it k})}  \ar@<1ex>[r] & Y,}
\vskip .2cm \noindent
where 
\vskip .2cm
$X = {\Lambda^{k+n+l}(V_1,..., V_{\it k}, W_1,..., W_n, U_1,..., U_l)}$, and 
\vskip .2cm
$Y =  {\Lambda ^{\it k}(V_1,..., V_{\it k}) \otimes \Lambda^{n+l}(W_1/V_{\it k},.., W_n/V_{\it k}, U_1/V_{\it k},..., U_l/V_{\it k})}$.
\item[(E4)] 
Given 
\[\xymatrix{{V_{1}\,} \ar@{>->} @<2pt>[r]& {V_{2}\,} 
\ar@{>->} @<2pt>[r] &{\cdots \, \,} \ar@{>->} @<2pt> [r]&{V_{k}\,} 
\ar@{>->} @<2pt> [r]&{W_{1}\,} \ar@{>->} @<2pt> [r] &{\cdots \, \,} \ar@{>->} @<2pt> [r] &{W_n \,}\ar@{>->} @<2pt> [r] &{U_1 \,}
\ar@{>->} @<2pt> [r] & {\cdots \, \,} \ar@{>->} @<2pt> [r] & {U_l} } \] 
$ \in \wS_{k+n+l}{\rm {Perf}}_{fl, \S_0}(\S)$, the following diagram commutes
\vskip .2cm \small
\xymatrix{{\Lambda^{\it k}(V_1,..., V_{\it k}) \otimes \Lambda^{n+l}(W_1,..., W_n , U_1,..., U_l)} \ar@<1ex>[r] \ar@<-1ex>[d] &{\Lambda^{k+n+l}(V_1,..., V_{\it k}, W_1,..., W_n, U_1,..., U_l)} \ar@<1ex>[d]\\
{\Lambda^{\it k}(V_1,..., V_{\it k}) \otimes \Lambda^n(W_1,..., W_n) \otimes \Lambda ^l(U_1/W_n,..., U_l/W_n)} \ar@<1ex>[r] & Z,}
\vskip .2cm \noindent \normalsize
where
\vskip .2cm
$Z = { \Lambda ^{k+n}(V_1,..., V_{\it k}, W_1,.., W_n) \otimes \Lambda^l(U_1/W_n,..., U_l/W_n)}$.
\item[(E5)] Given 
\[ \xymatrix{{V_{1}\,} \ar@{>->} @<2pt>[r]& {V_{2}\,} 
\ar@{>->} @<2pt>[r] &{\cdots \, \,} \ar@{>->} @<2pt> [r]&{V_{k}\,} 
\ar@{>->} @<2pt> [r]&{W_{1}\,}  \ar@{>->} @<2pt> [r] &{W_2\,}\ar@{>->} @<2pt> [r] &{U_1\,}
\ar@{>->} @<2pt> [r] & {\cdots \, \,} \ar@{>->} @<2pt> [r] & {U_l} } \]
$ \in \wS_{k+2+l}{\rm {Perf}}_{fl, \S_0}(\S)$, the following sequence of perfect complexes is an {\it exact} sequence:
\vskip .2cm
$\Lambda^{k+l+1}(V_1,..., V_{\it k}, {W_1}, U_1,..., U_l) \ra
{\Lambda^{k+l+1}(V_1,..., V_{\it k}, W_2, U_1,..., U_l)} $
\vskip .2cm
$\ra {{\Lambda^{\it k}(V_1,..., V_{\it k}) \otimes \Lambda^{l+1}(W_2/{ W_1}, U_1/{W_1},..., U_l/{W_1})}}$,
\vskip .2cm \noindent
that is, the first map is a cofibration and the second is its quotient. 
\item[$({\rm E5})_0$]
We will also allow the case $k=0$, which is the statement 
that we get an exact sequence 
\[\Lambda^{l+1}(W_1, U_1, \ldots, U_l) \ra \Lambda ^{l+1}(W_2, U_1, \ldots, U_l) \ra \Lambda^{l+1}(W_2/W_1, U_1/W_1, \ldots, U_l/W_1). \]
\end{itemize}
\vskip .1cm \noindent
Here is an outline of how to establish these properties. 
\subsubsection{}
\label{tensor.TotN}
Let $\rmN= \rmN^{\it v} \circ \rmN_{\it h}$ as before. Now we observe from ~\eqref{pairings.EZ.AW} that the functor $\Tot \circ \rmN$ is
compatible with the obvious tensor structures: that is, given $C$, $C' \in {\rm {Cos.mixt}}({\rm P}_{fl}(\S))$, there 
are natural maps $\Tot(\rmN(C)) \otimes \Tot (\rmN(C')) \ra \Tot\circ \rmN(C \otimes C') $, 
and $\Tot\circ \rmN(C \otimes C') \ra \Tot(\rmN(C)) \otimes \Tot (\rmN(C'))$,
that are associative (and are in fact quasi-isomorphisms). (Observe that the tensor structure on ${\rm {Cos.mixt}}({\rm P}_{fl}(\S))$
is given by sending $C=\{C^i_j|i, j\}$ and $C'=\{{C'}^i_j|i,j\}$ to $C\otimes C' = \{C^i_j \otimes {C'}^i_j|i, j\}$. The tensor structure 
on complexes is the obvious one.)
\vskip .2cm 
We will now consider 
the statement in (E1). Since the functor $\Lambda^{\it k}_{cs}$ is induced by the functor $\Lambda ^{\it k}$ (as in 
 (~\ref{Lambda.0})), the existence of the corresponding map in (E1) when  
\[\xymatrix{{V_{1}\, } \ar@{>->} @<2pt>[r]& {V_{2}\,} 
\ar@{>->} @<2pt>[r] &{\cdots \, \,} \ar@{>->} @<2pt> [r]&{V_{k}} 
\ar@{>->} @<2pt> [r]&{W_{1}\, }\ar@{>->} @<2pt> [r]& {\cdots \, \,} \ar@{>->} @<2pt> [r] &{W_n}} \]
belongs
to ${\rm {\wS}}_{k+n}{\rm {Cos.mixt}}({\rm P}_{fl}(\S))$ is clear. (In fact, this follows readily from the case where
each $V_i$ and $W_j$ is an $\O_S$-module in ${\rmP}_{fl}(\S)$.) Now one applies the functor $\Tot \circ \rmN$ to both sides and makes
use of the observation ~\ref{tensor.TotN} above to obtain the map in (E1), when 
\[\xymatrix{{V_{1} \, } \ar@{>->} @<2pt>[r]& {V_{2} \,} 
\ar@{>->} @<2pt>[r] &{\cdots \, \,} \ar@{>->} @<2pt> [r]&{V_{k}\,} 
\ar@{>->} @<2pt> [r]&{W_{1}\,} \ar@{>->} @<2pt> [r] &{\cdots \, \,} \ar@{>->} @<2pt> [r] &{W_n}\,}  \in \wS_{k+n}{\rm {Perf}}_{fl, \S_0}(\S).\]
\vskip .2cm
One proves (E2) by first observing the corresponding statement is true when $\Lambda^{\it k}$ is replaced
by $\Lambda^{\it k}_{cs}$ and for  \[\xymatrix{{V_{1}\, } \ar@{>->} @<2pt>[r]& {V_{2}\,} 
\ar@{>->} @<2pt>[r] &{\cdots\, \, } \ar@{>->} @<2pt> [r]&{V_{k}\,} 
\ar@{>->} @<2pt> [r]&{W_{1}\, } \ar@{>->} @<2pt> [r]&{\cdots \, \,} \ar@{>->} @<2pt> [r] &{W_n}}  \in \wS_{k+n}({\rm {Cos.mixt}}({\rm P}_{fl}(\S))).\]
This follows readily from the definition of the functor $\Lambda^{\it k}$ in ~\eqref{Lambda.0}. Next apply
$\Tot \circ \rmN$ and use the observation in ~\ref{tensor.TotN} to obtain the associativity of the maps there.
The remaining assertions (E3) through (E5) are established similarly: one observes these are true for
the functor $\Lambda^{\it k}_{cs}$, and then applies the functor $\Tot \circ \rmN$ along with the observation in
~\ref{tensor.TotN}.
\subsubsection{}
\label{tensor.1}
Let $\otimes: {\rm {\wS_{\bullet}Perf}}_{fl, \S_0}(\S) \times {\rm {\wS_{\bullet}Perf}}_{fl, \S_0}(\S) \ra {\rm {\wS_{\bullet}Perf}}_{fl, \S_0}(\S)$
denote the functor taking two perfect complexes of flat $\O_{\S}$ modules and sends it to their  tensor product. Observe that this preserves cofibrations and weak equivalences
in each argument. 
\subsubsection{}
Following \cite[sections 3 and 5]{Gray} and \cite[sections 1 and 2]{GSch}, we next define for each $A \in \Delta$ and each integer $k \ge 1$, a category (actually a partially ordered set) $\Gamma ^{\it k}(A)$. First one defines $\gamma(A)$ to be the partially ordered
set $\{\rmL, \rmR\} \sqcup A$ with $\rmL <a$, $\rmR <a$ for all $ a \in A$. For $c, d \in A$, we have $c<d$ if $c<d$ in the usual order in $A$. 
In fact, if $A$ is the category $0 \ra 1 \ra 2 \ra \cdots \ra n$, then $ \gamma(A)$ is given by the diagram:
\[ \xymatrix{{L} \ar@<1ex>[dr]\\
              & 0 \ar@<1ex>[r] & {1} \ar@<1ex>[r] & 2 \ar@<1ex>[r] & {\cdots} \ar@<1ex>[r] & n\\
              {R} \ar@<1ex>[ur]}
\]

$\Gamma (A)$ is the category whose objects are the morphisms in the category $\gamma (A)$, except for the identity morphisms
$\rmL \ra \rmL$ and $\rmR \ra \rmR$. The morphism $j \ra i$ will be denoted $i/j$. The morphisms in the category
$\Gamma (A)$, $i'/j' \ra i/j$ are the obvious commutative squares. We define a sequence of categories $\Gamma^k(A)$, for $k \ge 1$, with
$\Gamma ^1(A) = \Gamma (A)$. We take for the objects of $\Gamma^{\it k}(A)$, the collections
$\alpha = (i_1/l_1, *_2, i_2/l_2, *_3,..., *_{\it k}, i_{\it k}/l_{\it k})$, where for each $r$ the following conditions
are satisfied: 
\vskip .2cm
(A1)  $i_r \in \gamma (A)$, $l_r \in \gamma (A)$, and $*_r \in \{ \wedge, \otimes\}$,
\vskip .2cm
(A2) $l_r \le i_r$, $i_r \in A$, and
\vskip .2cm
(A3) if $*_r= \wedge$ and $r>1$, then $l_{r-1} = l_r$ and $i_{r-1} \le i_r$.
\vskip .2cm \noindent
(Note: in \cite[section 2]{GSch}, $\wedge$ ($\otimes$) is replaced by $\diamond$ ($\boxtimes$, \res).)
One defines morphisms and {\it exact sequences} in the category $\Gamma ^{\it k}(A)$ as in \cite[sections 1 and 2]{GSch} or \cite[section 5]{Gray}.
One may call these exact sequences {\it cofibration sequences}.
With this structure, the categories $\Gamma ^{\it k}(A)$ may be viewed as  categories with cofibrations. 
Moreover, one defines a functor 
\be \begin{equation}
\label{Xi}
\Xi: \Gamma(A_1) \times \cdots \times \Gamma (A_{\it k}) \ra \Gamma ^{\it k}(A_1 \cdots A_{\it k}),
\end{equation} \ee
where $A_1 \cdots A_{\it k}$ is the concatenation. (See \cite[2.4]{GSch} or \cite[section 5]{Gray}.)
\vskip .2cm
Now one considers the categories ${\rm {Exact}}(\Gamma ^{\it k}(A), {\rm {Perf}}_{fl, \S_0}(\S))$ of exact functors (that is, functors preserving cofibrations)
$F:\Gamma ^{\it k}(A) \ra {\rm {Perf}}_{fl, \S_0}(\S)$, for $k \ge 1$, $n \ge 0$. One may define the subcategory
${\rm {{\it w}Exact}}(\Gamma ^{\it k}(A), {\rm {Perf}}_{fl, \S_0}(\S))$ to have the same objects as  ${\rm {Exact}}(\Gamma ^{\it k}(A), {\rm {Perf}}_{fl, \S_0}(\S))$, and where a morphism $\phi:F' \ra F$ is a natural transformation, so that  for each object $\gamma \in \Gamma ^{\it k}(A)$, the induced map $\phi(\gamma): F'(\gamma) \ra F(\gamma)$
belongs to the subcategory ${\rm {{\it w}Perf}}_{fl, \S_0}(\S)$. 
As in \cite[section 1, p. 5]{GSch} one
obtains the identification 
\[{\it w}\rmG_A{\rm {Perf}}_{{\it fl}, \S_0}(\S) = {\rm {{\it w}Exact}}(\Gamma (A), {\rm {Perf}}_{{\it fl}, \S_0}(\S)),\]
for each $A \in \Delta$. Next one
defines 
\be \begin{equation}
\label{bigwedge}
\wedge^{\it k}:{\rm { {\it w}Exact}}(\Gamma (A), {\rm {Perf}}_{fl, \S_0}(\S)) \ra {\rm {{\it w}Exact}} (\Gamma ^{\it k}(A), {\rm {Perf}}_{fl, \S_0}(\S))
\end{equation} \ee
\vskip .2cm \noindent
by the same formula as in \cite[section 7]{Gray}. We will recall this here: first, we  denote 
\[\Lambda^{\it k}:Filt_{\it k}({\rm {Perf}}_{fl}(\S)) \ra {\rm {Perf}}_{fl}(\S)\]
applied to 
an object \xymatrix{{\rmK_{0,1}\,} \ar@{>->} @<2pt>[r]& {\rmK_{0,2}\,} \ar@{>->} @<2pt>[r] &{\cdots \, \,} \ar@{>->} @<2pt> [r]&{\rmK_{0,k}}} $ \in Filt_{\it k}({\rm {Perf}}_{fl, \S_0}(\S))$ by 
$\rmK_{0,1} \Lambda \rmK_{0,2} \Lambda \cdots \Lambda \rmK_{0, k}$. If 
 $*_{i}$ denotes
either $\Lambda $ or $\otimes$ for each $1 \le i \le k$, we will let 
$\rmK_{0, 1} *_1 \rmK_{0,2}*_2  \cdots *_{\it k} \rmK_{0,k}$  denote an
iterated product involving $\Lambda$ and $\otimes $ with $\Lambda$ always
having higher precedence than $\otimes$. Let $M \in {\rm {{\it w}Exact}}(\Gamma (A), {\rm {Perf}}_{fl, \S_0}(\S))$. Now $\wedge^{\it k}(M)$ applied to the object 
$(i_1/l_1, *_2, i_2/l_2, \cdots, *_{\it k}, i_{\it k}/l_{\it k}) \in \Gamma ^{\it k}(A)$ is given by 
\[M(i_1/l_1)*_2 M(i_2/l_2) \cdots *_{\it k}M(i_{\it k}/l_{\it k}). \]
\vskip .2cm 
Observe that if $f:\S' \ra \S$ is a map of algebraic stacks, the induced map 
\[f^*: {\rm {\wS_{\bullet}Perf}}_{fl, \S_0}(\S)  \ra {\rm {\wS_{\bullet}Perf}}_{fl, \S_0'}(\S')\]
(where $\S_0' = \S'\times _{\S} \S_0$), 
commutes with $\wedge ^{\it k}$. This follows from the above definition, the observation that
$f^*$ commutes with $\otimes$ and $\Lambda^{\it k}$ (as in ~\ref{Lambda.0}),
as well as from Proposition
~\ref{N.DN} (in Appendix B) which shows it commutes with the functors ${\rm N}$ and ${\rm DN}$. 
\vskip .2cm
Let ${\it w}Exact(\Xi, \quad)$ denote the functor obtained by pre-composing ${\it w}Exact(\quad, \quad)$ and $\Xi$ 
(with $\Xi$ applied to the first factor of ${\it w}Exact(\quad, \quad)$).
On replacing $A$ by the concatenation $A_1...A_{\it k}$, and
following $\wedge ^{\it k}$ by the composition with $Exact(\Xi, \quad)$, we obtain 
\be \begin{equation}
\label{lambda.1}
\lambda ^{\it k}: {\rm {{\it w}Exact}}(\Gamma (A_1... A_{\it k}), {\rm {Perf}}_{fl, \S_0}(\S)) \ra {\rm {{\it w}Exact}}(\Gamma (A_1) \times ... \times \Gamma (A_{\it k}), {\rm {Perf}}_{fl, \S_0}(\S)).
\end{equation} \ee
\vskip .2cm \noindent
(Observe that since all cofibrations are maps of complexes that are degree-wise split injective, the extension
condition in \cite[ 4.3 Definition and 4.4 Remark]{GSch} is satisfied. Therefore, the last term may be identified with ${\rm {{\it w}\rmG}}_{\bullet}^{\it k}{\rm Perf}_{fl}(\S)$, which is the $k$-th iterate of the construction in ~\eqref{G.constr}.)
Therefore, identifying the first term with ${\rm {sub}}_{\it k}{\rm {{\it w}\rmG}}_{\bullet}{\rm {Perf}}_{fl, \S_0}(\S)$ (where ${\rm {sub}}_{\it k}$ denotes the $k$-th subdivision  which produces a multi-simplicial set of order $k$: see \cite[section 4]{Gray}), one obtains the exterior power operation:
\be \begin{equation}
\label{lambda.2}
\lambda^{\it k}: {\rm {sub}}_{\it k}{\rm {\it w}\rmG}_{\bullet}{\rm {Perf}}_{fl, \S_0}(\S) \ra {\rm {{\it w}\rmG}}_{\bullet}^{\it k}{\rm {Perf}}_{fl, \S_0}(\S).
\end{equation} \ee
\vskip .2cm \noindent
It is shown in \cite[section 4]{Gray} that the realization of the first term is homeomorphic to $|{\rm {\wG_{\bullet}Perf}}_{fl, \S_0}(\S)| \simeq \rmK_{\S_0}(\S)$. 
It is shown in \cite[p. 264]{GSVW} that the realization of the last is homeomorphic to $|{\it w}\rmG_{\bullet}{\rm {Perf}}_{\S_0}(\S)| \simeq |\Omega {\rm {\wS_{\bullet}Perf}}_{fl, \S_0}(\S)| \simeq \rmK_{\S_0}(\S)$. Therefore, 
~\eqref{lambda.2} defines
the exterior power operations on $\rmK_{\S_0}(\S)$ as the map:
\be \begin{equation}
\label{lambda.3}
\lambda^{\it k}: |{\rm {\it w}\rmG}_{\bullet}{\rm {Perf}}_{fl, \S_{\rm 0}}(\S)| \ra |{\rm {{\it w}\rmG}}_{\bullet}{\rm {Perf}}_{fl, \S_0}(\S)|.
\end{equation} \ee
\vskip .2cm \noindent

Moreover, the naturality of the above operations shows 
that they are (strictly) compatible with pull-back maps associated to morphisms
$\rmf:\S' \ra \S$ of algebraic stacks, that is, the operations $\lambda^{\it k}$ are compatible with the induced map
\[\xymatrix{{|{\it w}\rmG_{\bullet}({\rm {Perf}}_{\it fl, \S_{\rm 0}}(\S))|} \ar@<1ex>[r]^{|{\it w}\rmG _{\bullet}(f^*)|} & {|{\it w}\rmG_{\bullet}({\rm {Perf}}_{\it fl, \S'_{\rm 0}}(\S'))|} }, \]
 where $ {\S'}_0 = \S' \times_{\S} \S_0$. Therefore, in view of the pull-back square
\be \begin{equation}
 \label{Gwf}
\xymatrix{ {{\it w}\rmG(f)} \ar@<1ex>[d]^{\pi'} \ar@<1ex>[r] & {\rmP(|{\it w}\rmG_{\bullet}({\rm {Perf}}_{\it fl, \S'_{\rm 0}}(\S'))_*|)} \ar@<1ex>[d]^{\pi}\\
           {|{\it w}\rmG_{\bullet}({\rm {Perf}}_{\it fl, \S_{\rm 0}}(\S))|} \ar@<1ex>[r]^{|{\it w}\rmG _{\bullet}(f^*)|} & {|{\it w}\rmG_{\bullet}({\rm {Perf}}_{\it fl, \S'_{\rm 0}}(\S'))|},}
\end{equation} \ee
\vskip .1cm \noindent
and the observation that the map induced by $\lambda^k$ on the path space ${\rmP(|{\it w}\rmG_{\bullet}({\rm {Perf}}_{\it fl, \S'_{\rm 0}}(\S'))_*|)}$ is compatible with the
map $\pi$, one obtains induced maps 
\be \begin{equation}
 \label{lambda.on.ho.fiber}
\lambda^k: {{\it w}\rmG(f)} \ra {{\it w}\rmG(f)},
\end{equation} \ee
\vskip .1cm \noindent
compatible under the map $\pi'$ with the corresponding operation $\lambda ^k$ on ${|{\it w}\rmG_{\bullet}({\rm {Perf}}_{fl, \S_0}(\S))|}$.
Taking the map $\rmf$ to be the closed immersion $\S' \ra \S$ of algebraic stacks, shows one may define 
exterior power operations in relative K-theory, that is, on $\rmK (\S, \S')$ which is defined as the canonical homotopy fiber of
the restriction map $\rmK (\S) \ra \rmK (\S')$. 
\vskip .2cm
We proceed to verify these satisfy the usual relations so that $\pi_*\rmK_{\S_0}(\S)$ is a pre-$\lambda$-ring without unit 
 when $\S$ is smooth. For this it is necessary to define the pullback squares, for each fixed $k  \ge 1$:
\be \begin{equation}
 \label{Gwf.k.1}
\xymatrix{ {{\it w}\rmG_k(f)} \ar@<1ex>[d]^{\pi'} \ar@<1ex>[r] & {\rmP(|{\rm sub}_{\it k}{\it w}\rmG_{\bullet}({\rm {Perf}}_{\it fl, \S'_{\rm 0}}(\S'))_*|)} \ar@<1ex>[d]^{\pi}\\
           {|{\rm sub}_k{\it w}\rmG_{\bullet}({\rm {Perf}}_{\it fl, \S_{\rm 0}}(\S))|} \ar@<1ex>[r]^{|{\rm sub}_k{\it w}\rmG _{\bullet}(f^*)|} & {|{\rm sub}_k{\it w}\rmG_{\bullet}({\rm {Perf}}_{\it fl, \S'_{\rm 0}}(\S'))|} \mbox{ and }}
\end{equation} \ee
\vskip .1cm \noindent
\be \begin{equation}
 \label{Gwf.k.2}
\xymatrix{ {{\it w}\rmG^k(f)} \ar@<1ex>[d]^{\pi'} \ar@<1ex>[r] & {\rmP(|{\it w}\rmG^k_{\bullet}({\rm {Perf}}_{\it fl, \S'_{\rm 0}}(\S'))_*|)} \ar@<1ex>[d]^{\pi}\\
           {|{\it w}\rmG^k_{\bullet}({\rm {Perf}}_{\it fl, \S_{\rm 0}}(\S))|} \ar@<1ex>[r]^{|{\it w}\rmG^k _{\bullet}(f^*)|} & {|{\it w}\rmG^k_{\bullet}({\rm {Perf}}_{\it fl, \S'_{\rm 0}}(\S'))|}.}
\end{equation} \ee
\vskip .1cm \noindent
Here ${\rm sub}_k{\it w}\rmG_{\bullet}({\rm {Perf}}_{\it fl, \S'_{\rm 0}}(\S'))_*$ and ${\it w}\rmG^k_{\bullet}({\rm {Perf}}_{\it fl, \S'_{\rm 0}}(\S'))_*$
denote the corresponding path component containing the base point.
Then $\lambda^k$ defines a map from the bottom two vertices and the top right vertex of the first square to the corresponding vertices of the second square, making the corresponding diagrams commute, so
that one obtains induced maps:
\be \begin{equation}
 \label{lambda.on.ho.fiber.1}
\lambda^k: {{\it w}\rmG_k(f)} \ra {{\it w}\rmG^k(f)}.
\end{equation} \ee
Making use of the observation that the vertices of the square ~\eqref{Gwf.k.1} are homeomorphic to the corresponding vertices of 
the square ~\eqref{Gwf} and further observing that the vertices of the square ~\eqref{Gwf.k.2} are weakly-equivalent to the
corresponding vertices of the square ~\eqref{Gwf}, one sees that the maps $\lambda^k$ in ~\eqref{lambda.on.ho.fiber.1}
are variants of the same maps $\lambda^k$ considered in ~\eqref{lambda.on.ho.fiber}, but able to handle $k$-different arguments,
$(\rmK_1, \cdots, \rmK_k)$ of complexes in ${\rm {Perf}}_{fl, \S_0}(\S))$.
Moreover the above extension of the operations $\lambda^k$ enable us to verify these satisfy the usual relations, so that we will show $\pi_*\rmK_{\S_0}(\S)$ is a pre-$\lambda$-ring without unit.
 Making use of the above observations and Example ~\ref{eg}, 
we are able to reduce this to verifying the above 
relations hold in certain relative
Grothendieck groups.
\vskip .2cm
Let $\S$ denote a given smooth algebraic stack over the given base scheme $\rmS$. 
Let 
\[\Delta[n] = Spec ({\mathcal O}_{
S}[x_0,..., x_n]/(\Sigma _i x_i -1)),\]
let $\delta^i \Delta [n]$ denote its $i$-th face, and let $\delta \Delta [n]$ denote its boundary that is, $\cup _{i=0}^n \delta^i \Delta [n]$. 
 The relative
K-theory space 
 $\rmK_{\S' \times \Delta[n]}(\S \times \Delta [n], \S \times  {\underset {i=0, \cdots, k} \cup} \delta^i \Delta [n])$ 
 is defined as $w\rmG({\it i}_{\it k}^*)$, where 
 \be \begin{equation}
 \label{i*k}
 i_k^*: {\rm Perf}_{{\it fl},\S' \times \Delta[n]}(\S \times \Delta[n]) \ra {\rm Perf}_{{\it fl},\S' \times  \Delta[n]}(\S \times   {\underset {i=0, \cdots, k} \cup} \delta^i \Delta[n]) 
 \end{equation} \ee
 is the functor 
 of categories with cofibrations and weak equivalences  induced by the closed immersion 
 \[ \S \times {\underset {i=0, \cdots, k} \cup}\delta^i \Delta[n] \ra \S \times \Delta [n]. \]
 
\begin{lemma}
\label{red.pi0}
$\rmK_{\S' \times \Delta [n]}(\S \times \Delta [n], \S \times  ({\underset {i=0, \cdots, k} \cup} \delta^i \Delta [n]))$ is contractible for all $n$ and all $0 \le k  \le n-1$.
\end{lemma}
\begin{proof} One proves this using ascending induction on $k$. Observe that $\Delta [n]$ is isomorphic to the affine space ${\mathbb A}^n _{\S}$. 
The case $k=0$ follows from the fact $\S \times \Delta [n]$ and $\S \times \delta ^0 \Delta [n] \cong \S \times \Delta [n-1]$ are smooth, and therefore 
$\rmK_{S' \times \Delta [n]}(\S \times \Delta [n], \S \times \delta ^0 \Delta [n]) \simeq G(\S' \times \Delta [n], \S' \times \Delta [n-1])$, and because G-theory has been shown to have the homotopy property. (See  \cite[Theorem 5.17]{J-6}.) To continue the induction, one uses the fibration sequence: 
\vskip .2cm
$ \rmK_{\S' \times \Delta [n]}(\S \times \Delta [n], \S \times {\underset {i=0, \cdots, k} \cup}\delta ^i \Delta [n]) \ra \rmK_{\S' \times \Delta [n]}( \S \times \Delta [n], \S \times {\underset {i=0, \cdots, k-1} \cup}\delta ^i \Delta [n]) $
\vskip .1cm
$ \ra \rmK_{\S' \times \delta ^{\it k}\Delta [n]}(\S \times \delta ^{\it k} \Delta [n], \S \times ((\delta ^0 \Delta [n] \cup \cdots \cup \delta ^{k-1} \Delta [n]) \cap  \delta ^{\it k} \Delta [n])) \cong $
\vskip .1cm
$\rmK_{\S' \times \Delta [n-1]}(\S \times \Delta [n-1], \S \times ({\underset {i=0, \cdots, k-1} \cup}\delta ^i \Delta [n-1]))$.
\vskip .2cm \noindent
The last two terms are contractible by the inductive hypothesis, so that the first one is also. 
\end{proof}
\vskip .1cm
Let
\be \begin{equation}
i^*:  {\rm Perf}_{{\it fl},\S' \times \Delta[n]}(\S \times \Delta[n])  \ra  {\rm Perf}_{{\it fl},\S' \times \delta \Delta[n]}(\S \times \delta \Delta[n])
\end{equation} \ee
denote the functor induced by the closed immersion $\S \times \delta \Delta[n] \ra \S \times \Delta [n]$.
\begin{proposition}
\label{reduct.pi0}
One obtains the isomorphism:
\[\pi_n\rmK_{\S'}(\S) \cong \pi_0\rmK_{\S' \times \Delta [n]}(\S \times \Delta [n], \S \times \delta \Delta [n]) \cong \pi_0({\it w}\rmG({\it i}^*)).\]
\end{proposition}
\begin{proof}
The key idea is the observation that one obtains a fibration sequence:
\vskip .2cm
$ \rmK_{\S' \times \Delta [n]}(\S \times \Delta [n], \S \times \delta \Delta [n]) \ra \rmK_{\S' \times \Delta [n]}(\S \times \Delta [n], \S \times \Sigma ) \ra \rmK_{\S' \times \Delta [n-1]}(\S \times \Delta [n-1], \S \times \delta \Delta [n-1]),$
\vskip .2cm \noindent
where $\Sigma = {\underset {i=0, \cdots, n-1} \cup} \delta ^i \Delta [n]$, the last map is the restriction to the
last face of $\Delta [n]$, and the first map is the obvious inclusion of the fiber. The middle term is 
contractible by the above lemma, so that the long exact sequence associated to the above fibration provides us with 
an isomorphism
\vskip .2cm
$\pi_{k-1}(\rmK_{\S' \times \Delta[n]}(\S \times \Delta [n], \S \times \delta \Delta [n])) \cong \pi_{k}(\rmK_{\S' \times \Delta [n-1]}(\S \times \Delta [n-1], \S \times \delta \Delta [n-1]).$
\vskip .2cm \noindent
Repeating this $n$-times, we obtain the first isomorphism in the Proposition. The second isomorphism in the Proposition follows from the fact ${\it w}\rmG({\it i}^*)$ is the
homotopy fiber of the $\rmG$-construction applied to the functor $i^*$ defined as in ~\eqref{Gwf.1} and ~\eqref{Gwf}.
\end{proof}
\begin{remarks}
 \begin{enumerate}[\rm(i)]
\item Lemma ~\ref{red.pi0} and Proposition ~\ref{reduct.pi0} are clearly inspired by \cite[Lemmas (1.2.1) and (1.2.2)]{BL}, which play
a key role in the construction of the Bloch-Lichtenbaum spectral sequence.
\vskip .1cm
\item
The product structure on relative $K$-theory 
may now be viewed as the pairing: 
\begin{multline}
 \begin{split}
  \label{pairing}
\pi_0\rmK_{\S' \times \Delta [n]}(\S \times \Delta[n], \S \times \delta \Delta[n]) \otimes \pi_0\rmK_{\S' \times \Delta [m]}(\S \times \Delta[m], \S \times \delta \Delta[m])\\
\ra \pi_0\rmK_{\S' \times \Delta[n] \times \Delta [m]}(\S \times \Delta[n] \times \Delta[m], \S \times (\delta \Delta[n] \times \Delta[m] \cup \Delta[n] \times \delta \Delta [m]))\\
\cong  \pi_0\rmK_{\S' \times \Delta [n+m]}(\S \times \Delta[n+m], \S \times \delta \Delta[n+m]). 
\end{split}
\end{multline}
\end{enumerate}
\end{remarks}
\vskip .2cm \noindent
{\bf Proof of Theorem ~\ref{mainthm}(i) and (ii).} 
Recall  $\S$ denotes a smooth algebraic stack with $\S'$ a closed algebraic sub-stack. 
\vskip .1cm
Then we need to show
{\it $\pi_0(\rmK (\S))$ is a pre-$\lambda$-ring, and that
each
$\pi_n(\rmK_{\S'}(\S))$ is a pre-$\lambda$-algebra over the pre-$\lambda$-ring $\pi_0(\rmK (\S))$}. Moreover, we need to show {\it the pre-$\lambda$-algebra structure is compatible with pull-backs}: that is, if $f: \tilde \S \ra \S$ is a map of smooth algebraic stacks and $\tilde \S' = \tilde S {\underset {\S} \times} \S'$,  
then the induced map $f^*:\pi_0(\rmK (\S)) \ra \pi_0(\rmK (\tilde \S))$ is a map of pre-$\lambda$-rings, and the induced map 
$f^*: \pi_n(\rmK_{\S'}(\S)) \ra \pi_n(\rmK_{\tilde S'}(\tilde \S))$ is a map of  pre-$\lambda$-algebras over $\pi_0(\rmK (\S))$, for each fixed $n \ge 0$.
In addition, we need to show {\it the $\lambda$-operations are homomorphisms on $\pi_n(\rmK_{\S'}(\S))$ for all $n>0$}.
\vskip .2cm
Observe first that the isomorphisms in the last Proposition are compatible with respect to the
$\lambda$-operations defined in ~\eqref{lambda.2}: this follows from the naturality of these operations with respect
to pull-backs. In fact, one may verify readily that one has the following homotopy commutative diagram of fibration
sequences:
\[\xymatrix{
{\rmK_{\S' \times \Delta [n]}(\S \times \Delta [n], \S \times \delta \Delta [n])} \ar@<1ex>[r] \ar@<-1ex>[d]^{\lambda ^{\it k}} & {\rmK_{\S' \times \Delta [n]}(\S \times \Delta [n], \S \times \Sigma )} \ar@<1ex>[d]^{\lambda ^{\it k}}  \\
{\rmK_{\S' \times \Delta [n]}(\S \times \Delta [n], \S \times \delta \Delta [n])} \ar@<1ex>[r] & 
{\rmK_{\S' \times \Delta [n]}(\S \times \Delta [n], \S \times \Sigma)} }\]
\[{\xymatrix{{}  \ar@<1ex>[r] & {\rmK_{\S' \times \Delta [n-1]}(\S \times \Delta [n-1], \S \times \delta \Delta [n-1] )} \ar@<1ex>[d]^{\lambda ^{\it k}}\\
{} \ar@<1ex>[r] & {\rmK_{\S' \times \Delta [n-1]}(\S \times \Delta [n], \S \times \delta \Delta [n-1])}.}}\]
Therefore, it follows that the $\lambda$-operations are compatible with the boundary maps of the corresponding
long exact sequence of homotopy groups. Using the isomorphism in Proposition ~\ref{reduct.pi0}, it suffices to show that
$\pi_0\rmK_{\S' \times \Delta [n]}(\S \times \Delta [n], \S \times \delta \Delta [n])$ is a pre-$\lambda$-ring without unit, and for
$\S'= \S$ and $n=0$, it is a pre-$\lambda$-ring. This may be done as for vector bundles: that is, the proof of this theorem follows along the same lines as the proof in \cite[section 8]{Gray}. Here are some details.
\vskip .2cm
Using the identification of $\pi_0(\rmK_{ \S' \times \Delta [n]}(\S \times \Delta [n], \S \times \delta \Delta [n]))$ as 
$\pi_0({\it w}\rmG({\it i}^*))$, the definition of ${\it w}\rmG({\it i}^*)$ as in ~\eqref{Gwf}
shows that a connected component of ${\it w}\rmG({\it i}^*)$ corresponds to a pair of perfect complexes $V$ and $W$ on $S \times \Delta[n]$, with supports contained in $\S' \times \Delta[n]$
together with a zig-zag path (as in ~\eqref{zig.zag.path}) 
$p$ joining the restriction $(i^*(V), i^*(W))$ to the base point, namely the pair $(0, 0)$ in $w\rmG(Perf_{{\it fl}, \S' \times \delta \Delta [n]}(\S\times \delta \Delta[n])$.
 This pair corresponds to the difference $[V]-[W]$ in the above Grothendieck group. We will begin with the special case where $W=0$.
\vskip .2cm
Viewing $\lambda^k$ as a map ${\it w}\rmG_{\it k}({\it i}^*) \ra {\it w}\rmG^{\it k}({\it i}^*) \simeq {\it w}\rmG({\it i}^*)$, we see that
\be \begin{equation}
\label{compat}
[\lambda^{\it k}(V)] = [\Lambda^{\it k}(V)], 
\end{equation} \ee
\vskip .2cm \noindent
where $\Lambda^{\it k}$ denotes the functor defined in ~\eqref{Lambda.1} and $[\Lambda^{\it k}(V)]$, $[\lambda^k(V)]$ denote
the corresponding classes 
in  $\pi_0({\it w}\rmG({\it i}^*)) =\pi_0(\rmK_{\S' \times \Delta [n]}(\S \times \Delta [n], \S \times \delta \Delta [n]))$. This follows from the observations as in \cite[section 8]{Gray}, but we will
provide the relevant details. The vertices of $sub_k{\it w}\rmG(i^*)$ correspond to pairs $(\rmV, \rm W)$ of perfect complexes on $\S \times \Delta[n]$ with
 supports contained in $\S' \times \Delta[n]$ (together with a zig-zag path {\it p} (as in the last paragraph) joining the restriction 
 $(i^*(V), i^*(W))$ to $(0,0)$), positioned at the vertices of a $k$-dimensional cube. Then the multi-simplicial map
 $\lambda ^k$ sends such a vertex $(V, W)$ to a sequence, each term of which is of the form 
 $\Lambda ^a(V) \otimes \Lambda^{b_1}(W) \otimes \cdots \otimes \Lambda ^{b_u}(W)$, for some choice of $a, b_1, \cdots, b_u \ge 0$, so that
 $a+b_1+ \cdots + b_u = k$. When $W=0$ as we have chosen, then this has only one nonzero term, namely $\Lambda^k(V)$.
\vskip .2cm 
Next let \xymatrix{{K'} \quad \ar@{>->} @<2pt>[r] & K} denote a cofibration in ${\rm {Perf}}_{fl, \S' \times \Delta[n]}(\S\times \Delta [n])$, so that together with the choice of paths
 joining their restrictions to $(0, 0)$, both
belong to ${\it w}\rmG({\it i}^*)$. Then one obtains the following formula in $\pi_0({\it w}\rmG({\it i}^*))$:
\be \begin{equation}
\label{lambda.on.sum}
[\lambda^m(K )] = \Sigma _{\rm k=0}^m [\lambda ^{\it k}({K'}) \otimes \lambda^{m-k}(K/K')] = \Sigma _{\rm k=0}^m [\lambda ^{\it k}({K'})]\cdotp [\lambda^{m-k}(K/K')]
\end{equation} \ee
\vskip .2cm \noindent
with the understanding that $\lambda^0(K') \cdotp \lambda^m(K/K') = \lambda^m(K/K')$ and $\lambda^m(K') \cdotp \lambda^0(K/K') = \lambda^m(K')$. In view of ~\eqref{compat} above, it suffices to prove this with $\lambda^{\it k}(K )$ replaced by $\Lambda^{\it k}(K )$. This holds 
by repeatedly applying Corollaries ~\ref{additivity.0}, ~\ref{additivity.cor} and (E5) by taking, $m=k+l+1$,  first with $k=0$, $W_1=K'$, $W_2=K = U_j$, $j=1, ..., l$, which gives
\[[\Lambda^m({\overset {m} {\overbrace {K, K, \cdots, K}}})] = [\Lambda^m({\overset {m} {\overbrace {K', K, \cdots, K}}})] + [\Lambda^m({\overset {m} {\overbrace {K/K', K/K', \cdots, K/K'}}})].\]
Then
with $k=1$, $V_1 =K'$,  $W_1=K'$, $W_2=K  =U_j$, $j=1, ..., l-1$, enables us to obtain
\[[\Lambda^m({\overset {m} {\overbrace {K', K, \cdots, K}}})] = [\Lambda^1(K') \otimes \Lambda^{m-1}({\overset {m-1} {\overbrace {K/K', K/K', \cdots, K/K'}}})]+ [\Lambda^m({\overset {m} {\overbrace{K', K', K, \cdots, K}}})], \cdots ,\]
ending with
\[[\Lambda^m({\overset {m-1} {\overbrace {K', \cdots, K'}}}, K)] = [\Lambda^{m-1}({\overset {m-1} {\overbrace {K', \cdots, K'}}})\otimes \Lambda^1(K/K')] + [\Lambda^m({\overset {m} {\overbrace {K', \cdots, K'}}})].\]
 Moreover, in view of ~\eqref{compat}, one observes readily that if $n=0$,  $\Lambda ^0(K) = {\mathcal O}_{\S}$ and in general, $\Lambda ^1(K) = K$, $ K \in \pi_0\rmK_{\S' \times \Delta[n]}(\S \times \Delta[n], \S \times \delta \Delta [n])$.
\vskip .2cm
\subsubsection{}
\label{neg.classes}
At this point we make the following important observation. Given a  perfect complex $K$ on $S\times \Delta[n]$ and acyclic on 
$(\S-\S') \times \Delta[n]$, and so that the pair $(K, 0)$ denotes a class in $\pi_0({\it w}\rmG({\it i}^*))$, the canonical construction of its cone (that is, the
mapping cone of the identity map $K \ra K$) along with Theorem ~\ref{add.thm} (see also Corollary ~\ref{additivity.cor})  shows that
the class of $K[1]$ denotes the additive inverse of the class of $K$ in the above Grothendieck group. (Given a perfect complex
$K$ in ${\rm {Perf}}_{fl, \S' \times \Delta[n]}(\S\times \Delta [n])$, so that $(K, 0)$ represents a class in
$\pi_0({\it w}\rmG({\it i}^*))$, we will
let $[K]$ denote its class in the above Grothendieck group.)
Therefore, given two perfect complexes $K', K$ in ${\rm {Perf}}_{fl, \S' \times \Delta[n]}(\S\times \Delta [n])$, so that $(K, 0)$ and $(0, K')$ represent classes in
$\pi_0({\it w}\rmG({\it i}^*))$, the class $[K]-[K']$ in the above Grothendieck group is represented
by the class of the perfect complex $K\oplus K'[1]$. It follows that the identity in ~\eqref{lambda.on.sum} 
suffices to prove 
the identity 
\[\lambda ^n(r+s) = \Sigma_{i=0}^n \lambda ^i(r) \cdotp \lambda^{n-i}(s)\]
holds for all $r, s$ in the  group $\pi_0\rmK_{\S' \times \Delta [n]}(\S \times \Delta [n], \S \times \delta \Delta [n]) \cong \pi_n\rmK_{\S'}(\S)$, with the
understanding that $\lambda^0(r) \cdotp \lambda^n(s) = \lambda^n(s)$ and $\lambda^n(r) \cdotp \lambda^0(s) = \lambda^n(r)$.
These observations prove that there is
the structure of a pre-$\lambda$-ring without unit on each $\pi_n(\rmK_{\S'}(\S)) \cong \pi_0(\rmK_{\S' \times \Delta[n]}(\S \times \Delta [n], \S \times \delta \Delta [n]))$
and the structure of a pre-$\lambda$-ring on $\pi_0(\rmK (\S))$ (that is, on $\pi_n(\rmK_{\S'}(\S)) \cong \pi_0(\rmK_{\S' \times \Delta[n]}(\S \times \Delta [n], \S \times \delta \Delta [n]))$, when $n=0$ and $\S' = \S$).
\vskip .2cm
 In view of these observations, one may define a pre-$\lambda$ algebra structure on $\pi_0(\rmK (\S)) \oplus \pi_n(\rmK_{\S'}(\S))$ by defining
$\lambda^m$ on $\pi_0(\rmK (\S)) \oplus \pi_n(\rmK_{\S'}(\S))$ by $ \lambda ^m(r, s) =
(\lambda ^m(r), \Sigma _{i=0}^{m-1}\lambda ^i(r)  \cdotp \lambda^{m-i}(s))$. (See Lemma ~\ref{prelambda.alg} below.) Observe that
each $\pi_n(\rmK_{S'}(\S))$ has the structure of a module over $\pi_0(\rmK (\S))$ in the obvious manner using the tensor product of perfect complexes.
\vskip .2cm 
 The naturality with
respect to pull-back is clear from the construction. It may be also important to point out the following: the product
structure on each $\pi_n\rmK_{\S'}(\S)$ is trivial for all $n>0$. This is because this product structure is defined making use
of the pull-back by the diagonal map $\Delta_1: \S\times \Delta [n] \ra \S \times \Delta [n] \times \S \times \Delta [n]$, and
involves the co-H-space structure on $S^n \simeq \Delta[n]/(\delta \Delta[n])$: see \cite[Lemme 5.2]{Kr-2}. (This is distinct from
the product $\pi_n(\rmK_{\S'}(\S)) \otimes \pi_m(\rmK_{\S'}(\S)) \ra \pi_{n+m}(\rmK_{S'}(\S))$, which makes use of pull-back by the
diagonal, $ \Delta_2: \S \times \Delta [n] \times \Delta[m] \ra (\S \times \Delta [n]) \times (\S \times \Delta [m])$.) In view of
this observation, the $\lambda$-operations are all {\it homomorphisms} on $\pi_n(\rmK_{\S'}(\S))$, for all $n>0$.
\vskip .1cm
In view of the following lemma, these prove the first two statements of Theorem ~\ref{mainthm}.
\begin{lemma}
\label{prelambda.alg}
 Let $\rmR$ denote a pre-lambda ring and $\rmS$ an $\rmR$-module, so that it is also provided with the structure of
 a pre-lambda ring without unit. Then $\rmR \oplus \rmS$ has the structure of a pre-lambda ring, where
 \vskip .2cm
\be \begin{align}
 \label{lambda.def}
(r, s) + (r', t) &=(r+r', s+t), \\
(r, s) \circ (r', t) &= (r\cdotp r', r\cdotp t+r'\cdotp s+s \cdotp t), \mbox{ and} \notag \\
\lambda ^n(r,s) &= (\lambda_{\rmR}^n(r), \Sigma_{i=0}^{n-1}\lambda_{\rmR}^i(r).\lambda_{\rmS}^{n-i}(s)), \notag
\end{align} \ee
and where $\lambda^i_{\rmR}$ ($\lambda^j_{\rmS}$) denote the pre-lambda operations of $\rmR$ ($\rmS$, \res).
\end{lemma}
\begin{proof} We define $\lambda ^0(r, s) =1$, where $1$ denotes the multiplicative unit in $\rmR$. We also let
$\lambda^i(r, 0) = \lambda_{\rmR}^i(r)$ and $\lambda^j(0, s) = \lambda_{\rmS}^j(s)$, for all $r \in \rmR$, $s \in \rmS$ and $i\ge 0$, $j >0$. 
We also let $\lambda ^i(r, 0)\cdotp \lambda^0(0,s) = \lambda ^i(r, 0)$ for all $i \ge 0$ and $r \in \rmR, s \in \rmS$.
Next we let
$\lambda ^1(r,s ) = (\lambda_{\rmR}^1(r), \lambda_{\rmS}^1(s))$. In general, we define $\lambda ^n$ on $\rmR \oplus \rmS$, by
\[\lambda^n(r, s) = (\lambda_{\rmR}^n(r), \lambda_{\rmR}^0(r)\cdotp \lambda_{\rmS}^n(s)+ \cdots + \lambda_{\rmR}^{n-1}(r)\cdotp \lambda_{\rmS}^1(s)) = (\lambda^n_{\rmR}(r), \Sigma_{i=0}^{n-1}\lambda_{\rmR}^i(r). \lambda_{\rmS}^{n-i}(s)).\]
In view of the above definitions, clearly we may identify the right-hand-side above with 
$\Sigma_{i=0}^n \lambda^i(r)\cdotp \lambda^{n-i}(s).$ One may also verify that if $r, r' \in \rmR$ and $s, s' \in \rmS$,
then $\lambda^n(r+r') = \Sigma_{i=0}^n\lambda^i(r)\cdotp \lambda^{n-i}(r')$ and $\lambda^n(s+s') = \Sigma_{i=0}^n \lambda^i(s)\cdotp \lambda^{n-i}(s')$ since both $\rmR$ and $\rmS$ are assumed to be pre-lambda rings.
In view of these observations, it suffices to check that 
\be \begin{equation}
\label{eq.lemma5.5}
\lambda^n((r, s) +(r', s')) = \Sigma_{i=0}^n \lambda^i(r,s)\cdotp \lambda^{n-i}(r',s').
\end{equation} \ee
In fact the term on the left side is given by:
\be \begin{align}
 \label{lemma5.5.1}
 \lambda^n(r+r', s+s') &= (\lambda^n(r+r'), \lambda^{n-1}(r+r')\cdotp \lambda^1(s+s'), \lambda^{n-2}(r+r')\cdotp \lambda^2(s+s'), \cdots, \\
                       & \mbox{\, \, \, }   \lambda^1(r+r')\cdotp  \lambda^{n-1}(s+s'), \lambda^n(s+s')) \notag\\
                       &=(\Sigma_{i=0}^n \lambda^i(r)\cdotp \lambda^{n-i}(r'), \Sigma_{j=0}^{n-1}\lambda^j(r)\cdotp \lambda^{n-1-j}(r')[\lambda^1(s) + \lambda^1(s')], \notag \\
                       & \mbox{\, \, \, } \Sigma_{j=0}^{n-2}\lambda^j(r)\cdotp \lambda^{n-2-j}(r')[\lambda^2(s) + \lambda^1(s)\cdotp \lambda^1(s')+ \lambda^2(s')], \cdots, \notag \\
                       & \mbox{\, \, \, }[\lambda^1(r)+ \lambda^1(r')]\cdotp(\lambda^{n-1}(s')+ \lambda^1(s)\lambda^{n-2}(s')+ \cdots + \lambda^{n-1}(s)), \notag \\
                       & \mbox{\, \, \,} \lambda^n(s') + \lambda^1(s)\cdotp\lambda^{n-1}(s') + \cdots + \lambda^{n-1}(s)\cdotp \lambda^1(s') + \lambda^n(s)). \notag
\end{align} \ee
The term on the right side of ~\eqref{eq.lemma5.5} is given by
\be \begin{align}
 \label{lemma5.5.2}
 \Sigma_{i=0}^n \lambda^i(r,s)\cdotp \lambda^{n-i}(r',s') &= \Sigma_{i=0}^n(\lambda^i(r), \lambda^{i-1}(r)\cdotp \lambda ^1(s) + \cdots + \lambda^1(r).\lambda^{i-1}(s) + \lambda^i(s))*\\
                                                     & \mbox{\, \, \, } ( \lambda^{n-i}(r') , \lambda^{n-i-1}(r')\cdotp \lambda^1(s')\cdots + \lambda^1(r')\cdotp\lambda^{n-i-1}(s')+ \lambda^{n-i}(s')). \notag 
\end{align} \ee                                                 
Now it is straightforward to check that we obtain equality in ~\eqref{eq.lemma5.5}.
(Moreover, in case the multiplication on $\rmS$ is trivial, $\lambda^n(s+s') =\lambda^n(s)+ \lambda^n(s')$ as well.)
 
\end{proof}

\begin{remark}
 In dealing with the K-theory of exact categories, there is no analogue of the suspension functor $K \mapsto K[1]$ (used in ~\ref{neg.classes} above), and
 as a result it takes much more effort to deduce that the $\lambda$-operations defined as in ~\eqref{lambda.1} and ~\eqref{lambda.2}
 define a pre-$\lambda$-ring structure even on the Grothendieck groups: see \cite[section 8]{Gray}.
\end{remark}
\section{\bf The Lambda-ring structure on higher K-theory:  Proof of Theorem ~\ref{mainthm}(iii)}
In this section, we consider statement (iii)  in Theorem ~\ref{mainthm}. 
Recall that this is the following statement: {\it in case every coherent sheaf on the smooth stack $\S$ is the quotient of a vector bundle, then  each $\pi_n(\rmK_{\S'}(\S))$ is a  $\lambda$-algebra over $\pi_0(\rmK (\S))$ in the sense of Definition ~\ref{def1.1}.}
\vskip .1cm
The additional assumption that every coherent sheaf is the quotient of a
vector bundle first enables one to restrict to strictly perfect complexes. This observation, together with an
adaptation of some arguments of Gillet, Soul\'e and Deligne (see \cite[section 4]{GS}) enable us to show that the $\lambda$ operations we define satisfy all the expected relations, so that we obtain the structure of a 
$\lambda$-ring on the higher K-groups of all smooth stacks satisfying this hypothesis.
\vskip .1cm
We begin by recalling  the framework from \cite{Kr-1} and \cite{GS}. (See also \cite[pp. 22-27]{SABK} for a somewhat simplified account of this, as well as \cite{Ser}.) Let $\rmH ={\mathbb Z}[\rm\rmM_N \times \rm\rmM_N]$ denote the bialgebra
${\mathbb Z}[\rm\rmM_N \times \rm\rmM_N]$ of the multiplicative monoid of pairs of $\rmN \times \rmN$ matrices for some fixed integer $N \ge 1$, that is, it is the algebra of polynomials ${\mathbb Z}[X_{11}, X_{12}, \cdots , X_{\rm NN}; Y_{11}, Y_{12}, \cdots Y_{\rm NN}]$ with the 
co-product $\mu: \rmH \ra \rmH \otimes \rmH$ satisfying $\mu(X_{ij}) = \Sigma _{k=1}^{\rm N} X_{ik} \otimes X_{kj}$ and $\mu(Y_{ij}) = \Sigma_{k=1}^{\rm N} Y_{ik} \otimes Y_{kj}$. Let $\rmP_{\mathbb Z}(\rm\rmM_N \times \rm\rmM_N)$ denote the exact category of left-${\mathbb Z}[\rmM_N \times \rmM_N]$ comodules that are free and finitely generated over ${\mathbb Z}$: an element of this category
is called {\it a representation} of $\rmM_N \times \rmM_N$. We let $\rmR_{\mathbb Z}(\rm\rmM_N \times \rm\rmM_N)$ denote the Grothendieck group of $\rmP_{\mathbb Z}(\rm\rmM_N \times \rm\rmM_N)$, which is a ring via the tensor product of comodules.
\vskip .2cm
Let $p_i$, $i=1, 2$
denote the representation of $\rmM_N \times \rmM_N$ corresponding to the projection to the $i$-th factor. It is shown in \cite[Theorem 4.2]{GS} (see also \cite[Proposition 4.3]{Kr-1}) that the following are true:
\begin{proposition}
\label{Gillet-Soule-prop1}
 The ring $\rmR_{\mathbb Z}(\rmM_N \times \rmM_N)$ is isomorphic to the polynomial ring 
\vskip .2cm
${\mathbb Z}[\lambda ^1(p_1), \cdots \lambda ^N(p_1); \lambda ^1(p_2), \cdots, \lambda ^N(p_2)]$. 
\vskip .1cm \noindent
Exterior powers make it a $\lambda$-ring.
\end{proposition}
\vskip .2cm
It is shown in \cite[4.3]{GS} that the center of the monoid $\rm\rmM_N \times \rm\rmM_N$ is $\rmM_1 \times \rmM_1$ imbedded diagonally,  and that the category of representations of $\rmM_1 \times \rmM_1$ is equivalent to the category of positively bi-graded ${\mathbb Z}$-modules. For any representation $\rmE$ of $\rm\rmM_N \times \rm\rmM_N$, the decomposition $\rmE = {\underset {p, q} \oplus}\rmE^{p,q}$ over the center of $\rm\rmM_N \times \rmM_N$ is stable under the action of $\rmM_N \times \rm\rmM_N$. $\rmE$ has {\it degree at most d} if $\rmE^{p,q}=0$ unless $p+q \le d$.
Let $\rmR_{\mathbb Z}(\rm\rmM_N \times \rm\rmM_N)^d$ denote the Grothendieck group of the category of representations of $\rm\rmM_N \times \rm\rmM_N$ of degree at most $d$.
The following is also known (see \cite[Lemma 4.3]{GS}):
\begin{proposition}
\label{Gillet-Soule-prop2}
The group $\rmR_{\mathbb Z}(\rm\rmM_N \times M\rm_N)^d$ maps injectively into $\rmR_{\mathbb Z}(\rmM_N \times \rmM_N)$ and the image of this map consists of the elements $\rmR(\lambda^1(p_1), \cdots, \lambda ^N(p_1); \lambda ^1(p_2), \cdots, \lambda^N(p_2))$, where $\rmR$ runs over all polynomials of weight at most $d$.
\end{proposition}
\vskip .2cm
\subsection{\bf The functor $\rmT_{\rmE}$: see \cite[4.4, 4.5 and Lemma 4.5]{GS}}\footnote{It needs to be pointed out that this functor does not extend to one on pairs of all perfect complexes, but only on to pairs of strictly perfect complexes. This is the reason the results of this section hold only under the strong assumption that every coherent sheaf is a quotient of a vector bundle.}
\label{TE}
Next, for any representation $\rmE$ of $\rmM_N \times \rmM_N$, a scheme $\rmX$ and
two locally free coherent sheaves $\rmP$, $\rmQ$ of rank at most $\rmN$ on $\rmX$, a vector bundle
${\rm T}_{\rm E}(\rmP, \rmQ)$ on $\rmX$ is defined. It is shown that for fixed $\rmP$ and $\rmQ$, the functor $\rmE \mapsto {\rm T}_{\rmE}(\rmP, \rmQ)$ 
has the following properties:
\begin{enumerate}[\rm(i)]
\item if the representation $\rmE$ has degree at most
$d$, the functor $\rmE \mapsto {\rm T}_E(\rmP, \rmQ)$ has degree at most $d$ (that is,  the cross-effect functor $\rmE \mapsto {\rm \rmT_{\rmE}}(\rmP, \rmQ)_s=0$ for $s>d$), 
\item the above functor is exact in $\rmE$ and it commutes with direct sum, tensor product and exterior powers in $\rmE$ for a fixed $\rmP$ and $\rmQ$, 
\item $\rmT_{\rm p_1}(\rmP, \rmQ) = \rmP$, $\rmT_{\rm p_2}(\rmP, \rmQ) = \rmQ$, and ${\rm \rmT_{\rmE}}(0, 0) =0$,
\item commutes with base-change of the scheme $\rmX$, that is, the following holds: if $\rmp: \rmY \ra \rmX$ is a map of schemes, 
$\rmT_{\rmE}(p^*(\rmP), p^*(\rmQ))$ is canonically isomorphic to $\rmp^*(\rmT_{\rmE}(\rmP, \rmQ))$, and
\item ${\rm T}_{\rmE}(\rmP, \rmQ)$ is functorial in $\rmP$ and $\rmQ$ for a fixed $\rmE$. 
\end{enumerate}
\begin{proposition} 
\label{TE.st}
The functor ${\rm T}_{\rm E}$ extends to algebraic stacks with the same properties.
Given two bounded complexes of vector bundles $\rmP$, $\rmQ$ on an algebraic stack $\S$
 the functor $\rmE \mapsto \rmT_{\rmE}(\rmP, \rmQ)$ defines a bounded complex of vector bundles on the stack $\S$. The functor $\rmT_{\rm E}$ preserves quasi-isomorphisms in either argument.
 If $\rmP$ and $\rmQ$ have cartesian cohomology sheaves, then
 $\rmT_{\rmE}(\rmP, \rmQ)$ also has cartesian cohomology sheaves.
\end{proposition}
\begin{proof}
One may first consider bounded complexes of vector bundles on a scheme. Then one may show that the
functor $\rmT_{\rmE}(\quad, \quad)$ preserves degree-wise split short-exact sequences of bounded complexes of vector bundles in both arguments and that if $\rmP$ or $\rmQ$ is acyclic, then $\rmT_{\rmE}(\rmP, \rmQ)$ is also acyclic. The latter follows
by working locally on a given scheme $\rmX$, where we may assume that there is a null chain-homotopy for both $\rmP$ or $\rmQ$ and by using the observation that the functor preserves chain homotopies. Since the
functor $\rmT_{\rmE}(\quad, \quad)$ preserves degree-wise split short exact sequences in each argument, it follows that it preserves 
quasi-isomorphisms in either argument. 
\vskip .1cm
Next we show this functor extends to algebraic stacks.
 Let $x: \rmX \ra \S$ denote an atlas for the stack $\S$ with $\rmX$ a scheme and let $\rmP$, $\rmQ$ denote two bounded complexes of vector bundles on the stack $\S$. Then $\rmP_0 = \x^*(\rmP)$ and $\rmQ_0 = \x^*(\rmQ)$ define
 two bounded complexes of vector bundles on the scheme $\rmX$. The property that the functor $\rmE \mapsto \rmT_{\rmE}(\rmP, \rmQ)$ commutes
 with respect to base-change on schemes, shows that
 \[{\rm pr}_1^*\rmT_{\rmE}(\rmP_0, \rmQ_0) \cong \rmT_{\rmE}(pr_1^*\x^*(\rmP), pr_1^*\x^*(\rmQ)) = \rmT_{\rmE}(pr_2^*\x^*(\rmP), pr_2^*\x^*(\rmQ)) \cong pr_2^*\rmT_{\rmE}(\rmP, \rmQ),\]
 where ${\rm pr}_i: \rmX\times_{\S} \rmX \ra \rmX$, $i=1,2$ are the two projections. (We skip the verification that the
  above isomorphism satisfies a co-cycle condition on further pull-back to $\rmX \times_{\S} \rmX \times_{\S} \rmX$.) Therefore, it follows readily that
 $\rmT_{\rmE}(\rmP, \rmQ)$ defines a bounded complex of vector bundles on the stack $\S$.  
 The properties (i) through (v) may be checked by pulling back to 
 \[\rm  X \times_{\S} \rm X \substack{\overset{\rm pr_1}\rightarrow \\ \underset{\rm pr_2} \rightarrow} \rm X,\]
 where $x:\rmX \ra \S$ is an atlas for the stack.
 \vskip .1cm
 Next we consider the last statement in the Proposition. Assume $\rmP$ and $\rmQ$ have cartesian cohomology sheaves.
 As in  the proof of Proposition ~\ref{cart.coh.sheaves}, it suffices to show that for a smooth map $f: \rmU \ra \rmV$ in $\S_{lis-et}$,
 one obtains isomorphisms $f^*{\mathcal H}^i(\rmT_{\rmE}(\rmP, \rmQ)) \cong {\mathcal H}^i(\rmT_{\rmE}({\it f}^*(\rmP), {\it f}^*(\rmQ))$, for all $i$.
 The base-change property shows $f^*\rmT_{\rmE}(\rmP, \rmQ) \cong \rmT_{\rmE}({\it f}^*(\rmP), {\it f}^*(\rmQ))$. Since ${\it f}$ is a smooth map, $f^*$ commutes with 
 taking cohomology sheaves, as it 
 is an exact functor.  This completes the proof.
\end{proof}
\vskip .2cm \indent \indent
Throughout the following discussion we will fix a smooth algebraic stack $\S$ with $\S'$ a closed
sub-stack. {\it We will assume throughout the rest of the discussion that every coherent sheaf on $\S$ is the quotient of a vector bundle.} Let $Vect(\S)$ ($Vect_N(\S)$) denote the category of vector bundles on the stack $\S$ (vector bundles on the stack $\S$ with
rank $\le N$, \res). Let $t_{\it k}$ denote the {\it naive truncation} functor that sends a simplicial object to the corresponding truncated simplicial object, truncated in degree $\le k$. If ${\bold E}$ is an exact category and $\rmC$ denotes a chain complex with differentials of degree $-1$ in ${\bold E}$ and
trivial in negative degrees, we will also use $t_{\it k}(\rmC)$ to denote the corresponding truncated chain complex, truncated to degrees $\le k$.  
We let $Simp({\bold E})$ ($Simp_{\it k}({\bold E})$) denote the category of all simplicial objects in ${\bold E}$  (simplicial objects
truncated to degrees $\le k$, \res). Similarly, ${\rm C}({\bold E})$ (${\rm C}_{\it k}({\bold E})$) will denote the category of complexes in ${\bold E}$ that
are trivial in negative degrees and with differentials of degree $-1$ ( the category of complexes in ${\bold E}$ that
are trivial in negative degrees and in degrees $>k$ and with differentials of degree $-1$, \res). We let $e_{\it k}: {\rm C}_{\it k}({\bold E}) \ra {\rm C}({\bold E})$
denote the obvious inclusion functor. 
\vskip .1cm
Next let $\rm\rmN(1)$ denote the normalized chain complex of the standard $1$-simplex: $\rmN(1)_n=0$ if $n>1$, $\rmN(1)_1 = {\mathbb Z}[e]$, $\rmN(1)_0 = {\mathbb Z}[e_0] \oplus {\mathbb Z}[e_1]$ with $\delta (e) = e_0 -e_1$. Let $\rmK (1) = {\rm DN}_{\it h}(\rmN(1))$, where ${\rm DN}_{\it h}$ denotes the de-normalizing functor
as in appendix B: this is a simplicial abelian group. Given 
an object $\rmS_{\bullet} \in {\it Simp(Vect(\S))}$, $\rmK (1) \otimes S_{\bullet}$ will denote the obvious simplicial object: $(\rmK (1) \otimes S_{\bullet})_n
= \rmK (1)_n \otimes S_n = {\underset {{k_n} \in \rmK (1)_n} \oplus } S_n$ and with the obvious structure maps. Observe that $\rmK (1) \otimes S_{\bullet} \in {\it Simp(Vect(\S))}$. 
\vskip .2cm
 Let $\rmA$, $\rmB$ denote two strictly perfect complexes on $\S \times \Delta[n]$ and let $\rmC$, $\rmD $ denote two strictly perfect complexes on $\S \times \Delta[n]$ acyclic on $(\S-\S') \times \Delta[n]$.
Choose  $m$ so that $\rmA^n =\rmB^n=0= \rmC^n =\rmD^n$ if $n>m$. We will first apply the shift $[m]$ so that $ \rmA[m]^n =\rmB[m]^n= \rmC[m]^n=\rmD[m]^n=0$ for all $n>0$.
Therefore we may view   $\rmA[m]$, $\rmB[m]$, $\rmC[m]$ and $\rmD[m]$ as complexes in non-negative degrees with differentials of degree $-1$. We denote these by
 $\rmA'$, $\rmB'$, $\rmC'$ and $\rmD'$ \res. Choose the integer $k$ so that $k>md$. Choose the integer ${ N}$ so that all the components of $t_{\it k}(\rmK (1)\otimes ({\rm DN}(\rmA')))$, $t_{\it k}(\rmK (1) \otimes {\rm DN} (\rmB')) $ lie in $Simp_{\it k}(Vect_N(\S))$ 
and all the components of $t_{\it k}(\rmK (1) \otimes ({\rm DN}(\rmC')))$, $t_{\it k}(\rmK (1) \otimes ({\rm DN}(\rmD')))$ lie
in $Simp_{\it k}(Vect_N(\S \times \Delta[n], (\S-\S') \times \Delta[n])$: the latter denotes the full subcategory of 
$Simp_{\it k}(Vect_N(\S \times \Delta[n]))$ with supports contained in $\S' \times \Delta [n]$. For any representation $\rmE$ of $\rmM_N \times \rmM_N$ of degree at most $d$,  we define
\be \begin{equation}
\label{Sk.1}
\rmS_{\rmE}({\it k}; \rmA', \rmB') = {\it e}_{\it k} \rmN_{\it k}\rmT_{\rmE}{\it t}_{\it k}\Delta ({\rm DN}(\rmA'), {\rm DN}(\rmB')), \mbox{ and }
\end{equation} \ee
\vskip .1cm
\be \begin{equation}
\label{Sk.2}
 {\rm S}_{\rm E}({\it k}; \rmC', \rmD') = {\it e}_{\it k} \rmN_{\it k}\rmT_{\rmE}{\it t}_{\it k}\Delta (({\rm DN}(\rmC'), {\rm DN}(\rmD')), 
\end{equation} \ee
\vskip .2cm \noindent
where $\Delta$ denotes the diagonal of the bisimplicial object appearing there, and $\rmN_{\it k}$ denotes the normalization functor for truncated simplicial objects in an abelian category: this is defined as the normalization functor for simplicial objects in appendix B. 
\vskip .2cm
Throughout the following discussion we will let $i: \S \times \delta \Delta[n] \ra \S \times \Delta [n]$ denote the obvious closed immersion.
Let $\S'$ denote a closed algebraic substack of $\S$. 
Then we let 
\[i^*:{\it w}\rmG_{\bullet}StPerf_{{\it fl}, \S' \times \Delta [n]}(\S \times \Delta [n]) \ra {\it w}\rmG_{\bullet}StPerf_{{\it fl},\S' \times \Delta [n]}(\S \times \delta \Delta [n])\]
denote the corresponding pull-back functor, and let ${\it w}\rmG(i^*)$ denote its homotopy fiber defined as in
~\eqref{Gwf} or ~\eqref{Gwf.1}.
\vskip .2cm
\begin{proposition} 
\label{Gillet-Soule-prop3}
(See also \cite[Lemma 4.8]{GS}.) Let $\rmC$, $\rmD $ denote two strictly perfect complexes on $\S \times \Delta[n]$ acyclic on $(\S-\S') \times \Delta[n]$, and provided with an explicit zig-zag path $p$ as in ~\eqref{zig.zag.path} running from  
the restriction of the pair $(\rmC, \rmD)$ to $\S \times \delta \Delta [n]$, to the base point $(0, 0)$
in ${\it w}\rmG(i^*)$. Let $d \ge 1$ be an integer. Let $\rmA$, $\rmB$ denote two strictly perfect complexes on $\S $. Then there exists an integer $N\ge 1$ and  homomorphisms
\vskip .2cm 
$\alpha: R_{\mathbb Z}(\rmM_N \times \rmM_N) ^d \ra \pi_0(\rmK (\S))$ and
\vskip .2cm 
$ \beta: R_{\mathbb Z}(\rmM_N \times \rmM_N) ^d \ra \pi_0(\rmK_{\S' \times \Delta [n]}(\S \times \Delta[n], \S \times \delta \Delta [n]))$,
\vskip .2cm \noindent
which  preserve the additive, 
multiplicative and pre-$\lambda$-ring structures. Moreover, the following holds:
\vskip .2cm \noindent
(i) $[\rmA] = \alpha (p_1)$, $[\rmB] = \alpha(p_2)$, $[\rmC] = \beta (p_1)$, $[\rmD] = \beta (p_2)$
\vskip .2cm \noindent
(ii) $\alpha (xy) = \alpha (x) \alpha (y) $ and $\beta(xy) = \beta(x) \beta (y)$ if $x, y$ and $xy$ are in $\rmR_{\mathbb Z}(\rmM_N \times \rmM_N)^d$ (where the product on $\pi_0(\rmK (\S))$ ($\pi_0(\rmK_{\S' \times \Delta [n]}(\S \times \Delta[n], \S \times \delta \Delta [n]))$ is given by the tensor products of
perfect complexes.
\vskip .2cm \noindent
(iii) $\alpha (\lambda^{\it k}(x)) = \lambda^{\it k}(\alpha(x))$ and $\beta (\lambda ^{\it k}(x)) = \lambda^{\it k}(\beta (x))$ if $x$ and $\lambda^{\it k}(x)$ are in $\rmR_{\mathbb Z}(\rmM_N \times \rmM_N)^d$. 
\end{proposition}
\begin{proof} In \cite[Lemma 4.8]{GS} they consider a similar result for complexes $\rmC$ and $\rmD$ that are acyclic {\it off} of a closed sub-scheme of the given
scheme: therefore the above result does not follow by simply extending their result to stacks. Instead one needs to argue as follows.
Choose  $m$ so that ${\rmC}^n =\rmD^n=\rmA^n =\rmB^n=0$ if $n>m$, and the integer $k$ so that $k>md$. We will first apply the shift $[m]$ so that $\rmC[m]^n= \rmD[m]^n= \rmA[m]^n =\rmB[m]^n=0$ for all $n>0$.
Therefore we may view  $\rmC[m]$, $\rmD[m]$, $\rmA[m]$ and $\rmB[m]$ as complexes in non-negative degrees with differentials of degree $-1$. We denote these by
$\rmC'$, $\rmD'$, $\rmA'$ and $\rmB'$ \res. Choose the integer ${\rm N}$, and 
for any representation 
$\rmE$ of $\rmM_N \times \rmM_N$ of degree at most $d$,  we define the functors $\rmE \mapsto \rmS_{\rmE}({\it k}, \rmA', \rmB')$ and $\rmE \mapsto \rmS_{\rmE}({\it k}, \rmC', \rmD')$ as in
~\eqref{Sk.1} and ~\eqref{Sk.2}.
\vskip .2cm
That $ \rmE \mapsto \rmT_{\rmE}({\it k}; \rmA', \rmB')$ defines a map $\rmR_{\mathbb Z}(\rmM_N \times \rmM_N)^d \ra Simp_{\it k}(Vect_N(\S))$  is clear from the definition.
The property that the functor $\rmT_{\rmE}$ commutes with base-change shows there is a natural isomorphism
$ i^* \circ {\rm S}_{\rmE}({\it k};  \rmC',  \rmD') \cong {\rm S}_{\rmE}({\it k}; i^*(\rmC'), i^*( \rmD'))$. 
 Next assume that $\rmC'$ and $\rmD'$ are {\it acyclic} on $\S' \times \Delta [n]$.
Since they are both
complexes of locally free coherent sheaves, locally on the stack $\S' \times  \Delta[n]$, one may find a contracting homotopy for the restriction of 
$\rmC'$ and $\rmD'$ to $\S' \times \Delta[n]$. Therefore, again the same argument  applied to 
a presentation of the stack $\S$ (by an affine scheme) proves that ${\rm S}_{\rmE}(k;  \rmC',  \rmD')$ is {\it acyclic} on restriction $\S' \times \Delta[n]$. 
(In more detail: let $x:X \ra \S$ be a presentation of the stack $\S$ with $\rmX$ affine, $y: Y \ra \S'$ be a presentation of $\S'$. Now apply \cite[ Lemma 3.5(iii)]{GS} to $\rmX-Y$.) 
It follows that ${\rm S}_{\rmE}({\it k};  \rmC', \rmD')$ is a strictly perfect complex on $ \S \times \Delta[n]$ so that it is acyclic on 
$\S' \times  \Delta [n]$. 
It follows that $\rmE \mapsto {\rm S}_{\rmE}({\it k}; \rmC', \rmD')$ defines a map $\rmR_{\mathbb Z}(\rmM_N \times \rmM_N)^d \ra Simp_{\it k}(Vect_N(\S \times \Delta [n], (\S-\S') \times \Delta[n])$.
\vskip .2cm 
Recall that $ \pi_0(\rmK (\S))$ has been
proven to be a pre-lambda ring, and $\pi_0(\rmK_{\S' \times \Delta [n]}(\S \times \Delta[n], \S \times \delta \Delta [n]))$ has been
proven to be a pre-lambda ring without unit by the first two statements in Theorem ~\ref{mainthm}. At this point we recall from
Proposition ~\ref{Gillet-Soule-prop1} that the ring $\rmR_{\mathbb Z}(\rmM_N \times \rmM_N)$ is isomorphic to the polynomial ring 
${\mathbb Z}[\lambda ^1(p_1), \cdots \lambda ^N(p_1); \lambda ^1(p_2), \cdots, \lambda ^N(p_2)]$. One may also recall from Proposition 
 ~\ref{Gillet-Soule-prop2} that the group $\rmR_{\mathbb Z}(\rm\rmM_N \times M\rm_N)^d$ maps injectively into 
 $\rmR_{\mathbb Z}(\rmM_N \times \rmM_N)$ and the image of this map consists of the elements 
\[\rmR(\lambda^1(p_1), \cdots, \lambda ^N(p_1); \lambda ^1(p_2), \cdots, \lambda^N(p_2)),\]
where $\rmR$ runs over all polynomials of weight at most $d$.
Therefore, it should be clear now that  
the maps $\alpha $  and $\beta$ {\it are completely 
determined by their values on the representations $p_1$ and $p_2$}, that is, assuming both $\alpha$ and $\beta$ commute with $\lambda$-operations.
\vskip .1cm
Therefore, for a representation $\rmE$ of $\rmM_{\rmN} \times \rmM_{\rmN}$ of degree at most $d$, 
we define: 
\be \begin{align}
 \label{alpha.beta.def}
\alpha (\rmE) \in\pi_0(\rmK (\S)) &\mbox{ to be the class of } \rmS_{\rmE}({\it k}; \rmA', \rmB')[-m] \mbox{ and }\\
\beta(\rmE) \in \pi_0(\rmK_{\S' \times \Delta [n]}(\S \times \Delta [n], \S \times \delta \Delta [n])) &\mbox { to be the class of } {\rm S}_{\rmE}({\it k}; \rmC', \rmD')[-m], \res)\notag.
\end{align} \ee
The property that $\rmT_{\rm p_1}(\rmA', \rmB') = \rmA'$ and $\rmT_{\rm p_2}(\rmA', \rmB') =\rmB'$ shows that $\alpha(\rmp_1) =[\rmA]$ and $\alpha(\rmp_2) = [\rmB]$. Similarly, $\beta(\rmp_1) = [\rmC]$ and $\beta(\rmp_2) = [\rmD]$.
This proves (i). 
\vskip .2cm
By Proposition ~\ref{Gillet-Soule-prop2}, any element in $\rmR_{\mathbb Z}(\rmM_N \times \rmM_N)^d$ is a polynomial of weight at most $d$ in the exterior powers of $p_1$ and $p_2$. 
Since the functor $\rmE \mapsto \rmT_{\rmE}(\rmP, \rmQ)$ (for a fixed $\rmP$ and $\rmQ$) is exact in $\rmE$ and preserves
sums, products and exterior powers in $\rmE$ as already observed (see ~\ref{TE} and Proposition ~\ref{TE.st}), and
both $\pi_0(\rmK (\S))$ and $\pi_0(\rmK_{\S' \times \Delta [n]}(\S \times \Delta [n], \S \times \delta \Delta [n]))$ are pre-$\lambda$-rings,
we see that we obtain additive homomorphisms
\vskip .2cm 
$\alpha: \rmR_{\mathbb Z}(\rmM_N \times \rmM_N) ^d \ra \pi_0(\rmK (\S))$ and
\vskip .2cm 
$\beta:  \rmR_{\mathbb Z}(\rmM_N \times \rmM_N) ^d \ra \pi_0(\rmK_{\S'\times \Delta [n]}(\S \times \Delta [n], \S \times \delta \Delta [n]))$,
\vskip .1cm \noindent
which are also multiplicative, and preserve the $\lambda$-operations. 
Moreover, each element in the image of $\alpha$ ($\beta$) can be written
as $\rmT_{\rmE}({\it k}; \rmA, \rmB)$ (${\rm S}_{\rmE}({\it k}; \rmC, \rmD)$, \res) for some $\rmE \in \rmR_{\mathbb Z}(\rmM_N \times \rmM_N)^d$.
\vskip .2cm 
By the properties of the functor $\rmT_{\rmE}$ discussed before, it follows that if \\
$\rmR(X_1, \cdots, X_N; Y_1, \cdots, Y_N)$ is a polynomial with integral coefficients and
of weight at most $d$, we obtain:
\be \begin{equation}
\label{equality.lambda}
\alpha (\rmR(\lambda^1(\rmp_1), \cdots, \lambda^N(\rmp_1); \lambda ^1(\rmp_2), \cdots, \lambda^{\rm N}(\rmp_2))) = \rmR(\lambda^1({\it x}), \cdots, \lambda^{\rmN}({\it x}); \lambda^1({\it y}), \cdots, \lambda^{\rmN}({\it y})) \end{equation} \ee
\vskip .2cm \noindent
where $x= \alpha ([\rmA])$ and $y= \alpha([\rmB])$. Similarly, 
\be \begin{equation}
\beta (\rmR(\lambda^1(\rmp_1), \cdots, \lambda^{\rmN}(\rmp_1); \lambda ^1(\rmp_2), \cdots, \lambda^{\rmN}(\rmp_2))) =R(\lambda^1({\it x}), \cdots, \lambda^{\rmN}({\it x}); \lambda^1({\it y}), \cdots, \lambda^{\rmN}({\it y}))
\end{equation} \ee
\vskip .2cm \noindent
where $x= \beta([\rmC])$ and $y= \beta([\rmD])$.   Since the functors $\rmE \mapsto \rmT_{\rmE}({\it k}; \rmA, \rmB)$ and $\rmE \mapsto {\rm S}_{\rmE}({\it k}; \rmC, \rmD)$ are compatible with
tensor products and exterior powers, (ii) and (iii) of the Proposition follow readily.
\end{proof}
First we draw the following consequences of the last proposition
\begin{corollary}
\label{lambda.ring.0}
\label{cor1} $\pi_0(\rmK (\S))$ is a $\lambda$-ring and for each $n \ge 0$, $\pi_n(\rmK_{\S'}(\S)) \cong \pi_0(\rmK_{\S'\times \Delta [n]}(\S\times \Delta [n], \S \times \delta \Delta [n]))$ is a $\lambda$-ring without a unit element.
\end{corollary}
\begin{proof} We already know from Theorem ~\ref{mainthm}(i) that $\pi_0(\rmK (\S))$ is a pre-$\lambda$-ring with unit and that $\pi_n(\rmK_{\S'}(\S))\\ \cong \pi_0(\rmK_{\S'\times \Delta [n]}(\S\times \Delta [n], \S \times \delta \Delta [n]))$ is a pre-$\lambda$-ring (without a unit): the $\lambda$-operations in both cases are defined by the
exterior powers of perfect complexes. Therefore, what remains to be shown is that they satisfy the relations in ~\eqref{lambda.ids}. 
 This is a formal consequence of the last proposition.
Let $\rmC$, $\rmD$ denote two strictly perfect complexes on $\S \times \Delta[n]$ acyclic on $(\S-\S') \times \Delta[n]$ 
and provided with an explicit zig-zag path $p$ as in ~\eqref{zig.zag.path} running from  
their restriction to $\S \times \delta \Delta [n]$, to the base point $(0, 0)$
in ${\it w}\rmG(i^*)$.
 Let $\rmA$, $\rmB$ denote two
strictly perfect complexes on $\S$ and let $x =[\rmA]$, $y =[\rmB]$. 
\vskip .2cm
To check the
identity $\lambda^{\it k} (\lambda^l({\it x})) = {\rmP}_{\it k,l}(\lambda ^1({\it x}), \cdots, \lambda ^{\it kl}({\it x}))$ for
a certain universal polynomial $\rmP_{{\it k,l}}$, let
$d= kl$ and choose ${\rm N}$ as in the last proposition. Then 
\[\lambda ^{\it k} (\lambda ^l({\it x})) = \alpha (\lambda ^{\it k}( \lambda ^l(\rmp_1))) \mbox{ and } \rmP_{{\it k,l}}(\lambda ^1({\it x}), \cdots, \lambda ^{\it kl}({\it x})) = \alpha (\rmP_{{\it k,l}}(\lambda ^1(\rmp_1), \cdots, \lambda ^{\it kl}(\rmp_1))).\]
Since $\rmR_{\mathbb Z}(\rmM_N \times \rmM_N)^d$ is contained in the $\lambda$-ring $\rmR_{\mathbb Z}(\rmM_N \times \rmM_N)$, 
we obtain
the equality 
\[\lambda^{\it k}(\lambda^l(\rmp_1)) = \rmP_{\it k,l}(\lambda^1(\rmp_1), \cdots, \lambda ^{\it kl}(\rmp_1)).\] In view of properties of the functor $\rmE \mapsto \rmT_{\rmE}$ as discussed
above, it follows that $\alpha$ is an additive homomorphism that commutes with products and exterior powers. Therefore, we obtain the formula 
\[\lambda ^{\it k}(\lambda ^l(x)) = \rmP_{\it k,l}(\lambda^1({\it x}), \cdots, \lambda ^{\it kl}({\it x})).\]
Similarly, one checks the identity $\lambda ^{\it k}(xy) = \rmP_{\it k}(\lambda ^1({\it x}), \cdots, \lambda ^{\it k}({\it x}); \lambda^1({\it y}), \cdots, \lambda ^{\it k}({\it y}))$ for a certain universal polynomial $\rmP_{\it k}$. These prove that $\pi_0(\rmK (\S))$ is a $\lambda$-ring.
\vskip .2cm
Next one lets $x =[\rmC]$ and $y=[\rmD]$, and repeats the above argument with $\beta$ in the place of
$\alpha$ to prove  $\pi_0(\rmK_{\S' \times \Delta [n]}(\S\times \Delta [n], \S \times \delta \Delta [n]))$ is a $\lambda$-ring without a unit element. 
\end{proof}
\vskip .2cm
Let ${\mathbb Z}$ denote the ring of integers with its canonical $\lambda$-ring structure: see \cite[section 1]{AT}.
\begin{lemma}
\label{lambda.ring.1}
 Assume the above framework. Then 
${\mathbb Z} \oplus \pi_0(\rmK_{\S'\times \Delta [n]}(\S\times \Delta [n], \S \times \delta \Delta [n]))$ is a $\lambda$-ring where the operations are defined as follows.  If $n, m \in {\mathbb Z}$ and $s, t \in \pi_0(\rmK (\S\times \Delta [n], \S \times \delta \Delta [n]))$, 
\vskip .2cm
$(n, s) + (m, t) =(n+m, s+t)$, $(n, s) \circ (m, t) = (n.m, n.t+m.s+s.t)$ and
\vskip .2cm
$\lambda ^n(k,s) = (\lambda ^n(k), \Sigma _{i=0}^{n-1}\lambda ^i(k). \lambda ^{n-i}(s))$, $n>0$.
\vskip .2cm \noindent
Here $\circ$ denotes the multiplication in the graded ring ${\mathbb Z} \oplus \pi_0(\rmK_{\S'\times \Delta [n]}(\S\times \Delta [n], \S \times \delta \Delta [n]))$.
\end{lemma}
\begin{proof} It is straightforward to verify that these define a pre-$\lambda$-ring structure, 
where $(1, 0)$ is the unit element: see Lemma ~\ref{prelambda.alg}. We proceed to verify the relations in ~\eqref{lambda.ids} are satisfied. For these the following observations will be helpful.
\vskip .2cm
For any ring  $\rmR$ with unit, let $\hat G(\rmR) = 1+\rmR[[t]]^+$ = the power series in $t$ with coefficients in $\rmR$ and with the starting term 1. This is a $\lambda$ ring with the addition (which  will be denoted by $\boxplus$) being the product of power series, and multiplication (denoted $\bullet$) and exterior power
operations defined as in \cite[p. 258]{AT}. If $\rmR$ is also a pre-$\lambda$ ring, then the map $r \mapsto \lambda_t(r) = \Sigma _i \lambda^i(r) t^i$ is an {\it additive homomorphism of abelian groups} from $\rmR$ to $\hat G(\rmR)$. The same map is a ring homomorphism (a map of pre-$\lambda$-rings) if and only if the first  relation in ~\eqref{lambda.ids} is satisfied (both the relations in ~\eqref{lambda.ids}, \res \, are satisfied).
\vskip .2cm 
Therefore, to prove the first relation in ~\eqref{lambda.ids}, it suffices to show that 
\be \begin{equation}
\label{ring.1}
\lambda_t((n, x) \circ (m, y)) = \lambda_t(n, x) \bullet \lambda_t(m, y), \quad n, m \in {\mathbb Z}, \quad x, y \in \pi_0(\rmK_{\S'\times \Delta [n]}(\S\times \Delta [n], \S \times \delta \Delta [n])).
\end{equation} \ee
\vskip .2cm \noindent
 Using the product on the ring ${\mathbb Z} \oplus \pi_0(\rmK_{\S' \times \Delta[n]}(\S\times \Delta [n], \S \times \delta \Delta [n]))$, the left-hand-side identifies with
\vskip .2cm 
$\lambda_t((n, 0)\circ (m, 0) +(n, 0) \circ (0, y)+  (0,x)\circ (m, 0) + (0, x)\circ (0, y))$.
\vskip .2cm \noindent
Since $\lambda_t$ is an additive homomorphism, this identifies with 
\vskip .2cm 
$\lambda_t((n, 0) \circ (m, 0)) \boxplus \lambda_t((n, 0) \circ (0, y)) \boxplus \lambda_t((0, x)\circ (m, 0)) \boxplus \lambda_t((0, x) \circ (0, y))$. 
\vskip .2cm \noindent
The term on the right-hand-side of ~\eqref{ring.1} identifies with
\vskip .2cm
$[(\lambda_t(n, 0)) \boxplus (\lambda_t(0, x))] \bullet [ (\lambda_t(m, 0)) \boxplus (\lambda_t(0, y))] $
\vskip .2cm
$= (\lambda_t(n, 0) \bullet \lambda_t(m, 0)) \boxplus (\lambda_t(n, 0) \bullet (\lambda_t(0, y)) \boxplus  ((\lambda_t(0, x)\bullet (\lambda_t(m, 0)) \boxplus (\lambda_t(0, x)) \bullet (\lambda_t(0, y))$. 
\vskip .2cm 
Since ${\mathbb Z}$ and $\pi_0(\rmK_{\S'\times \Delta [n]}(\S\times \Delta [n], \S \times \delta \Delta [n]))$ are $\lambda$-rings,  $\lambda_t((n, 0) \circ (m, 0)) = \lambda _t(n, 0) \bullet \lambda _t(m, 0)$ and $\lambda_t((0, x) \circ (0,y)) = \lambda _t(0, x) \bullet \lambda _t(0, y)$. 
Moreover, observe that $(0, x) \circ (m, 0) = (0, mx) = (m, 0) \circ (0, x)$.
Therefore, it suffices to show that $\lambda_t((m, 0) \circ (0, x)) = \lambda_t(m, 0) \bullet \lambda_t(0, x)$ for any positive integer $m$. However, $(m, 0) \circ (0, x) = (0, mx)$. Since $\lambda_t(1, 0) = 1+(1, 0)t$ is the
multiplicative unit in $\hat G({\mathbb Z} \oplus \pi_0(\rmK_{\S'\times \Delta [n]}(\S\times \Delta [n], \S \times \delta \Delta [n])))$, it follows that $\lambda_t(0, 1.x) = \lambda_t(1, 0) \bullet \lambda_t(0, x)$.
\vskip .2cm 
Assuming that $\lambda_t(0, nx) = \lambda_t(n,0) \bullet \lambda_t(0,x)$ for all $n < m$,
we observe that 
\vskip .2cm 
$\lambda_t(0, mx) = \lambda_t(0, x + (m-1)x) = \lambda _t(0, x) \boxplus \lambda_t(0, (m-1)x) $
\vskip .2cm
$= \lambda _t(0,x) \boxplus \lambda_t((m-1, 0) \circ (0, x)) = \lambda_t(1,0) \bullet \lambda _t(0, x) \boxplus \lambda_t((m-1), 0)\bullet \lambda _t(0,x)$
\vskip .2cm 
$ = (\lambda_t(1, 0) \boxplus \lambda _t((m-1), 0)) \bullet \lambda_t(0,x) = \lambda_t(m, 0) \bullet \lambda_t(0,x )$.
\vskip .2cm \noindent
This completes the proof of the first relation in ~\eqref{lambda.ids}. 
\vskip .2cm
To prove the second, we observe the square in \cite[ Expos\'e V, (3.7.1)]{SGA6}
\[\xymatrix{S \ar@<1ex>[r]^{\lambda _u} \ar@<-1ex>[d]_{\lambda_t} & {{\hat G}^u(S)} \ar@<1ex>[d]^{{\hat G}^u(\lambda_t)}\\
{\hat G}^t(S) \ar@<1ex>[r]^{\lambda_u} & {{\hat G}^u({\hat G}^t(S))},}\]
where $S = {\mathbb Z} \oplus \pi_0(\rmK_{\S' \times \Delta [n]}(\S\times \Delta [n], \S \times \delta \Delta [n]))$. Given a 
pre-$\lambda$-ring $\rmR$, ${\hat G}^u(R)$ (${\hat G}^t(R)$) denotes the power series ring considered above in the variable $u$ ($t$, \res). The second relation in ~\eqref{lambda.ids} holds if and only if
the above square {\it commutes}: see \cite[ Expos\'e V, 3.7]{SGA6}. 
Since all the maps in the above diagram are group homomorphisms, it suffices to show the square above commutes separately
for elements of the form $(n, 0)$ and $(0, x)$ with
$n \in {\mathbb Z}$ and $x \in \pi_0(\rmK_{\S' \times \Delta [n]}(\S\times \Delta [n], \S \times \delta \Delta [n]))$.
But this is equivalent to showing 
the required relations hold separately for elements of the form $(n, 0)$ and $(0, x)$ with
$n \in {\mathbb Z}$ and $x \in \pi_0(\rmK_{\S' \times \Delta [n]}(\S\times \Delta [n], \S \times \delta \Delta [n]))$.
This is clear since we already know from Corollary ~\ref{lambda.ring.0} that the elements in  $\pi_0(\rmK_{\S' \times \Delta [n]}(\S\times \Delta [n], \S \times \delta \Delta [n]))$ satisfy the second relation in ~\eqref{lambda.ids}. (Clearly the elements in ${\mathbb Z}$ also satisfy this relation since ${\mathbb Z}$ is a $\lambda$-ring with its canonical structure.)
This completes the proof of the lemma. \end{proof} 
\vskip .2cm
\begin{proposition}
\label{lambda.ring.2}
Assume the above framework. Then 
$ \pi_0(\rmK (\S)) \oplus \pi_0(\rmK_{\S'\times \Delta [n]}(\S\times \Delta [n], \S \times \delta \Delta [n]))$ is a $\lambda$-ring where the operations are defined as follows.  If $u, v \in \pi_0(\rmK (\S))$ and $s, t \in \pi_0(\rmK (\S\times \Delta [n], \S \times \delta \Delta [n]))$, 
\vskip .2cm
$(u, s) + (v, t) =(u+v, s+t)$, $(u, s) \circ (v, t) = (u.v, u.t+v.s+s.t)$ and
\vskip .2cm
$\lambda ^n(u,s) = (\lambda ^n(u), \Sigma _{i=0}^{n-1}\lambda ^i(u). \lambda ^{n-i}(s))$, $n>0$.
\vskip .2cm \noindent
Here $\circ$ denotes the multiplication in the graded ring $\pi_0(\rmK (\S)) \oplus \pi_0(\rmK_{\S'\times \Delta [n]}(\S\times \Delta [n], \S \times \delta \Delta [n]))$.
\end{proposition}
\begin{proof} The proof of this Proposition will be very similar to the proof of Lemma ~\ref{lambda.ring.1}. Observe first that the proof of the
 second relation in ~\eqref{lambda.ids} given in the proof of Lemma ~\ref{lambda.ring.1} carries over verbatim with the 
 ring ${\mathbb Z}$ replaced by $ \pi_0(\rmK (\S))$. Therefore, it suffices to consider the proof of the first relation in ~\eqref{lambda.ids}, that is, it suffices to prove:
 \be \begin{equation}
  \label{ring.2}
\lambda_t((u, x) \circ (v, y)) = \lambda_t(u, x) \bullet \lambda_t(v, y), \quad u, v \in \pi_0(\rmK (\S)), \quad {\it x}, {\it y} \in \pi_0(\rmK_{\S'\times \Delta [n]}(\S\times \Delta [n], \S \times \delta \Delta [n])).
\end{equation} \ee
Using the product on the ring $\pi_0(\rmK (\S)) \oplus \pi_0(\rmK_{\S' \times \Delta[n]}(\S\times \Delta [n], \S \times \delta \Delta [n]))$, the left-hand-side identifies with
\vskip .2cm 
$\lambda_t((u, 0)\circ (v, 0) +(u, 0) \circ (0, y)+ (v,0) \circ (0,x) + (0, x)\circ (0, y))$.
\vskip .2cm \noindent
Since $\lambda_t$ is an additive homomorphism, this identifies with 
\vskip .2cm 
$\lambda_t((u, 0) \circ (v, 0)) \boxplus \lambda_t((u, 0) \circ (0, y)) \boxplus \lambda_t((v, 0) \circ (0, x)) \boxplus \lambda_t((0, x) \circ (0, y))$. 
\vskip .2cm \noindent
The term on the right-hand-side of ~\eqref{ring.2} identifies with
\vskip .2cm
$[(\lambda_t(u, 0)) \boxplus (\lambda_t(0, x))] \bullet [ (\lambda_t(v, 0)) \boxplus (\lambda_t(0, y))] $
\vskip .2cm
$= (\lambda_t(u, 0) \bullet \lambda_t(v, 0)) \boxplus (\lambda_t(u, 0) \bullet (\lambda_t(0, y)) \boxplus  ((\lambda_t(v, 0)\bullet (\lambda_t(0, x)) \boxplus (\lambda_t(0, x)) \bullet (\lambda_t(0, y))$. 
\vskip .2cm 
Since $\pi_0(\rmK (\S))$ and $\pi_0(\rmK_{\S'\times \Delta [n]}(\S\times \Delta [n], \S \times \delta \Delta [n]))$ are $\lambda$-rings,  $\lambda_t((u, 0) \circ (v, 0)) = \lambda _t(u, 0) \bullet \lambda _t(v, 0)$ and $\lambda_t((0, x) \circ (0,y)) = \lambda _t(0, x) \bullet \lambda _t(0, y)$. 
Moreover, observe that $(0, x) \circ (v, 0) = (0, vx) = (v, 0) \circ (0, x)$.
Therefore, it suffices to show that $\lambda_t((v, 0) \circ (0, x)) = \lambda_t(v, 0) \bullet \lambda_t(0, x)$ for any class $v \in \pi_0(\rmK (\S))$. 
\vskip .1cm
At this point, one may apply the splitting principle to elements of $\pi_0(\rmK (\S))$ (by taking the projective space bundle associated to a given
vector bundle on $\S$), so that we may assume the class $v$ breaks up into a finite sum of the classes of line bundles: $v= \Sigma_{i=1}^m[\L_i]$,
where each $\L_i$ is a line bundle on $\S$. 
\vskip .2cm 
Assuming that $\lambda_t(0, (\Sigma_{i=1}^n[\L_i])x) = \lambda_t(\Sigma_{i=1}^n[\L_i],0) \bullet \lambda_t(0,x)$ for all $n < m$,
we observe that 
\vskip .2cm 
$\lambda_t(0, (\Sigma_{i=1}^m[\L_i])x) = \lambda_t(0, [\L_m]x + (\Sigma_{i=1}^{m-1}[\L_i])x) = \lambda _t(0, [\L_m]x) \boxplus \lambda_t(0, (\Sigma_{i=1}^{m-1}[\L_i])x) $
\vskip .2cm
$= \lambda _t(0, [\L_m]x) \boxplus \lambda_t((\Sigma_{i=1}^{m-1}[\L_i]), 0) \circ (0, x)) = \lambda_t([\L_m],0) \bullet \lambda _t(0, x) \boxplus \lambda_t((\Sigma_{i=1}^{m-1}[\L_i]), 0)\bullet \lambda _t(0,x)$
\vskip .2cm 
$ = (\lambda_t([\L_m], 0) \boxplus \lambda _t((\Sigma_{i=1}^{m-1}[\L_i]), 0)) \bullet \lambda_t(0,x) = \lambda_t((\Sigma_{i=1}^{m}[\L_i]), 0) \bullet \lambda_t(0,x )$.
\vskip .2cm \noindent
Therefore, it suffices to prove that if $v=[\L]$ is the class of a line bundle on $\S$, and $x$ denotes a class in $\pi_0(\rmK_{\S' \times \Delta[n]}(\S\times \Delta [n], \S \times \delta \Delta [n]))$, then one obtains:
\[\lambda_t(0, [\L]x) = \lambda_t([\L], 0) \bullet \lambda_t(0,x).\]
We may assume the class $x$ is represented by the class of a perfect complex $\rmP$  on $\S\times \Delta[n]$ acyclic on $ (\S - \S') \times \Delta [n]$, 
provided with a zig-zag path (as in ~\eqref{zig.zag.path}) 
$p$ joining the restriction $(i^*(\rmP), 0)$ to the base point, namely the pair $(0, 0)$ in ${\it w}\rmG(i^*)$.
Now verifying the above relation amounts to verifying the first relation in ~\eqref{lambda.ids}: as is well known, since $\L$ is a line bundle on $\S$,
this amounts to observing the (functorial) isomorphism: 
\[ [\Lambda^n(\L \otimes \rmP)] = [\L^{\otimes n} \otimes \Lambda^n(\rmP)], n\ge 0. \]
as classes in $\pi_0(\rmK_{\S' \times \Delta[n]}(\S\times \Delta [n], \S \times \delta \Delta [n]))$.
This is clear since there is a functorial isomorphism $\Lambda^n(\L \otimes \rmP) \cong \L^{\otimes n} \otimes \Lambda ^n(\rmP)$.
This completes the proof of the first relation in ~\eqref{lambda.ids} and hence the proof of the proposition.
 \end{proof}
\vskip .1cm
{\it This concludes the
proof of Theorem ~\ref{mainthm}}. \qed
\vskip .2cm
\begin{remark} 
\label{end.proof.mainthm}
 Observe also that the restriction to smooth stacks becomes necessary so that one has the homotopy property for K-theory. (This fails, in general, even for non-regular schemes.)  \end{remark} 
\vskip .2cm
\section{\bf $\gamma$-operations and Absolute cohomology}
\subsection{\bf Standing hypothesis}
\label{std.hyp}
For the rest of the paper we will assume  that all algebraic stacks $\S$ we consider are smooth, and every coherent sheaf on $\S$ is 
the quotient of a vector bundle. If  $\S'$ is a closed substack of $\S$, Theorem ~\ref{mainthm}(iii) shows that 
there is the structure of a  $\lambda$-algebra (in the sense of Definition ~\ref{def1.1}) on each $\pi_n(\rmK_{\S'}(\S))$ over $\pi_0(\rmK (\S))$.
\begin{definition}
\vskip .1cm \noindent
(a) Recall that each $ \pi_0(\rmK (\S)) \oplus \pi_n(\rmK_{\S'}(\S))$ is a $\lambda$-ring with the operations defined above. Therefore, one may define the operations $\gamma^n$ on $\pi_0(\rmK (\S)) \oplus \pi_n(\rmK_{\S'}(\S))$ as follows:
\be \begin{equation}
\label{gamma}
\gamma ^n(\alpha, \beta)  = \lambda ^n((\alpha + (n-1).\O_{\S}), \beta), \alpha \in \pi_0\rmK (\S), \ \beta \in \pi_n(\rmK_{\S'}(\S)).
\end{equation} \ee
\vskip .2cm \noindent
One may observe that if $\alpha = 0$, then $\gamma^n(0, \beta) = (0, \Sigma _{i=0}^{n-1}\lambda ^i((n-1). \O_{\S}). \lambda ^{n-i}(\beta))$ (see ~\eqref{lambda.def}), so that each $\gamma^n$ induces a map
on $\pi_n(\rmK_{\S'}(\S))$ which we will also denote by $\gamma ^n$. 
 \vskip .1cm \noindent
(b) One defines the $\gamma$-filtration on each $ \pi_n(\rmK_{\S'}(\S))$ as follows. Let $\epsilon:\pi_0(\rmK (\S)) \ra {\mathbb Z}$
denote the augmentation given by the {\it rank}-map: the function $\epsilon$ is the rank of a strictly perfect complex defined as an obvious  Euler characteristic involving the ranks of the constituent terms of the complex. Then we define $\rmF^m( \pi_n(\rmK_{\S'}(\S)))$ to be generated by
$\gamma^{i_1}a_1 \cdots \gamma^{i_{\it k}}a_{i_{\it k}}\gamma^{j_1}x_1 \cdots \gamma^{j_p}x_p$, where $a_i \in \pi_0(\rmK (\S))$ with
$\epsilon(a_i)=0$, for all $i=1, \cdots, k$, and $x_{j_i} \in \pi_n(\rmK_{\S'}(\S))$, so that $i_1+ \cdots i_{\it k} +j_1 +\cdots j_p \ge m$.
(See \cite[section 6]{Kr-2} and/or \cite[p. 105]{Wei}.) 
\vskip .1cm
(c) One may define the Adams operations $\psi^{\it k}$ using ascending induction on $k$ and the formula:
$\psi^{\it k} = \psi^{k-1}  \lambda^1 - \cdots + (-1)^{\it k} \psi^1 \lambda ^{k-1} + 
(-1)^{k+1} k \lambda^{\it k}$: see \cite[p. 102]{Wei}.
\end{definition}
\vskip .2cm
Then one may readily verify the following properties of the $\gamma$-filtration for each $n$:
\begin{enumerate}[\rm(i)]
 \item $\rmF^{m+1}( \pi_n(\rmK_{\S'}(\S))) \subseteq \rmF^{m}( \pi_n(\rmK_{\S'}(\S)))$, for each $m\ge 0$ and 
 \item $\rmF^1( \pi_n(\rmK_{\S'}(\S))) \subseteq \rmF^0( \pi_n(\rmK_{\S'}(\S))) =  \pi_n(\rmK_{\S'}(\S))$.
\end{enumerate}
Since the product on each $\pi_n(\rmK_{\S'}(\S))$ is trivial for all $n>0$, one may observe that the
$\gamma$-filtration $\rmF^m( \pi_n(\rmK_{\S'}(\S)))$, for $n>0$, is generated by $\gamma^{i_1}a_1 \cdots \gamma^{i_{\it k}}a_{i_{\it k}}\gamma^{j_1}x_{j_1} $, where $a_i \in \pi_0(\rmK_{\S'}(\S))$ with
$\epsilon(a_i)=0$, for all $i=1, \cdots, k$, and $x_{j_1} \in \pi_n(\rmK_{\S'}(\S))$, so that $i_1+ \cdots i_{\it k} +j_1 + \ge m$.
Now one may readily verify the following additional properties for each $m, \, m'\ge 0$:
\begin{enumerate}
\item one has a pairing: $\rmF^{m}(\pi_0(\rmK_{\S'}(\S))) \otimes \rmF^{m'}( \pi_0(\rmK_{\S'}(\S))) \ra \rmF^{m+m'}( \pi_0(\rmK_{\S'}(\S)))$, \mbox{ and}
\item $\rmF^{m}( \pi_0(\rmK_{\S'}(\S)))$ is a $\lambda$-ideal in $ \pi_0(\rmK_{\S'}(\S))$. 
\end{enumerate}

\vskip .2cm \noindent
{\bf Proof of Theorem ~\ref{thm.2}}.
The  properties of the $\gamma$-operations follow from the observation that
$\pi_0(\rmK (\S)) \oplus \pi_n(\rmK_{\S'}(\S))$ is a $\lambda$-ring.
Again one observes that each $\psi^{\it k}$ induces a self-map of
$\pi_n(\rmK_{\S'}(\S)) $ for each closed sub-stack $\S'$ of $\S$. The last but one
statement in (i) follows from the functoriality of the $\lambda$ and $\gamma$-operations with respect to pull-back. 
The last statement in (i) is a pure consequence of the $\lambda$-ring structure on $\pi_0(\rmK (\S)) \oplus \pi_n(\rmK_{\S'}(\S))$.

These prove the statements in (i); the proof of statements in (ii) are clear
since the $\lambda$-operations are compatible with respect to pull-backs. \qed
\begin{remarks} 
\label{non.nilp}
1. It is important to point out that the action of $\psi_{\it k}$ above is {\it not} locally nilpotent, which is necessary to conclude that $\pi_*\rmK (\S) \otimes \Q$ is
isomorphic to the sum of the associated graded terms of the $\gamma$-filtration. This is false in general as may be seen from the following simple
counter-example: consider $\S = [(Spec \,k)/\rmG]$ where $\rmG$ is a finite group and $k$ is a field. In this case, it is shown in \cite[Proposition (6.13)]{A}
 that the $\gamma$-filtration has just two terms modulo torsion. 
\vskip .2cm
2. Observe also that the $\gamma$-operations on $\pi_0(\rmK (\S))$ are compatible with the
$\gamma$-operations on  $\pi_n(\rmK_{\S'}(\S))$ in the following sense. Let $\alpha \in
\pi_0(\rmK (\S))$ and $\beta \in \pi_n(\rmK_{\S'}(\S))$. Then $(\alpha, 0) .(0, \beta) = (0, \alpha. \beta)$ using the module structure of $\pi_n(\rmK_{\S'}(\S))$ over $\pi_0(\rmK (\S))$. Now
\[\gamma ^i(\alpha, 0). \gamma^j(0, \beta ) = (\gamma ^i(\alpha), 0). (0, \gamma ^j(\beta)) = (0, \gamma^i(\alpha). \gamma ^j(\beta)).\] 
Moreover, since $(0, \alpha. \beta) = (\alpha, 0).(0, \beta)$, it follows that 
\be \begin{align}
(0, \gamma ^{\it k}(\alpha. \beta)) = \gamma^{\it k}(0, \alpha. \beta) &=\gamma ^{\it k}((\alpha, 0). (0, \beta)) \\
 &= Q_{\it k}((\gamma ^1(\alpha), 0), \cdots, (\gamma ^{\it k}(\alpha), 0); (0, \gamma ^1(\beta)), \cdots , (0, \gamma ^{\it k}(\beta ))) \notag\\
 &= (0, \rmQ_{\it k}(\gamma ^1(\alpha), \cdots, \gamma ^{\it k}(\alpha); \gamma ^1(\beta), \cdots , \gamma ^{\it k}(\beta ))) \notag.
 \end{align} \ee
\end{remarks}
\begin{definition}
\label{def.abs.coh}
Let ${\rm gr}^n(\pi_j\rmK_{\S'}(\S) \otimes \Q)$ denote the $n$-th graded piece of the $\gamma$-filtration. We let $\rmH^i_{\S', abs}(\S, \Q(j)) =
gr^j(\pi_{2j-i}(\rmK_{\S'}(\S)) \otimes \Q)$. We define the $i$-th Chern class 
\[c_i(j): \pi_0(\rmK (\S)) \oplus \pi_i(\rmK (\S)) \ra \rmH^{2j}_{abs}(\S; \Q(j)) \oplus \rmH^{2j-i}_{abs, \S'}(\S; \Q(j))\]
\noindent
by $c_i(j)(\alpha, \beta) = \gamma^j(\alpha - rk(\alpha). {\mathcal O}_{\S}, \beta)$ where $\gamma ^j$ is the $j$-th $\gamma$-operation on $\pi_0(\rmK (\S)) \otimes \Q \oplus \pi_i(\rmK_{\S'}(\S)) \otimes \Q$. 
 If $i=0$ and $\beta =0$, we let the Chern class $c_i(j)$ be denoted
$C(j)$. If $\alpha =0$, we obtain Chern classes $c_i(j):\pi_i(\rmK_{\S'}(\S)) \ra \rmH^{2j-i}_{abs, \S'}(\S; \Q(j))$. We define the  Chern-character into ${\underset i \Pi} \rmH^{2j-i}_{abs}(\S; \Q(j))$ by the usual formula: see \cite[Expos\'e 0: Appendix]{SGA6}. (Observe we are taking the product in the last expression and not the sum, only because the $\gamma$-filtration is not locally nilpotent.) For a vector bundle ${\mathcal E}$, one may define its Todd class by the usual Todd polynomial in the Chern classes: see \cite[Chapter 1, section 4]{FL}. \end{definition}
\vskip .2cm \noindent
{\bf Proof of Theorem ~\ref{loc.abs.coh}}. Recall that the statement we want to prove is the existence of 
the long exact sequence of absolute cohomology groups:
\vskip .1cm
$\cdots \ra \rmH^n_{\S_0', abs}(\S, \Q(i)) \ra \rmH^n_{\S_1', abs}(\S, \Q(i)) \ra \rmH^n_{\S_1'-\S_0', abs}(\S - \S_0', \Q(i)) \ra \rmH^{n+1}_{\S_0', abs}(\S, \Q(i)) \ra \cdots$,
\vskip .1cm \noindent
where $\S$ is a smooth algebraic stack with the property that every coherent sheaf is the quotient of a vector bundle and that
$\S_0' \subseteq \S_1'$ are two closed algebraic substacks. 
\vskip .2cm 
We begin with the fibration sequence (localized at $\Q$):
\vskip .2cm
$\Omega( \rmK_{\S_1'- \S_0'}(\S - \S_0')_{\Q}) \ra \rmK_{\S_0'}(\S)_{\Q} \ra \rmK_{\S_1'}(\S)_{\Q} \ra \rmK_{\S_1'- \S_0'}(\S - \S_0')_{\Q} $
\vskip .2cm \noindent
On taking the associated homotopy groups one obtains a long exact sequence
\be \begin{equation}
\label{l.ex.seq}
\xymatrix{{\cdots} \ar@<1ex>[r] & {\pi_{\it k}(\rmK_{\S_0'}(\S))\otimes \Q} \ar@<1ex>[r]^{\alpha} & {\pi_{\it k}(\rmK_{\S_1'}(\S))\otimes \Q} \ar@<1ex>[r]^{\beta} & {\pi_{\it k}(\rmK_{\S_1'-\S_0'}(\S-\S_0'))\otimes \Q} \ar@<1ex>[r]^{\qquad \qquad \gamma} & \cdots}
\end{equation} \ee
Since the $\gamma$-filtration is compatible with respect to pull-backs, one obtains  the commutative diagram:
\vskip .2cm
\be \begin{equation}
\label{exact.diagm}
\xymatrix{{\cdots} \ar@<1ex>[r] &{A_{n+2}^{i+1}} \ar@<1ex>[r]^{\alpha^{i+1}} \ar@<-1ex>[d]^{f^{i+1}_{n+2}} & {B_{n+2}^{i+1}} \ar@<1ex>[r]^{\beta^{i+1}} \ar@<-1ex>[d]^{g^{i+1}_{n+2}} & {C_{n+2}^{i+1}} \ar@<1ex>[r]^{\gamma^{i+1}} \ar@<-1ex>[d]^{h^{i+1}_{n+2}} & {A_{n+3}^{i+1}} \ar@<1ex>[r] \ar@<-1ex>[d]^{f^{i+1}_{n+3}} & {\cdots}\\
{\cdots} \ar@<1ex>[r] &{A_n^{i}} \ar@<1ex>[r]^{\alpha ^i} & {B_n^{i}} \ar@<1ex>[r]^{\beta^i}  & {C_n^{i}} \ar@<1ex>[r]^{\gamma ^i}  & {A_{n+1}^{i}} \ar@<1ex>[r]  & {\cdots}}
\end{equation} \ee
\vskip .2cm \noindent
where $A_n^i = \rmF^i(\pi_{2i-n}(\rmK_{\S'_0}(\S))\otimes \Q)$, $B_n ^i = \rmF^i(\pi_{2i-n}(\rmK_{\S_1'}(\S)) \otimes \Q)$ and $C_n^i = \rmF^i(\pi_{2i-n}(\rmK_{\S_1'-\S_0'}(\S - \S_0')) \otimes \Q)$. The maps $\alpha^i$ ($\beta ^i$, $\gamma^i$) are the maps induced by
$\alpha$ ($\beta$, $\gamma$, \res). Observe that {\it all the vertical maps are given by the 
inclusion of $\rmF^{i+1}$ into $\rmF^i$, and are therefore injective} and that
\[\rmH^n_{\S_0', abs}(\S, \Q(i)) = coker (f^{i+1}_{n+2}), \rmH^n_{\S_1', abs}(\S, \Q(i)) = coker (g^{i+1}_{n+2}) \mbox{ and }\]
\[\rmH^n_{\S_1'-\S_0', abs}(\S-\S_0', \Q(i)) = coker (h^{i+1}_{n+2}).\]
\vskip .2cm
We proceed to show that both rows in the diagram ~\eqref{exact.diagm} are {\it exact}. For example, we will show that
$ ker(\beta^{i}) = Im (\alpha ^{i})$. Let $b \in B_n^{i}$ so that $\beta^{i} (b) =0$.
Then the exactness of the long exact sequence of homotopy groups in ~\eqref{l.ex.seq} shows that there is a class $a \in \pi_{2i-n}(\rmK_{\S_0'}(\S)) \otimes \Q$ so that $\alpha (a) = b$.
 Now $A_n^i$ is a direct factor of  $\pi_{2i-n}(\rmK_{\S_0'}(\S)) \otimes \Q$. We let $a'$ denote the projection of $a$ to the factor $A_n^i$. 
Now, both $b=\alpha (a)$ and $\alpha (a')$ belong to $B_n^i$. It suffices to show
$\alpha (a) -\alpha (a') = \alpha(a-a') =0$. Observe that the associated graded terms in the $\gamma$-filtration of $ a-a'$
are of weight {\it strictly lower} than $i$. In particular, when  one breaks $a-a' $ into the sum of terms $a_j$ belonging to eigen spaces for the Adams operations $\psi^{\it k}$, the eigen values will
all be of the form $k^j$, $0 \le j <i$. (Observe that this also means $a-a'$ breaks up into a finite sum $\Sigma a_j$, with $a_j$
belonging to the eigenspace for $\psi^k$ with eigenvalue $k^j$, $0 \le j <i$.) 
\vskip .1cm
Since $\alpha$ preserves the $\gamma$-filtrations,  the Adams operations act on $\alpha (a_j)$ with eigen value $k^j$, $j <i$. The eigen values of $\psi^{\it k}$ on $B_n^i= F^i( \pi_{2i-n}(\rmK_{\S'_1}(\S)) \otimes \Q)$ are all of the form $k^j$, $j \ge i$. 
 Therefore, the projection of $\alpha (a) - \alpha (a')$ to $B_n^i$ is zero and $b=\alpha (a) = \alpha (a') = \alpha^{i}(a')$ as classes in $B_n^i$. 
A similar argument shows the exactness of both rows. Now a diagram-chase shows that the sequence
\vskip .2cm
$\cdots \ra \rmH^n_{\S'_0, abs}(\S, \Q(m)) \ra \rmH^n_{\S'_1, abs}(\S, \Q(m)) \ra \rmH^n_{\S'_1 - \S'_0, abs}(\S - \S'_0, \Q(m)) \ra \rmH^{n+1}_{\S'_0, abs}(\S, \Q(m)) \ra \cdots$
\vskip .2cm \noindent
is exact at the second term. See, for example, \cite[Proposition 1.4]{Iv}: observe that the sequence of absolute cohomology groups above is obtained by taking the cokernels of each column in the diagram ~\eqref{exact.diagm}. The exactness at the remaining terms may be proved similarly. \qed
\begin{remark} 
Assume  that the stack $\S$ has a coarse moduli space ${\mathfrak M}$. 
In this case, the observation that the $\gamma$-filtration on  $\pi_*\rmK (\S) \otimes \Q$ is
compatible with the $\gamma$-filtration on $\pi_*\rmK ({\mathfrak M}) \otimes \Q$ shows that
the absolute cohomology of the stack  we have defined is an algebra over the (usual) absolute
cohomology of the moduli space when the latter is defined.  \end{remark}
\section{\bf Examples}
We begin with the following theorem of Thomason as a source of several examples.
\begin{theorem} (Thomason: see \cite[Lemmas 2.4, 2.6, 2.10 and 2.14]{T-3}.)
\label{thomason.thm}
 Let $k$ denote a field, $\rmX$ a normal Noetherian scheme over $k$ with an ample
family of line bundles (for example, a smooth separated Noetherian scheme). Let $\rmG$ denote an affine flat group scheme of finite type over $k$ which is an extension of a finite flat group scheme by a 
 smooth connected group-scheme; let $\rmG$ act on $\rmX$. Then the
quotient stack $[\rmX / \rmG]$ has the resolution property. \end{theorem}
\vskip .2cm
To keep things simple, we will restrict to Noetherian schemes defined over a field $k$.
\begin{examples}
\label{egs}
\begin{enumerate}[\rm (i)]
\item{Let $\rmD$ denote a diagonalizable group scheme acting {\it trivially} on a smooth scheme $\rmX$. Then any
$\rmD$-equivariant vector bundle on $\rmX$ corresponds to giving a grading by the characters of $\rmD$ on
the vector bundle obtained by forgetting the action. It follows readily that $\pi_*\rmK ([\rmX/\rmD]) \cong \rmR(\rmD) \otimes \pi_*\rmK (\rmX)$. 
This is an isomorphism of $\lambda$-rings. Moreover on computing the absolute cohomology, we obtain: 
\[\rmH^*_{abs}([\rmX/\rmD], \Q(\bullet)) \cong grd(\rmR(\rmD) \otimes \Q) \otimes \rmH^*_{abs}(\rmX, \Q(\bullet))\]
\[=grd(\pi_0({\it K}([{\rm Spec \,} k/\rmD])\otimes \Q))\otimes \rmH^*_{abs}(\rmX, \Q(\bullet)),\]
where ${\rm grd}(\rmR(\rmD) \otimes \Q)  $ denotes the associated graded terms with respect to the $\gamma$-filtration.
\vskip .1cm
Clearly the graded ring ${\rm grd}(\rmR(\rmD) \otimes \Q) =grd(\pi_0({\it K}([{\rm Spec \,} k/\rmD]))\otimes \Q)$  has a natural decreasing filtration, 
and completing it with respect to this filtration
we obtain: 
\[\Pi_{n=0}^{\infty}grd_{n}(\pi_0({\it K}([{\rm Spec \,} k/\rmD]))\otimes \Q).\]
It is shown in \cite[(5.1) Proposition and (5.3) Proposition]{K}
that the latter is isomorphic to the completion $(\pi_0({\it K}([{\rm Spec \,} k/\rmD]))\otimes \Q) \compl_{I_D}$, where $\compl_{I_{\rmD}}$ denotes
completion at the augmentation ideal. Moreover, by \cite{EG}, the latter
 is isomorphic to $ \Pi_{i=0}^{\infty} {\rm CH}^i({\rm BD}, \Q)$, where ${\rm BD}$ denotes the {\it classifying space} for $\rmD$ defined as in \cite{Tot99} or \cite{MV99}.
 Thus we see that on completing the graded ring $\rmH^*_{abs}([\rmX/\rmD], \Q(\bullet)) \cong grd(\pi_0({\it K}([{\rm Spec \,} k/\rmD]))\otimes \Q)\otimes \rmH^*_{abs}(\rmX, \Q(\bullet))$ with respect to the 
 natural decreasing filtration induced from the one on $grd(\pi_0({\it K}([{\rm Spec \,} k/\rmD])\otimes \Q)$, we obtain the isomorphism:
 \[\rmH^*_{abs}([\rmX/\rmD], \Q(\bullet))\compl \, \cong (\Pi_{i=0}^{\infty} C\rmH^i(BD, \Q )) \otimes \rmH^*_{abs}(X, \Q(\bullet )).\]} 
\item{Let $\rmT$ denote a split torus acting on a smooth scheme $\rmX$. Assume further that there
is a {\it $\rmT$-stable stratification of $\rmX$ by strata which are all affine spaces}. In this case one obtains
the isomorphism of $\lambda$-rings: $\pi_*(K[\rmX/\rmT]) \cong \rmR(\rmT) \otimes \pi_*(\rmK (\rmX))$. One may obtain this isomorphism as follows: see \cite{J-1} for related results. One shows the obvious map of spectra $\rmK ([Spec \, k/\rmT]){\overset L {\underset {\rmK (Spec \, k)} \otimes}} \rmK (\rmX) \ra \rmK ([\rmX/\rmT])$ is a weak equivalence. Here one needs
to use the framework of \cite{J-1} of ring and module-spectra to be able to define the derived tensor product. 
Since $\rmX$ is smooth, its $K$-theory identifies with $\rmG$-theory and one uses the localization sequence associated to the 
stratification of $\rmX$ to show the above map is a weak equivalence. Now one obtains an associated spectral sequence with 
$E_2$-terms given by 
\[Tor^{\pi_*(\rmK (Spec \, k ))} (\pi_*(\rmK ([Spec \, k/\rmT])), \pi_*\rmK (\rmX)) \Ra \pi_*(\rmK ([\rmX/\rmT])).\]
This spectral sequence degenerates at the $E_2$-terms in view of the isomorphism 
\[\pi_*(\rmK ([Spec \, k/\rmT])) \cong \rmR(\rmT) \otimes \pi_*(\rmK (Spec \,k))\]
and provides the isomorphism $\pi_*\rmK ([\rmX/\rmT]) \cong R(T) \otimes \pi_*\rmK (X)$. This result applies to the case when $\rmX$ is a flag variety or a smooth projective variety on which $\rmT$ acts with finitely many fixed points. One also
obtains the isomorphism of absolute cohomology 
\[\rmH^*_{abs}([\rmX/\rmT], \Q(\bullet)) \cong (grd(\pi_0({\it K}([{\rm Spec \,} k/\rmT])\otimes \Q) \otimes \rmH^*_{abs}(\rmX, \Q(\bullet))\]
and therefore,
\[\rmH^*_{abs}([\rmX/\rmT], \Q(\bullet))\compl \, \cong (\Pi_{n=0}^{\infty}grd_n(\pi_0({\it K}([{\rm Spec \,} k/\rmT])\otimes \Q) \otimes \rmH^*_{abs}(\rmX, \Q(\bullet))\]
\[= (\Pi_{n=0}^{\infty}{\rm CH}^n({\rm BT}, \Q)) \otimes \rmH^*_{abs}(\rmX, \Q(\bullet)).\]
where $\compl \, $ denotes completion with respect to the decreasing filtration on the graded ring 
\newline \noindent
$(grd(\pi_0({\it K}([{\rm Spec \,} k/\rmT])\otimes \Q) \otimes \rmH^*_{abs}(\rmX, \Q(\bullet))$ and ${\rm BT}$ again denotes the 
classifying space of ${\rmT}$ in the sense of \cite{Tot99} or \cite{MV99}.  These isomorphisms follow along the same lines as in (i).}
\item{Next, let $\rmG$ denote any split reductive group over $k$ with $\pi_1(\rmG)$ torsion free. Let $\rmT$ denote fixed maximal torus in $\rmG$. Let $\rmX$ denote a smooth $\rmG$-scheme. 
Then \cite[ Proposition 4.1]{Merk} shows the isomorphism (of $\lambda$-rings): 
\[\pi_*\rmK ([\rmX/\rmT]) \cong \rmR(\rmT) {\underset {\rmR(\rmG)} \otimes} \pi_*\rmK ([\rmX/\rmG]), \mbox{ and therefore},\]
\[\rmH^*_{abs}([\rmX/\rmT], \Q(\bullet)) \cong grd(\rmR(\rmT) \otimes \Q){\underset {grd(\rmR(\rmG) \otimes \Q)} \otimes}\rmH^*_{abs}([\rmX/\rmG], \Q(\bullet)).\]
Observe that there is a natural conjugation action by ${\rm N(T)}$ on ${\rm T}$, which induces a ${\rm W}={\rm N(T)/T}$-action on $\pi_*\rmK ([\rmX/\rmT])$ and
on $\rmH^*_{abs}([\rmX/\rmT], \Q(\bullet))$. Moreover ${\rm R(T)}^{\rm W} \cong {\rm R(G)}$.
Therefore, taking the ${\rm W}$-invariants of both sides, one obtains  the isomorphism $\pi_*\rmK ([\rmX/\rmT])^{\rm W} \cong \pi_*(\rmK ([\rmX/\rmG]))$. At the level of absolute cohomology one obtains: 
\[\rmH^*_{abs}([\rmX/\rmT], \Q(\bullet))^W \cong \rmH^*_{abs}([\rmX/\rmG], \Q(\bullet)).\]}
\end{enumerate}
\end{examples}
\begin{example}({\it Hironaka's example}.)
Here is a well-known example due to Hironaka. (See \cite[p. 15]{Kn}.) Assume the base field is algebraically closed. (We may also
assume the characteristic is $0$ as in the original example of Hironaka.) Let $V_0$ be the projective 3-space and $\gamma_1$ and $\gamma_2$ two conics intersecting normally in exactly two points $\rmP_1$ and $\rmP_2$. For $i=1, 2$, we construct $\bar V_i$ by blowing up first $\gamma_i$ and then $\gamma_{3-i}$ in the result. Let $V_i$ be the open set in $\bar V_i$ of points lying over $(V_0-P_{3-i})$. Let $\rmU$ be obtained by patching $V_1$ and $V_2$ together along the 
common open subset. Now $\rmU$ is a non-singular variety and over $\rmP_1$ and $\rmP_2$ the curves $\gamma _1$ and $\gamma _2$ have been blown up in opposite order. Let $\sigma _0:V_0 \ra V_0$ denote the projective
transformation of order $2$ that permutes $\rmP_1$ and $\rmP_2$ and $\gamma _1$ and $\gamma _2$. $\sigma _0$ induces an automorphism $\sigma : U \ra U$ of order $2$. Therefore we may take the finite group $\rmG= {\mathbb Z}/2$ and let it act on $\rmU$ by the action of $\sigma $. In this case the geometric quotient $\rmU/G$ fails to exist in the category of schemes, but exists only in the category of algebraic spaces. Nevertheless
Theorem ~\ref{thomason.thm} shows that the quotient stack $[\rmU/\rmG]$ has the resolution property so that
for each $n \ge 0$, $\pi_n(\rmK ([\rmU/\rmG]))$ is a $\lambda$-ring.
\end{example}
\begin{example} For the next example let $\E$ denote an elliptic curve. Then there are no nontrivial
representations of $\E$ so that $\pi_*\rmK ([Spec \,k/\E]) \cong \pi_*\rmK (Spec \,k)$. 
It follows that 
$\rmH^*_{abs}([Spec \, k/\E], \Q(\bullet)) \cong \rmH^*_{abs}(Spec \, k, \Q (\bullet))$.
\end{example}
\subsection{Comparison with the higher equivariant Chow groups and further examples}
The comparison with the higher equivariant Chow groups (in the sense of \cite{EG} or \cite{Tot99}) is much more involved in {\it general} than is possible in the 
examples considered above. This is due to the fact that the absolute cohomology for algebraic stacks obtained above is a {\it Bredon-style}
cohomology theory in the sense of \cite{J-5}. In the case of quotient stacks this is related to the more familiar equivariant higher Chow groups defined by making use of a Borel-construction
(as in \cite{EG}, \cite{Tot99}) by a completion at the augmentation ideal of the representation ring of the given
linear algebraic group.  However, such a completion is not an exact functor in general, unless the modules that one considers are finitely generated
over the representation ring. In fact it shown in \cite{CJ23} that one needs to apply a {\it derived completion} to pass from the Algebraic
K-theory of smooth quotient stacks to the Algebraic K-theory of the corresponding Borel construction. One may apply results of
\cite{Lev} to the latter to define $\gamma$-operations and a form of absolute cohomology theory, which will then identify with the
equivariant higher Chow groups with rational coefficients, as in \cite{EG} or \cite{Tot99}.
\section{\bf Appendix A: Key theorems of Waldhausen K-theory}
\begin{definition} 
\label{Waldh.cat.def}
(See \cite[1.2.1]{T-T}.) A {\it category with cofibrations} ${\bold A}$ is a category with a zero object $0$, together with a chosen subcategory $co({\bold A})$ satisfying the following axioms: (i) any isomorphism in ${\bold A}$ is a morphism in $co({\bold A})$, (ii) for every object $A \in {\bold A}$, the unique map $0 \ra A$ belongs to $co({\bold A})$ and (iii) morphisms in $co({\bold A})$ are closed under co-base change by arbitrary maps in ${\bold A}$. The morphisms of $co({\bold A})$ are {\it cofibrations}. A category with fibrations is a category with 
a zero -object so that the dual category ${\bold A}^o$ is a category with cofibrations. 
A category with cofibrations and weak equivalences (or a {\it Waldhausen category}) is a category with cofibrations, $co({\bold A})$ together with a 
subcategory $w({\bold A})$ so that the following conditions are satisfied :(i) any isomorphism in ${\bold A}$ belongs to $w({\bold A})$, (ii) if 
\xymatrix{B \ar@<1ex>[d] & A \ar@<1ex>[d] \ar@<-1ex>[l]  \ar@<1ex>[r] & C \ar@<1ex>[d]\\
{B'} & {A'}\ar@<1ex>[l]  \ar@<-1ex>[r] & {C'}}
\vskip .2cm \noindent
 is a commutative diagram with the vertical maps
all weak equivalences and the horizontal maps in the left square are cofibrations, then the
induced map $B{\underset A \sqcup}C \ra B'{\underset {A'} \sqcup}C'$ is also a weak equivalence. (iii) If ${\it f}$, $g$ are two composable morphisms in $w({\bold A})$ and two of the three ${\it f}$, $g$ and $f \circ g$ are in $w({\bold A})$, then so is the third. A functor $F: {\bold A} \ra {\bold B}$ between categories with cofibrations and weak equivalences  is {\it exact} if it preserves cofibrations and weak equivalences.
\vskip .2cm
Given a Waldhausen category $({\bold A}, co({\bold A}), w({\bold A}))$, one associates to it the following simplicial category denoted $wS_{\bullet}{\bold A}$: see \cite[1.5.1 Definition]{T-T}. The objects of the category
$\wS_n{\bold A}$ are sequences of cofibrations 
\xymatrix{ {A_1\, } \ar@{>->} @<2pt> [r] & {A_{2}\, } \ar@{>->} @<2pt> [r]& {\cdots \, \,} \ar@{>->} @<2pt> [r] &{A_{ n}}}
in $co({\bold A})$ together with the choice of a quotient $A_{i,j}=A_j/A_i$ for each $i<j$ above. (The understanding is that  
$\wS_0{\bold A}$ is the category consisting of just the zero object $0$.) The morphisms between
two such objects 
\xymatrix{ {A_1\, } \ar@{>->} @<2pt> [r] & {A_{2}\, } \ar@{>->} @<2pt> [r]& {\cdots \, \,} \ar@{>->} @<2pt> [r] &{A_{ n}\, }} and
\xymatrix{ {B_1 \,} \ar@{>->} @<2pt> [r] & {B_{2}\, } \ar@{>->} @<2pt> [r]& {\cdots \, \,} \ar@{>->} @<2pt> [r] &{B_{ n}\, }}
are compatible collections of maps $A_{i,j} \ra B_{i, j}$ in $w{\bold A}$. Varying  $n$, one obtains the
simplicial category $\wS_{\bullet}{\bold A}$ as discussed in \cite[1.5.1 Definition]{T-T}.  
\vskip .1cm
Such  a Waldhausen category is {\it pseudo-additive} (see \cite[Definition 2.3]{GSVW}) if for each cofibration \xymatrix{{A\quad }\ar@{>->} @<2pt> [r] &{C}}, the induced maps
$C{\underset A \oplus}C \ra C \times C/A \leftarrow C\oplus C/A$ are weak equivalences. As pointed out earlier when
\xymatrix{{A\quad }\ar@{>->} @<2pt> [r] &{C}} is a degree-wise split injective map of complexes of sheaves of $\O$-modules, for a 
sheaf of commutative Noetherian rings with $1$ (on any site with enough points), it is easy to see that the maps $C{\underset A \oplus}C \ra C \times C/A \leftarrow C\oplus C/A$
are isomorphisms in each degree.
\end{definition} 
\vskip .2cm
The only  categories with cofibrations and weak equivalences  considered in this paper are {\it complicial } Waldhausen categories in the sense of \cite[1.2.11]{T-T}: in this situation the category ${\bold A}$ will be a full additive subcategory of the category of chain complexes with values in some abelian category. The cofibrations will be assumed to be maps of chain complexes that split degree-wise and weak equivalences will contain all quasi-isomorphisms. All the complicial Waldhausen categories we consider will be closed under the formation of the canonical homotopy pushouts and homotopy pull-backs as in \cite[1.9.6, 1.2.11]{T-T}. 
Therefore, all such categories with cofibrations and weak equivalences  are pseudo-additive.
\begin{definition}
 \label{Kth.sp}
 Given a category ${\bold A}$ with cofibrations and weak equivalences that is pseudo-additive, we define its K-theory space to be given by the simplicial set ${\it w}\rmG_{\bullet}({\bold A})$,
 where $\rmG_{\bullet}$ denotes the $\rmG$-construction discussed in \cite[Definition 2.2]{GSVW}. (See also \cite[section 3]{GG}.)
\end{definition}
Then one of the main results of \cite{GSVW} is the following:
\begin{theorem}(See \cite[Theorem 2.6]{GSVW}.) If ${\bold A}$ is pseudo-additive, there exists a natural map
 ${\it w}\rmG_{\bullet}({\bold A}) \ra \Omega \wS_{\bullet}{\bold A}$ that is a weak equivalence, where $\wS_{\bullet}{\bold A}$ is the
 simplicial set defined by the Waldhausen $\rmS_{\bullet}$-construction as in \cite[1.3]{Wald}.
\end{theorem}
In view of the above weak equivalence, various key results proved in \cite{Wald} extend readily to the K-theory spaces of
complicial Waldhausen categories. We state these below. 
\begin{theorem}
\label{W.app.thm}
(The Waldhausen approximation theorem :see \cite[1.9.8]{T-T}.) Let $F:{\bold A} \ra {\bold B}$ denote an exact functor between two complicial Waldhausen categories. Suppose $F$ induces an equivalence of the derived categories $w^{-1}({\bold A})$ and $w^{-1}({\bold B})$. Then $F$ induces a weak-homotopy 
equivalence of the associated K-theory spaces, $\rmK ({\bold A})$ and $\rmK ({\bold B})$. \end{theorem}
\vskip .2cm
\begin{theorem}
\label{loc.thm}
(Localization Theorem :see \cite[1.8.2]{T-T} and \cite[1.6.4]{Wald}) Let ${\bold A}$ be a small category with cofibrations and provided with two subcategories of weak equivalences $v({\bold A}) \subseteq w({\bold A})$ so that both $({\bold A}, co({\bold A}), v({\bold A}))$ and
$({\bold A}, co({\bold A}), w({\bold A}))$ are complicial Waldhausen categories (as in \cite[section 1]{T-T} .) Let ${\bold A}^{w}$ denote the full subcategory of ${\bold A}$ of objects $A$ for which $0 \ra A$ is in $w({\bold A})$, that is, are $w$-acyclic. This is a Waldhausen category with 
$co({\bold A}^w) = co({\bold A}) \cap {\bold A}^w$ and $v({\bold A}^w) = v({\bold A}) \cap {\bold A}^w$. Then one obtains the fibration sequence of K-theory spaces: $\rmK (v{\bold A}^w) \ra \rmK (v({\bold A})) \ra \rmK (w({\bold A}))$. 
\end{theorem}
\begin{theorem}
\label{add.thm}
 (Additivity theorem: see \cite[1.3.2, 1.4.2]{Wald} and \cite[Theorem 2.10]{GSVW}.) Let ${\bold A}$ and ${\bold B}$ be small complicial Waldhausen categories. Let $F, {F}', {F''}:{\bold A} \ra {\bold B}$ be three exact functors so that there are natural transformations $F' \ra F$ and $F \ra F''$ so that (i) for all $A$ in ${\bold A}$, $F'(A) \ra F(A)$ is a cofibration with its cofiber $\cong F''(A)$ and (ii) for any
cofibration $A' \ra A$ in ${\bold A}$, the induced map $F'(A) {\underset {F'(A')} \sqcup} F(A') \ra F(A)$ is a cofibration. Then the induced maps $KF$, $KF'$ and $KF''$ on K-theory spaces have the property that $KF \simeq KF' + KF''$. 
\end{theorem}
\begin{proof} As we make strong use of the above additivity theorem, we will explain how to deduce the above form
 of the additivity theorem from the form of the additivity theorem proven in \cite[Theorem 2.10]{GSVW}. Recall that
 \cite[Theorem 2.10]{GSVW} says the following: given a complicial Waldhausen category ${\bold A}$, let $\rmE({\bold A})$ denote the
 Waldhausen category whose objects are short exact sequences 
 \[\xymatrix{{A\quad }\ar@{>->} @<2pt> [r] &{B \quad }\ar@{->>} @<2pt> [r] &{C}}\]
 which are degree-wise split. 
 \vskip .1cm
 A cofibration from \xymatrix{{A'\quad }\ar@{>->} @<2pt> [r] &{B' \quad }\ar@{->>} @<2pt> [r] &{C'}}  to
 \xymatrix{{A\quad }\ar@{>->} @<2pt> [r] &{B \quad }\ar@{->>} @<2pt> [r]& {C}} will be a commutative diagram:
 \[\xymatrix{{A'\quad }\ar@{>->} @<2pt> [r] \ar@{>->} @<2pt> [d] &{B' \quad }\ar@{->>} @<2pt> [r]\ar@{>->} @<2pt> [d]& {C'}\ar@{>->} @<2pt> [d]\\
              {A\quad }\ar@{>->} @<2pt> [r] &{B \quad }\ar@{->>} @<2pt> [r] &{C}}
 \]
so that the vertical maps are all cofibrations in ${\bold A}$ and the induced map $A\sqcup_{A'} B' \ra B$ is also 
a cofibration in ${\bold A}$. Then, we have a functor
\be \begin{equation}
 \label{gsvw.additivity}
 \rmE({\bold A}) \ra {\bold A} \times {\bold A}, \xymatrix{{A\quad }\ar@{>->} @<2pt> [r] &{B \quad }\ar@{->>} @<2pt> [r] &{C}} \mapsto A \oplus C.
\end{equation} \ee
Then \cite[Theorem 21.0]{GSVW} shows that the map ${\it w}\rmG_{\bullet}(E({\bold A})) \ra {\it w}\rmG_{\bullet}({\bold A}) \times {\it w}\rmG_{\bullet}({\bold A})$
is a weak equivalence. Now, one may deduce Theorem ~\ref{add.thm} from the above form of the additivity theorem, by the same
strategy adopted in \cite[Proposition 1.3.2]{Wald}: giving three exact functors $F', F$ and $F''$ as in Theorem ~\ref{add.thm}
is equivalent to giving an exact functor $\tilde F: {\bold A} \ra E({\bold B})$. Therefore, Theorem ~\ref{add.thm} above follows
 from the additivity theorem \cite[Theorem 2.10]{GSVW} by naturality.
 
\end{proof}

\vskip .2cm

\section{\bf Appendix B: Simplicial objects, Cosimplicial objects and Chain complexes}
In this section a chain complex (a cochain complex) 
will denote a complex trivial in negative degrees  and where the
differentials are all of degree $-1$ ($+1$, \res). 
Let ${\bold A} = {\rm {Mod}}(\S, \O_{\S})$ denote the Abelian category of all modules over $\O_{\S}$ where $\S$
is a given algebraic stack (which, as always in this paper, is assumed to be Noetherian). Let ${\rm {Mod}}_{fl}(\S, \O_{\S})$ denote the full subcategory of
flat modules with finitely generated stalks. Recall that one has normalization functors 
$\rmN_{\it h}$: (Simplicial objects in ${\rm {Mod}}(\S, \O_{\S})) \ra $(Chain complexes in ${\rm {Mod}}(\S, \O_{\S})$) and
its inverse ${\rm DN}_{\it h}$:((Chain complexes in ${\rm {Mod}}(\S, \O_{\S})) \ra $(Simplicial objects in ${\rm {Mod}}(\S, \O_{\S})$). \footnote{In the literature, the 
inverse functor ${\rm DN}_{\it h}$ is often denoted ${\rm K}$. However, as we have reserved $K$ to denote complexes, our choice of ${\rm DN}_{\it h}$ seems preferable.}
Recall $\rmN_{\it h}$ is defined by sending the simplicial object $\rmS_{\bullet}$ to the chain complex $\rmK_{\bullet} = \rmN(S_{\bullet})$ defined by $\rmK_n = {\underset {0 \le i \le n-1} \bigcap}(ker (d_i:S_n \ra S_{n-1}))$. The differential
$\delta: \rmK_n \ra \rmK_{n-1}$ is defined by $\delta = (-1)^n d_n$. The functor ${\rm DN}_{\it h}$ is defined by ${\rm DN}_{\it h}(\rmK_{\bullet})_n
= {\underset {0 \le m \le n} \oplus}\, {\underset {s_{\alpha}:[n] \ra [m]} \oplus}\rmK_m$ where the $s_{\alpha}$ range
over all iterated degeneracies $[n] \ra [m]$ in the category $\Delta$. (See \cite{Cu} for more details.)
There are corresponding functors defined between the categories of cosimplicial objects in ${\rm {Mod}}(\S, \O_{\S})$ and
cochain complexes in ${\rm {Mod}}(\S, \O_{\S})$. These will be denoted $\rmN^{\it v}$ and ${\rm DN}^{\it v}$. Given a double complex $\rmK^{\bullet}_{\bullet}$ trivial
everywhere except the second quadrant (that is, we assume $\rmK^j_i = 0$ for $i>0$ or $j<0$), we let ${\rm Tot}(\rmK^{\bullet}_{\bullet})$ denote the complex defined by 
${\rm Tot}(\rmK^{\bullet}_{\bullet})^n = {\underset {i+j=n} \oplus} \rmK_i^j$. We will often use $\rmN$ (${\rm DN}$) to denote either one of $\rmN_{\it h}$ or $\rmN^{\it v}$ (${\rm DN}_{\it h}$ or ${\rm DN}^{\it v}$, \res) when there is no chance for confusion.
\begin{proposition}
\label{N.DN}
(i) The functors $\rmN_{\it h}$ and ${\rm DN}_{\it h}$ are strict inverses of each other, that is, $\rmN_{\it h} \circ {\rm DN}_{\it h} =id$ and ${\rm DN}_{\it h} \circ N_{\it h} =id$.
Similarly, $\rmN^{\it v} \circ {\rm DN}^{\it v} =id$ and ${\rm DN}^{\it v} \circ \rmN^{\it v} =id$.
\vskip .2cm \noindent
(ii) The functors $\rmN$ and ${\rm DN}$ associated to both simplicial and cosimplicial objects in ${\rm {Mod}}(\S, \O_{\S})$
preserve {\it degree-wise flatness} and the property of having finitely generated stalks. They also commute
with filtered colimits. 
\vskip .2cm \noindent
(iii) The functors $\rmN$ and ${\rm DN}$ associated to both simplicial and cosimplicial
objects in \\${\rm {Mod}}_{fl}(\S, \O_{\S})$ commute with the pull-back $f^*:{\rm {Mod}}_{fl}(\S, \O_{\S}) \ra {\rm {Mod}}_{fl}(\S', \O_{\S'})$ associated to a map $f: \S' \ra \S$of
algebraic stacks.
\end{proposition}
\begin{proof}
(i) is a standard result and is therefore skipped. (See \cite{Cu} for the simplicial case.) 
We prove
(ii) first in the simplicial case. Let $\rmS_{\bullet}$ denote a simplicial object in ${\rm {Mod}}(\S, \O_{\S})$
where each $\rmS_n$ is a flat $\O_{\S}$-module with finitely generated stalks. We will now prove, using ascending induction on $n$ that
each $\rmK_n =\rmN(S_{\bullet})_n$ is a flat $\O_{\S}$-module. Since $\rmK_0 = \rmN(\rmS_{\bullet})_0 = S_0$ this is clear for $n=0$.
The general case follows from  Lemma ~\ref{norm.flat} below. The definition of the functor ${\rm DN}$ as a sum
in each degree shows that it preserves flatness. The situation for the cosimplicial objects and cochain
complexes is entirely similar and is therefore skipped.
\vskip .2cm 
Since the functors ${\rm DN}$ for simplicial and cosimplicial objects are defined
as iterated sums, it is clear $f^*$ commutes with ${\rm DN}$. The functor $\rmN$
for cosimplicial objects is defined as an iterated co-kernel and therefore it
commutes with $f^*$. The corresponding assertion for simplicial objects follows
from the lemma below.
\end{proof}
\begin{lemma}
\label{norm.flat}
 (i) Let $\rmS_{\bullet}$ denote a simplicial object in ${\rm {Mod}}(\S, \O_{\S})$ that is flat (with finitely generated stalks) in each degree.
Then for each integer $n \ge 1$, and $0 \le m \le n-1$, ${\underset {0 \le i \le m} \bigcap} (ker d_i:\rmS_n \ra \rmS_{n-1})$ is a flat $\O_{\S}$-module (with finitely generated stalks). 
\vskip .2cm
(ii) Let $f:\S' \ra \S$denote a map of algebraic stacks, let $x':X' \ra \S'$ ($x:X \ra \S$) denote an atlas with $B_{x'}\S'$ ($B_x\S$, \res) denoting the
associated simplicial classifying space. Assume the atlases are chosen so that
there is an induced map of simplicial algebraic spaces $Bf:B_{x'}\S' \ra B_x\S$. Let $\rmS_{\bullet}$ denote a 
simplicial object in ${\rm {Mod}}(B_x\S, \O_{B_x\S})$ that is flat in each degree. 
Then for each integer $n \ge 1$, and $0 \le m \le n-1$, $
f^*({\underset {0 \le i \le m} \bigcap} (ker d_i:\rmS_n \ra \rmS_{n-1})) \cong
{\underset {0 \le i \le m} \bigcap} (ker d_i:f^*(\rmS_n) \ra f^*(\rmS_{n-1}))$.
\end{lemma}
\begin{proof} As observed by B. Koeck, this Lemma may be readily proven by observing that $\rmN(\rmS.)_n$ is a direct summand of $\rmS_n$
and that $\rmN(\rmS.)_n$ arises by taking the quotient of $\rmS_n$ by the image of the degeneracy maps. 
We may also prove (i) and (ii) simultaneously  using ascending induction on  $m$ making use of the above 
observation. We skip the remaining details.
\end{proof} 
\vskip .1cm
\subsection{The Eilenberg-Zilber and Alexander-Whitney pairings}
These are well-known between Chain complexes in any abelian category and the corresponding simplicial objects: see \cite[p. 129 and p. 133]{May}.
These readily extend to similar pairings for cochain complexes and cosimplicial objects in any abelian category: 
for example, one may interpret cosimplicial objects in an abelian category as simplicial objects in the dual 
abelian category and make use of the well-known pairings for simplicial objects and chain complexes. Therefore,
such pairings extend to similar pairings between cosimplicial-simplicial objects in an abelian category ${\bold A}$ and
the corresponding category of cochain complexes in ${\bold A}$, in the setting of section 5. In more detail, we obtain the
following.
\vskip .1cm
Let ${\bold A}$ denote an abelian category and let $Double({\bold A})$ denote the category of double 
cochain complexes in ${\bold A}$ concentrated in the second quadrant.
Given such a double cochain complex $\rmK$, and applying the composite functor ${\rm DN}^{\it v} \circ {\rm DN}_{\it h}$
produces a cosimplicial-simplicial object. The category of such objects will be denoted ${\rm Cos.mixt}({\bold A})$. 
The inverse functor, 
\begin{equation}
\label{NvNh}
\rmN= \rmN^{\it v} \circ \rmN_{\it h} = \rmN_{\it h} \circ \rmN^{\it v}
\end{equation}
sends such a cosimplicial simplicial object to a double cochain complex concentrated in the
second quadrant. Then we obtain associative pairings:
\be \begin{align}
     \label{pairings.EZ.AW}
  {\rm Tot}(\rmN (P)) \otimes {\rm Tot}( \rmN(Q)) &\ra {\rm Tot}(\rmN (P \otimes Q)) \\
   {\rm Tot}( \rmN(P \otimes Q)) &\ra {\rm Tot}(\rmN(P)) \otimes  {\rm  Tot} (\rmN(Q))\notag 
\end{align} \ee
\vskip .1cm \noindent
for any two objects $\rmP, \rmQ \in {\rm Cos.mixt}({\bold A})$ which are both functorial in $\rmP$ and $\rmQ$.  Here ${\rm Tot}$ denotes
the total complex: for a double cochain complex $\rmK = \{\rmK^i_j|i\ge 0, j \le 0\}$ concentrated in the second quadrant, 
\begin{equation}
\label{Tot}
{\rm Tot}(\rmK)^n = \oplus_{i+j=n}\rmK^i_j.
\end{equation}

%

\end{document}